\documentclass[10pt,a4paper]{article}

\usepackage{geometry}
\usepackage{amsmath}
\usepackage{amssymb}
\usepackage{amsthm}
\usepackage{bbm}
\usepackage{mathtools}
\usepackage{stmaryrd}
\usepackage[all]{xy}
\usepackage{hyperref}
\usepackage{graphicx}
\usepackage[utf8]{inputenc}
\usepackage{blindtext}
\usepackage[T1]{fontenc}
\usepackage[document]{ragged2e}
\usepackage{subcaption}
\usepackage{placeins}
\usepackage[main=english]{babel}
\usepackage{enumitem}
\usepackage{float}
\usepackage{tikz}
\usetikzlibrary{calc}
\usetikzlibrary{patterns}
\usepackage{comment}
\usepackage{multirow}
\usepackage{hhline}
\usepackage{pdfpages}
\usepackage{cleveref}

\numberwithin{equation}{section}

\usepackage[
backend=bibtex,
autocite=inline,
autolang=hyphen,
autopunct,
style=numeric,
sorting=nty,
giveninits=true,
doi=false,
url=true,
isbn=false,
maxcitenames=50,
maxbibnames = 50,
]{biblatex}
\AtEveryBibitem{\clearfield{note}} 
\AtEveryBibitem{\clearfield{volume}}
\AtEveryBibitem{\clearfield{number}}  
\renewbibmacro{in:}{%
  \ifentrytype{article}{}{\printtext{\bibstring{in}\intitlepunct}}}

\newtheorem{theorem}{Theorem}[section]
\newtheorem{corollary}[theorem]{Corollary}
\newtheorem{definition}[theorem]{Definition}
\newtheorem{lemma}[theorem]{Lemma}
\newtheorem{remark}[theorem]{Remark}

\newtheorem{proposition}[theorem]{Proposition}
\newtheorem{notation}[theorem]{Notation}

\newtheorem{convention}[theorem]{Convention}

\makeatletter
\def\moverlay{\mathpalette\mov@rlay}
\def\mov@rlay#1#2{\leavevmode\vtop{%
   \baselineskip\z@skip \lineskiplimit-\maxdimen
   \ialign{\hfil$\m@th#1##$\hfil\cr#2\crcr}}}
\newcommand{\charfusion}[3][\mathord]{
    #1{\ifx#1\mathop\vphantom{#2}\fi
        \mathpalette\mov@rlay{#2\cr#3}
      }
    \ifx#1\mathop\expandafter\displaylimits\fi}
\makeatother

\newcommand{\supp}{\operatorname{supp}}
\newcommand{\CR}{\operatorname{CR}}
\newcommand{\vertiii}[1]{{\left\vert\kern-0.25ex\left\vert\kern-0.25ex\left\vert #1 
    \right\vert\kern-0.25ex\right\vert\kern-0.25ex\right\vert}}
\newcommand{\range}[1]{{\mathbb{N}_{\leq #1}}}
\newcommand{\multirange}[2]{{\mathbb{N}_{\leq #2}^{#1}}}
\newcommand{\llbrace}{\{ \! \! \{}
\newcommand{\rrbrace}{\} \!\! \}}

\def\XXint#1#2#3{{\setbox0=\hbox{$#1{#2#3}{\int}$ }
\vcenter{\hbox{$#2#3$ }}\kern-.6\wd0}}

\title{Discrete reliability for high-order Crouzeix--Raviart finite elements}
\author{Bohne, Nis-Erik\footnote{nis-erik.bohne@math.uzh.ch, Universität Zürich, Institut für Mathematik, Winterthurerstrasse 190, CH-8057 Zürich} and Sauter, Stefan A.\footnote{stas@math.uzh.ch, Universität Zürich, Institut für Mathematik, Winterthurerstrasse 190, CH-8057 Zürich}}
\date{17. Feb. 2026}

\begin{document}
\maketitle
\justifying
\begin{abstract}
In this paper, the adaptive numerical solution of a 2D Poisson model problem
by Crouzeix-Raviart elements ($\operatorname*{CR}_{k}$ $\operatorname*{FEM}$)
of arbitrary odd degree $k\geq1$ is investigated. The analysis is based on an
established, abstract theoretical framework: the \textit{axioms of adaptivity}
imply optimal convergence rates for the adaptive algorithm induced by a
residual-type a posteriori error estimator. Here, we introduce the error
estimator for the $\operatorname*{CR}_{k}$ $\operatorname*{FEM}$
discretization and our main theoretical result is the proof ot Axiom 3:
\textit{discrete reliability}. This generalizes results for adaptive lowest
order $\operatorname*{CR}_{1}$ $\operatorname*{FEM}$ in the literature. For
this analysis, we introduce and analyze new local quasi-interpolation
operators for $\operatorname*{CR}_{k}$ $\operatorname*{FEM}$ which are key for
our proof of discrete reliability. We present the results of numerical
experiments for the adaptive version of $\operatorname*{CR}_{k}$
$\operatorname*{FEM}$ for some low and higher (odd) degrees $k\geq1$ which
illustrate the optimal convergence rates for all considered values of $k$.
\end{abstract}

\textbf{AMS-Classification:} 65N30, 65N50

\textbf{Key words:} Discrete reliability, non conforming Crouzeix-Raviart elements, quasi-interpolation operators, axioms of adaptivity

\section{Introduction}
\label{Sec:Introduction}

Let $\Omega \subseteq \mathbb{R}^2$ be a bounded polygonal Lipschitz domain. We consider the Galerkin discretization of the Poisson equation
\begin{align}
- \Delta u & = f
\qquad \text{in } \Omega 
\end{align}
satisfying homogeneous Dirichlet boundary conditions, by the non-conforming Crouzeix-Raviart finite element ($\CR_k$ FEM) of arbitrary odd degree $k \geq 1$.
The Crouzeix-Raviart element was first introduced in \cite{CrouzeixRaviart} as a finite dimensional function space $W_h$ over a triangulation $\mathcal{T}$ by imposing orthogonality conditions of the jumps across interelement boundaries. This provides substantial freedom in designing concrete Crouzeix-Raviart type elements. Some examples of such elements in 2D include, but are not limited to, the lowest order Crouzeix-Raviart element \cite{CrouzeixRaviart} for degree $k=1$ (see, e.g, \cite{Brenner_Crouzeix} for a survey), the Fortin-Soulie element \cite{Fortin_Soulie} for $k = 2$, the Crouzeix-Falk element \cite{Crouzeix_Falk} for $k=3$, the Gauss-Legendre elements \cite{Baran_Stoyan} for $k \geq 4$ and of course the standard conforming $hp$-finite elements. Furthermore in \cite{ChaLeeLee} the cases $k = 4,6$ are discussed for the Stokes equation in 2D.
Some examples of 3D Crouzeix-Raviart elements include the linear $\CR_1$ FEM \cite{CrouzeixRaviart} and \cite{Fortin_d3} provides a local basis construction for quadratic Crouzeix-Raviart elements. In \cite{CDS} a family of maximal Crouzeix-Raviart elements which allow a local basis is presented and most recently, the work \cite{BohneSauterCiarletdD} extends the 2D Crouzeix-Raviart elements to any space dimension $d \geq 2$ for any degree $k \geq 1$ and provides a family of local interpolations operators for odd degrees $k \geq 1$.


In this paper, the adaptive version of the Crouzeix-Raviart finite element method ($\CR_k$ AFEM) generated by loops of the form
\begin{align}
\texttt{solve} \rightarrow 
\texttt{estimate} \rightarrow
\texttt{mark} \rightarrow
\texttt{refine}
\label{Eq:Adaptive Algorithm}
\end{align} 
driven by an residual a posteriori error estimator (cf., \eqref{Eq:A posteriori}) is investigated. 
We employ Dörfler marking \cite{Doerfler_convergent} as the marking strategy and newest vertex bisection \cite[Sec. 2.1]{Mitchell-NVB} as a refinement strategy. This work strives towards a proof of optimal convergence of $\CR_k$ AFEM of arbitrary odd degree $k \geq 1$ driven by our residual error estimator.

The axioms of adaptivity \cite{CFPP-AxiomsOfAdaptivity}: \emph{stability} \eqref{Eq:Stability}, \emph{reduction} \eqref{Eq:Reduction}, \emph{discrete reliability} \eqref{Eq:Discrete reliability} and \emph{general quasi orthogonality} \eqref{Eq:Quasi orthogonoality - general} in combination with Dörfler marking and newest vertex bisection, provide a general theoretical framework that is sufficient for a proof of optimal convergence rates of an adaptive process as in \eqref{Eq:Adaptive Algorithm}. 
For standart $h$-version conforming AFEM, optimal convergence with respect to its respective residual error estimator is known, see e.g., \cite{CKNS-AFEMOptimalRatesConf,Stevenson_adapt1}. Section 5.1 of \cite{CFPP-AxiomsOfAdaptivity} provides a proof of optimal convergence of conforming AFEM within the framework of the axioms of adaptivity. 
For non-conforming methods there are two popular methodical approaches: first, fully discontinuous methods within discontinuous Galerkin (dG) discretizations (see, e.g., \cite{ABCM02} for a unified analysis) and hybrid high order methods (HHO) (see e.g, \cite{DPD-HHOBook}). Both of these methods need stabilization of the bilinear form. Second, methods such as $\CR_k$ FEM which do not need stabilization of the bilinear form. 
For the lowest order $\CR_1$ AFEM optimal convergence is known (cf., e.g, \cite{BMS-NCAFEMOptimal,Rabus-OptimalRatesNCFEM}) and \cite[Sec. 5.2]{CFPP-AxiomsOfAdaptivity} provides a proof of optimal convergence within the axioms of adaptivity framework. Optimal convergence of high-order $\CR_k$ AFEM remains an open question. In \cite{Bonito2010}, optimal convergence of an adaptive symmetric interior penalty dG method is shown. For HHO the picture is less complete. In \cite[Sec. 4]{DPD-HHOBook} a reliable and efficient error estimator which contains a stabilization term is presented for HHO; however the reduction property \eqref{Eq:Reduction} is still open for HHO.
Recent research provides an efficient and reliable stabilization free residual estimator for HHO \cite{BCGT-StabFreeHHO} but discrete reliability as in \eqref{Eq:Discrete reliability} remains open.


In this paper we follow the theoretical framework of the axioms of adaptivity. We prove that the $\CR_k$ FEM of arbitrary odd degree $k \geq 1$ is discretely reliable \eqref{Eq:Discrete reliability} (cf. Theorem \ref{Thm:Discrete reliabilty}) and satisfies the stability and reduction properties \eqref{Eq:Stability} and \eqref{Eq:Reduction} with respect to our residual error estimator. 
The main difficulty for non-conforming finite element methods, like $\CR_k$ FEM, in the context of the axioms of adaptivity is that the finite element spaces are generally not nested under mesh refinement. For example, if $\widehat{\mathcal{T}}$ is a refinement of the triangulation $ \mathcal{T}$, then the $\CR_k$ spaces are generally non-nested, i.e., $\CR_k ( \mathcal{T} ) \nsubseteq \CR_k ( \widehat{\mathcal{T}} )$. One way of dealing with this non-nestedness is the introduction of conforming companion/enrichment operators $J$, which were already used in \cite{Brenner-confComp}.
The paper \cite{CarstensenPuttkammer-howToProofDiscreteReliabiltiyNonConf} provides a general dimension-independent approach for the proof of discrete reliability for non-conforming finite element methods via the use of an operator $J$, which is a conforming companion operator to a non-conforming quasi interpolation operator $I_{\operatorname{NC}} : V + V (\mathcal{T}) \to V ( \mathcal{T} )$, where $V$ is the space of weak solutions and $V ( \mathcal{T} ) \nsubseteq V$ is a non-conforming finite element space.

The approach in \cite{CarstensenPuttkammer-howToProofDiscreteReliabiltiyNonConf} relies on two key assumptions. It requires the non-conforming quasi-interpolation $I_{\operatorname{NC}}$ to \textbf{(a)} be the orthogonal projection with respect to the energy bilinear form and \textbf{(b)} be constructed from local degrees of freedom. For the lowest order $\CR_1$ FEM, the quasi-interpolation operator $I_1^{\mathcal{T}, \CR} : V +  \CR_{1,0} (\mathcal{T} )  \to \CR_{1,0} ( \mathcal{T} )$ constructed from the degrees of freedom of $\CR_{1,0} ( \mathcal{T} )$ (cf., \eqref{Eq:Non-Conf Approx Operator} with $k=1$) satisfies both \textbf{(a)} and \textbf{(b)}. Property \textbf{(a)} follows from integration by parts and \textbf{(b)} holds by construction. For $\CR_k$ FEM of odd degree $k \geq 3$, the simultaneous validity of  \textbf{(a)} and \textbf{(b)} for any given quasi-interpolation operator is still unknown. 

To circumvent the orthogonality requirement \textbf{(a)}, we generalize the approach of \cite{CarstensenPuttkammer-howToProofDiscreteReliabiltiyNonConf} by introducing 
novel quasi-interpolation operators. The construction of these operators utilize the local degrees of freedom of $\CR_k$ FEM for odd degrees $k \geq 1$ in $d = 2$ as described \cite[Sec. 6.1]{BohneSauterCiarletdD}. For even degree $k \geq 2$, \cite[Rem. 26]{BohneSauterCiarletdD} implies that the degrees of freedom for $\CR_k$ FEM must be global quantities and are therefore unsuited for our analysis. The paper \cite{CarstensenPuttkammer-howToProofDiscreteReliabiltiyNonConf} covers the case $k = 1$ for any dimension $d \geq 2$. Our focus is on the case $d = 2$ and odd $k \geq 1$. This is a key step towards a full theory for optimal convergence of the $\CR_k$ AFEM for arbitrary odd degree $k \geq 1$ in 2D.

The paper is structured as follows. In Section \ref{Sec:Notation} we introduce the relevant notation and function spaces. In Section \ref{Sec:dRel} we introduce the Poisson model problem and the residual error estimator for the $\CR_k$ FEM with a brief recap on the axioms of adaptivity. We also state and prove the main theorem of the paper via a series of auxiliary results. Sections \ref{Sec:NonConfApprox} -- \ref{Sec:Unseen} are dedicated to the proofs of the auxiliary results with a focus on the explicit construction of the quasi-interpolation operators involved. In Section \ref{Sec:StabilityReductionEfficiency} we prove both, the stability and reduction properties \eqref{Eq:Stability} and \eqref{Eq:Reduction} of the $\CR_k$ FEM with respect to the residual error estimator. Finally, in Section \ref{Sec:NumericalExperiments} we provide numerical evidence of optimal convergence of high order $\CR_k$ AFEM for $k \in \{ 1,3,5,7\}$.

%
%



\section{Notation}
\label{Sec:Notation}

For $n \in \mathbb{N}_0$ we define the index sets $\range{n} := \{0, 1, \dots , n \}$ and $\{  \boldsymbol{\alpha} \in \mathbb{N}_0^2 
\; \vert \; 0 \leq
\vert \boldsymbol{\alpha} \vert := \sum_{i=1}^2 \alpha_i  \leq n \}$.

\subsection{Meshes}

Let $\Omega \subseteq \mathbb{R}^2$ be an open, bounded, polygonal Lipschitz domain with boundary $\partial \Omega$ and let $\mathcal{T}$ be a triangulation of $\Omega$ consisting of closed triangles which is conforming: for any two triangles $K, K^{\prime} \in \mathcal{T}$, either $K = K^{\prime}$ or the intersection $K \cap K^{\prime}$ is a common edge $E$, a common vertex $\mathbf{z}$, or the empty set. For a triangle $K \in \mathcal{T}$ we denote its diameter by $h_K$, its volume by $\vert K \vert$, and the diameter of the largest inscribed ball in $K$ by $\rho_K$. The shape-regularity constant of $\mathcal{T}$ is given by
$
\gamma_{\mathcal{T}}
:=
{\max_{K\in\mathcal{T}} {h_{K}} / {\rho_{K}}}.
$
Let $T :=  [\{ \mathbf{Z}_0, \mathbf{Z}_1, \mathbf{Z}_2 \}] \subseteq \mathbb{R}^2$ be the reference triangle with vertices $\mathbf{Z}_0 := \mathbf{0}$, $\mathbf{Z}_1 = (1,0)^T$ and $\mathbf{Z}_2 := (0,1 )^T$, where $[ \omega ]$ denotes the convex hull of $\omega \subseteq \mathbb{R}^2$.


\begin{convention}\label{Conv:Hidden constants}
For two quantities $A,B > 0$, we write $A \lesssim B$ if there exists a constant $C > 0$ such that $A \leq C B$ depending on the shape-regularity constant $\gamma_{\mathcal{T}}$ and the polynomial degree $k$ with $C $ possibly deteriorating, if $k \to \infty$ or $\gamma_{\mathcal{T}} \to \infty$. We write $A \approx B$ if $A \lesssim B \lesssim A$. The value of the constant $C$ may change with every occurrence.
\end{convention}

For a sub-mesh $\mathcal{S} \subseteq \mathcal{T}$ we introduce $\omega_{\mathcal{S}} := \bigcup_{K \in  \mathcal{S}} K$ for the area covered by the triangles in $\mathcal{S}$ and $h_{\mathcal{S}} := \max \{ h_K \; \vert \; K \in \mathcal{S}\}$ for the (local) mesh size of $\mathcal{S}$. We collect all edges and vertices of $\mathcal{S}$ in the sets $\mathcal{E} (  \mathcal{S} )$ and $\mathcal{V} ( \mathcal{S} )$ respectively. To distinguish between inner-/boundary edges/vertices we use subscripts:
\begin{align}
\begin{matrix}
\qquad \ \ \mathcal{E}_{\partial \Omega} ( \mathcal{S} ) 
:=
\{  E \in \mathcal{E} ( \mathcal{S} ) 
\; \vert \; 
E \subseteq \partial \Omega \}
&\hspace{2mm}& \text{and} &
\qquad \mathcal{V}_{\partial \Omega} ( \mathcal{S} ) 
:=
\{  \mathbf{z} \in \mathcal{V} ( \mathcal{S} ) 
\; \vert \; 
\mathbf{z} \in \partial \Omega \},&\hspace{2mm}
\\
\mathcal{E}_{\Omega} ( \mathcal{S} ) := \mathcal{E} ( \mathcal{S} ) \setminus  \mathcal{E}_{\partial \Omega} ( \mathcal{S} )
&\hspace{2mm}& \text{and} &
 \mathcal{V}_{\Omega} ( \mathcal{S} ) := \mathcal{V} ( \mathcal{S} ) \setminus  \mathcal{V}_{\partial \Omega} ( \mathcal{S} ).&\hspace{2mm}
\end{matrix}
\label{Eq:Boundary edges vertices}
\end{align}
For any given triangle $K \in \mathcal{S}$, let $\mathcal{E} ( K )$ and $\mathcal{V} ( K )$ denote the sets of edges and vertices of $K$. If we add $\Omega$ or $\partial \Omega$ as a subscript, then we follow convention \eqref{Eq:Boundary edges vertices}. Let $\Gamma \subseteq \Omega$ denote the interface between $\omega_{\mathcal{S}}$ and $\omega_{\mathcal{T} \setminus \mathcal{S}}$ which consists of all edges which lie in both $\mathcal{E} (\mathcal{S})$ and $\mathcal{E} (\mathcal{T} \setminus \mathcal{S})$, i.e., 
\begin{align}
\begin{matrix}
\mathcal{E} ( \Gamma ) 
:= 
\mathcal{E} ( \mathcal{S} ) \cap \mathcal{E} ( \mathcal{T} \setminus \mathcal{S} )
&
\text{and} &
\Gamma 
:=
\bigcup_{E \in \mathcal{E} ( \Gamma )} E .
\end{matrix}
\label{Eq:Gamma}
\end{align} 
\begin{figure}
\centering
\begin{tikzpicture}[scale=.7]
\coordinate (zero) at (0,0);
\coordinate (ref) at (1,0);
\coordinate (z2) at ($(zero)!2!80:(ref)$);
\coordinate (z3) at ($(zero)!2.5!150:(ref)$);
\coordinate (z4) at ($(zero)!2.2!20:(ref)$);
\coordinate (z5) at ($(z2)!0.93!50:(z4)$);
\coordinate (z6) at ($(z3)!0.93!50:(z2)$);
\coordinate (z7) at ($(z3)!0.65!290:(zero)$);
\coordinate (z8) at ($(zero)!1.5!330:(ref)$);

\coordinate (z41) at ($(z4)!0.93!170:(z2)$);
\coordinate (z42) at ($(z4)!.8!250:(z2)$);
\coordinate (z21) at ($(z2)!.8!105:(z4)$);

\coordinate (z31) at ($(z3)!.8!125:(z2)$);
\coordinate (z32) at ($(z3)!.7!180:(z2)$);

\coordinate (zero1) at ($(zero)!1.5!270:(ref)$);

\draw[line width = 1, violet]  (z2) --(z3) -- (zero) -- (z4) -- (z2);
\draw [line width = 1, olive] (z4) -- (z5) -- (z2) -- (z6) -- (z3) -- (z7) -- (zero) --(z8) --(z4);
\draw[line width = 1,red] (z41) -- (z42) -- (z5) -- (z21) -- (z6) -- (z31) -- (z32) -- (z7) -- (zero1) -- (z8) -- (z41);

\draw (z2) -- (zero);
\draw (z4) -- (z41);
\draw (z4) -- (z42);
\draw (z2) -- (z21);
\draw (z3) -- (z31);
\draw (z3) -- (z32);
\draw (zero) -- (zero1);

\node at (barycentric cs:zero=2,z2=1,z3=1) {$K$};
\node at (barycentric cs:zero=1,z2=1,z4=1) {$K^{\prime}$};

\node at (barycentric cs:z6=1,z2=.6,z3=1) {$\mathcolor{olive}{\mathcal{S}^{1/2}}$};
\node at (barycentric cs:z31=1,z3=.5,z6=1) {$\mathcolor{red}{\mathcal{S}^1}$};
\node at (barycentric cs:zero=.31,z2=1,z3=1) {$\mathcolor{violet}{\mathcal{S}}$};

\end{tikzpicture}
\caption{For the adjacent triangles $ K, K^{\prime}$ the purple line depicts the internal boundary $\Gamma$. For $\mathcal{S} := \{K,K^{\prime}\}$, the areas $\omega_{\mathcal{S}^{1/2}}$ and $\omega_{\mathcal{S}^1}$ are encircled by the olive and red lines respectively.}
\label{Fig:Layers}
\end{figure}
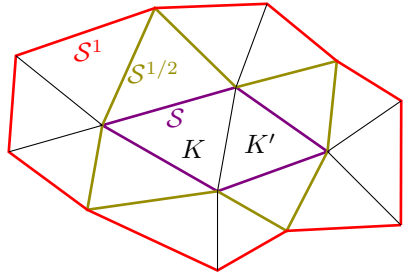
Furthermore, we set $\mathcal{S}^0 := \mathcal{S}$ and for $i \in \mathbb{N}_0$, we write
\begin{subequations}  \label{Eq:Layers}
\begin{align}
\mathcal{S}^{i+1}
&:= 
\{ K \in \mathcal{T} 
\; \vert \; 
\exists K^{\prime} \in \mathcal{S}^i: \
K  \cap K^{\prime} \neq \emptyset \},
\label{SubEq:Layers - plus one layer}
\\
\mathcal{S}^{i+1/2} 
&:=
\{ K \in \mathcal{T} 
\; \vert \;
\exists K^{\prime} \in \mathcal{S}^i: \ K \cap K^{\prime} \in \mathcal{E} (\mathcal{S}^i)\}.
\label{SubEq:Layers - Hedgehog}
\end{align}
\end{subequations}
The set $\mathcal{S}^{i+1/2}$ given by \eqref{SubEq:Layers - Hedgehog} is an intermediate step in a transition from $\mathcal{S}^i$ to $\mathcal{S}^{i+1}$, see Figure \ref{Fig:Layers} for an illustration. Given a triangle $K \in \mathcal{S}$, an edge $E \in \mathcal{E} ( \mathcal{S} )$, or a vertex $\mathbf{z} \in \mathcal{V} ( \mathcal{S} )$, we define the simplex-, edge-, and vertex patches $\mathcal{S}_K, \mathcal{S}_E, \mathcal{S}_{\mathbf{z}} \subseteq \mathcal{S}$ by
\begin{align}
\begin{matrix}
\mathcal{S}_K
:=
\{ K^{\prime} \in \mathcal{S} 
\; \vert \; 
K \cap K^{\prime} \neq \emptyset \},
&
\mathcal{S}_E
:=
\{ K^{\prime} \in \mathcal{S} 
\; \vert \;
E \in \mathcal{E} ( K^{\prime} ) \}
& \text{and} &
\mathcal{S}_{\mathbf{z}}
:=
\{ K^{\prime} \in \mathcal{S} 
\; \vert \;
\mathbf{z} \in \mathcal{V} ( K^{\prime} ) \}.
\end{matrix}
\label{Eq:Entity patches}
\end{align}
If a superscript is used, e.g. like in $(\mathcal{S}_K)^{1/2}$, we follow \eqref{Eq:Layers}. For any vertex $ \mathbf{z} \in \mathcal{V} ( \mathcal{S} )$, let the set
\begin{align}
\mathcal{E}_{\mathbf{z}} ( \mathcal{S} )
&:=
\{ E \in \mathcal{E} ( \mathcal{S} ) 
\; \vert \;
\mathbf{z} \in E \}
\label{Eq:S contained edge spider}
\end{align}
contain all edges in $ \mathcal{E} (\mathcal{S})$ which have $\mathbf{z}$ as an endpoint.

Since $\mathcal{T}$ is shape-regular, there exist constants $N_{\max}, N_{\operatorname{ov}} > 0 $ depending only on the shape-regularity constant $\gamma_{\mathcal{T}}$ such that
\begin{align}
\begin{matrix}
N_{\max} 
:= 
\max_{\mathbf{z} \in \mathcal{V} ( \mathcal{T} )}
\operatorname{card} \mathcal{T}_{\mathbf{z}} 
\lesssim 1
& \text{and} &
N_{\operatorname{ov}}:= 
\max_{K \in \mathcal{T}} \operatorname{card} 
\{ K^{\prime} \in \mathcal{T} 
\; \vert \; 
K \in \mathcal{T}_{K^{\prime}}^m
 \}
\lesssim 1
& m \in \{0,1\}
\end{matrix}
\label{Eq:Finite overlap and minimal angle constants}
\end{align}
with $\operatorname{card} ( \cdot )$ denoting the counting measure. Shape-regularity also implies the equivalence of local mesh sizes
\begin{align}
\begin{matrix}
\vert K \vert^{1/2}
\leq
h_K 
\leq
h_{\mathcal{T}_{\mathbf{z}}}
\leq
h_{\mathcal{T}_K}
\lesssim
\vert K \vert^{1/2}
& \text{and} &
h_E
\leq
h_K
\leq
h_{\mathcal{T}_E} 
\lesssim
h_E
\end{matrix},
\label{Eq:Equivalence local mesh size and volume}
\end{align}
for all $K \in \mathcal{T}$, $E \in \mathcal{E} ( K )$ and $\mathbf{z} \in \mathcal{V} ( K )$. The length of an edge $E \in \mathcal{E} ( \mathcal{T} )$ is denoted by $h_E := \vert E \vert$.

For any $E \in \mathcal{E} ( \mathcal{T} )$ we fix a numbering of the endpoints of $E$, i.e., $\mathcal{V} ( E ) = \{ \mathbf{z}_0, \mathbf{z}_1 \}$. This numbering induces both a tangential- and normal vector for $E$ in the following way: let $\mathbf{t}_E := (\mathbf{z}_1 - \mathbf{z}_0)/ h_E$ be the induced unit tangent vector of $E$ and let $\mathbf{n}_E := ( -t_{E,2}, t_{E,1})^T$ denote the induced unit normal vector of the edge $E$. As a convention, the numbering of the endpoints of any boundary edge $E \in \mathcal{E}_{\partial \Omega} (\mathcal{T})$ is chosen in such a way that $\mathbf{n}_E$ points into the exterior of the domain $\Omega$. The normal vector induces a labeling of the triangles in the edge patch $\mathcal{T}_E$ by
\begin{align}
\mathcal{T}_E := \{ K_{\operatorname{L}}, K_{\operatorname{R}} \},
\label{Eq:Enumeration edge patch}
\end{align} 
where $\mathbf{n}_E$ points into $K_{\operatorname{R}}$. 

\begin{definition} \label{Def:Admissible refinement}
A mesh $\widehat{\mathcal{T}}$ is an \emph{admissible refinement} of $\mathcal{T}$ if $\widehat{\mathcal{T}}$ is generated from $\mathcal{T}$ a finite series of \emph{newest vertex bisections} (NVB), see \cite[Sec. 2.1]{Mitchell-NVB}. Let $\mathbb{T} ( \mathcal{T} )$ denote the set of all admissible refinements of the mesh $\mathcal{T}$.
\end{definition}

\begin{remark} \label{Rem:Shape-regularity}
As a convention, if $\mathcal{T}$ and $\widehat{\mathcal{T}}$ appear in the same context, then $\widehat{\mathcal{T}}$ is an admissible refinement of $\mathcal{T}$. Since we know from \cite[Lem. 1]{nochetto2011primer} that $ \gamma_{\operatorname{spr}} := \sup_{\widehat{\mathcal{T}} \in \mathbb{T} ( \mathcal{T} )} \gamma_{\widehat{\mathcal{T}}} \lesssim 1$ only depends on $\mathcal{T}$, we assume that the hidden constant in Convention \ref{Conv:Hidden constants} depends on $\gamma_{\mathcal{T}}$ if $\mathcal{T}$ and $\widehat{\mathcal{T}}$ appear in the same context.
All mesh related notation is also applied to refined meshes.
\end{remark}

Given a coarse triangle $K \in \mathcal{T}$, the set 
$\operatorname{succ} ( K )
:=
\{  \widehat{K} \in \widehat{\mathcal{T}} 
\; \vert \;
\widehat{K} \subseteq K \}
$
contains all successor triangles of $K$ in $\widehat{\mathcal{T}} \in \mathbb{T} ( \mathcal{T} )$.
The same notation also applies for edges. Let the sets $\mathcal{R}  \subseteq \mathcal{T}$ and $\widehat{\mathcal{R}} \subseteq \widehat{\mathcal{T}}$ be given by
\begin{align}
\begin{matrix}
\mathcal{R} := \mathcal{T} \setminus \widehat{\mathcal{T}}
&
\text{and} &
\widehat{\mathcal{R}}
:=
\widehat{\mathcal{T}} \setminus \mathcal{T}
\end{matrix}.
\label{Eq:Refined elements}
\end{align}
The set $\widehat{\mathcal{R}}$ contains all fine triangles which have a predecessor $K \in \mathcal{R}$.
By construction, we have
\begin{align}
\begin{matrix}
\widehat{\mathcal{T}} \setminus \widehat{\mathcal{R}} 
=
\mathcal{T} \setminus \mathcal{R}
=
\mathcal{T} \cap \widehat{\mathcal{T}}
& \text{and}&
\omega_{\mathcal{R}} = \omega_{\widehat{\mathcal{R}}}
\end{matrix}.
\label{Eq:Refined domain}
\end{align}

\subsection{Function spaces}

We employ standard Lebesgue- and Sobolev notation for the continuous function spaces over the field of real numbers. Let $H_{0}^{1}( \Omega )  $ be the closure of the
space of infinitely smooth, compactly supported functions with respect
to the $H^{1}$ norm. 
The scalar product and norm in $L^{2} (  \Omega)  $ are standard: $(  u,v )  _{L^{2} (  \Omega )  } := \int_{\Omega} uv$ and $\Vert u \Vert _{L^{2} (  \Omega )} := (  u,u )  _{L^{2} (  \Omega )  }^{1/2}$.
Vector-valued analogues of these function spaces are denoted by bold letters, e.g., $\mathbf{L}^{2}(
\Omega )  =(  L^2(  \Omega)  )^{2}$ and analogously for other quantities. The $\mathbf{L}^{2}(  \Omega)  $ scalar product and norm for
vector-valued functions are given by $(  \mathbf{u},\mathbf{v})_{\mathbf{L}^{2}(  \Omega )} := \int_{\Omega} \langle \mathbf{u},\mathbf{v} \rangle$ and $\Vert \mathbf{u}\Vert _{\mathbf{L}^{2}(\Omega)} := (  \mathbf{u},\mathbf{u})_{\mathbf{L}^{2}(  \Omega )  }^{1/2}$,
with the standard Euclidean scalar product $\langle \cdot,\cdot \rangle$ in $\mathbb{R}^2$. For any function $u \in H^1 ( \Omega )$ we write $\nabla u \in \mathbf{L}^2 ( \Omega )$ for the gradient of $u$. We introduce the piecewise $H^1$-space
\begin{align}
H^1 ( \mathcal{T} )
&:=
\{ u  \in L^2 ( \Omega )
\; \vert \; 
\forall K \in \mathcal{T}: \
\left. u \right\vert_{\overset{\circ}{K}} \in H^1 ( \overset{\circ}{K} ) \}
\label{Eq:Brocken H1}
\end{align}
and denote the piecewise gradient of any $ u \in H^1 ( \mathcal{T} )$ by $\nabla_{\mathcal{T}} u \in \mathbf{L}^2 ( \Omega )$ defined by
\begin{align}
\left. \left( \nabla_{\mathcal{T}} u \right) \right\vert_{\overset{\circ}{K}}
:= \nabla \left(\left. u \right\vert_{\overset{\circ}{K}} \right).
\label{Eq:Broken gradient}
\end{align}

\begin{convention} \label{Not:Sumation convention in H1 broken}
A function $v\in H^{1}(  \mathcal{T})  $ may be discontinuous
across simplex facets. In this way the values on a facet are not uniquely
defined. For $K\in\mathcal{T}$ the function $\left.  v\right\vert
_{\overset{\circ}{K}}$ is well-defined and the function $\left.  v\right\vert
_{K}$ equals $\left.  v\right\vert _{\overset{\circ}{K}}$ in the interior of
$K$ and is defined on $\partial K$ as the continuous continuation of $\left.
v\right\vert _{\overset{\circ}{K}}$ to $\partial K$. In a similar fashion the sum of functions $v_{1},v_{2}\in H^{1}( \mathcal{T})$ on some $K \in \mathcal{T}$ is defined by
\begin{align*}
\left.  (  v_{1}+v_{2})  \right\vert _{K}
&:=
\left.  v_{1}\right\vert_{K}
+
\left.  v_{2}\right\vert _{K}.
\end{align*}
If the functions $v,v_{1},v_{2}$ belong to $C^{0}(  \overline{\Omega} ) $, these definitions coincide with the standard restriction to $K$.
\end{convention}

Following the numbering convention of edge patches \eqref{Eq:Enumeration edge patch}, we write
\begin{subequations} \label{Eq:Edge jump and mean}
\begin{align}
\llbracket u \rrbracket_E 
&:=
\begin{cases}
\left. \left. u \right\vert_{K_{\operatorname{L}}} \right\vert_E 
- 
\left. \left. u \right\vert_{K_{\operatorname{R}}} \right\vert_E
& E \in \mathcal{E}_{\Omega} ( \mathcal{T} ),
\\
\left. \left. u \right\vert_{K_{\operatorname{L}}} \right\vert_E
& 
E \in \mathcal{E}_{\partial \Omega} ( \mathcal{T} ),
\end{cases}
\qquad \forall u \in H^1 ( \mathcal{T} ),
\label{SubEq:Edge jump and mean - jump}
\\
\llbrace u \rrbrace_E
&:=
\begin{cases}
\frac{1}{2}
(\left. \left. u \right\vert_{K_{\operatorname{L}}} \right\vert_E 
+ 
\left. \left. u \right\vert_{K_{\operatorname{R}}} \right\vert_E)
& E \in \mathcal{E}_{\Omega} ( \mathcal{T} ),
\\
\left. \left. u \right\vert_{K_{\operatorname{L}}} \right\vert_E
& 
E \in \mathcal{E}_{\partial \Omega} ( \mathcal{T} ),
\end{cases}
\qquad \forall u \in H^1 ( \mathcal{T} ),
\label{SubEq:Edge jump and mean - mean}
\end{align}
\end{subequations}
for the jump (resp. mean) over the edge $E$. Recall that some elementary computations reveal that for any given edge $E \in \mathcal{E}_{\Omega} ( \mathcal{T} )$
\begin{align}
\llbracket uv \rrbracket_E 
&\overset{\text{\cite[(2.14)]{Shedensack-DiscreteHelmholtz}}}{=}
\llbracket u \rrbracket_E \llbrace v \rrbrace_E
+
\llbrace u \rrbrace_E \llbracket v \rrbracket_E
\qquad
\forall u,v \in H^1 ( \mathcal{T} ).
\label{Eq:Jump of products}
\end{align}

The set of polynomials of total degree $k \in \mathbb{N}_0$ over $\omega$ is denoted by $\mathbb{P}_k ( \omega )$
and $\perp_{L^2 (\omega)}$ denotes the standard $L^2 ( \omega )$-orthogonality. 
As a convention we set $\mathbb{P}_{-1} ( \omega) = \{ 0\}$.
We use \eqref{SubEq:Edge jump and mean - jump}, to define non-conforming approximation spaces of $H^1 ( \Omega )$ and $H_0^1 ( \Omega )$. 
For any polynomial degree $k \in \mathbb{N}$, let $H_k^{\CR} ( \mathcal{T} ), H_{k,0}^{\CR} ( \mathcal{T} ) \subseteq H^1 ( \mathcal{T} )$ be given by
\begin{subequations} \label{Eq:Broken H1 jumps}
\begin{align}
H_k^{\CR} ( \mathcal{T} )
&:=
\{ u \in H^1 ( \mathcal{T} ) 
\; \vert \;
\forall E \in \mathcal{E}_{\Omega} ( \mathcal{T} ) : \
\llbracket u \rrbracket_E \perp_{L^2 (E)} \mathbb{P}_{k-1} ( E ) \},
\label{SubEq:Broken H1 jumps - full}
\\
H_{k,0}^{\CR} ( \mathcal{T} )
&:=
\{ u \in H_k^{\CR} ( \mathcal{T} ) 
\; \vert \;
\forall E \in \mathcal{E}_{ \partial \Omega} ( \mathcal{T} ) : \
\left. u \right\vert_E 
 \perp_{L^2 (E)} \mathbb{P}_{k-1} ( E ) \}.
\end{align}
\end{subequations}
For the discrete spaces we write
\begin{subequations} \label{Eq:Discrete spaces}
\begin{align}
\mathbb{P}_k ( \mathcal{T} )
&:=
\{ p \in H^1 ( \mathcal{T} ) 
\; \vert \; 
\forall K \in \mathcal{T}: \
\left. p \right\vert_K \in \mathbb{P}_k ( K ) \},
\label{SubEq:Discrete spaces - pw poly}
\\
S_k ( \mathcal{T} ) 
&:=
\mathbb{P}_k ( \mathcal{T} )
\cap
H^1 ( \Omega )
\quad \text{and} \quad
S_{k,0} ( \mathcal{T} ) 
:=
S_k ( \mathcal{T} )
\cap
H_0^1 ( \Omega ),
\label{SubEq:Discrete spaces - 0bc conf}
\\
\CR_k ( \mathcal{T} )
&:=
H_k^{\CR} ( \mathcal{T} ) 
\cap  
\mathbb{P}_k ( \mathcal{T} ),
\quad \text{and} \quad
\CR_{k,0} ( \mathcal{T} )
:=
H_{k,0}^{\CR} ( \mathcal{T} ) 
\cap  
\mathbb{P}_k ( \mathcal{T} ).
\label{SubEq:Discrete spaces - CR spaces}
\end{align}
\end{subequations}

\begin{remark} \label{Rem:Jump condition}
Note that the jump condition of functions in $H_k^{\CR} ( \mathcal{T} )$ implies that, for any interior edge $E \in \mathcal{E}_{\Omega} ( \mathcal{T} )$, it holds that
$
\int_E q \left. u \right\vert_{K_{\operatorname{L}}}
=
\int_E q \left. u \right\vert_{K_{\operatorname{R}}}
$ 
for all functions $u \in H_k^{\CR} ( \mathcal{T} )$ and $q \in \mathbb{P}_{k-1} ( E )$.
\end{remark}

Given any $\mathcal{S} \subseteq \mathcal{T}$ and integrable function $u : \omega_{\mathcal{S}} \to \mathbb{R}$, we denote the integral mean of $u$ over $\omega_{\mathcal{S}}$ by
\begin{align}
\overline{u}_{\mathcal{S}}
&:=
\frac{1}{\vert \omega_{\mathcal{S}} \vert} \int_{\omega_{\mathcal{S}}} u.
\label{Eq:Integreal mean}
\end{align}
For any triangle $K \in \mathcal{T}$ we denote  the barycentric coordinate of $\mathbf{z} \in \mathcal{V} ( K )$ by $\lambda_{K, \mathbf{z} } \in \mathbb{P}_1 ( K )$, i.e., $\lambda_{K, \mathbf{z} } ( \mathbf{y} ) = \delta_{\mathbf{z},\mathbf{y}}$ for all $\mathbf{z}, \mathbf{y} \in \mathcal{V} ( K )$, where $\delta_{\mathbf{z},\mathbf{y}}$ denotes the Kronecker-delta. 
We will also need orthogonal polynomials on triangles. For any multi-index $\boldsymbol{\alpha} \in \mathbb{N}^2$, the function $P_{T,\boldsymbol{\alpha}}^{(1,1,1)} \in \mathbb{P}_{\vert \boldsymbol{\alpha} \vert} ( T )$, for the reference element $T$ reads 
\begin{align}
P_{T,\boldsymbol{\alpha}}^{(1,1,1)} ( x,y)
= 
P_{\alpha_1}^{( 1, 2 \alpha_2 + 3 )} ( 1 - 2x)
( 1 -x )^{\alpha_2} 
P_{\alpha_2}^{( 1,1 )} \left( 2 \frac{y}{1-x} - 1 \right)
\qquad
( x, y) \in T
\label{Eq:Bivariate ortho poly - ref}
\end{align}
where $P_n^{(\alpha, \beta )}$ are the Jacobi polynomials of degree $n \in \mathbb{P}$ for the parameters $\alpha, \beta > -1$, cf. \cite[Table 18.3.1]{NIST:DLMF}.
The polynomials $P_{T, \boldsymbol{\alpha}}^{( 1,1,1 )}$ where first introduced in \cite{Proriol-orthoPolyTriangle} and satisfy
\begin{align*}
\left( W_T  P_{T,\boldsymbol{\alpha}}^{(1,1,1)}, P_{T,\boldsymbol{\alpha}^{\prime }}^{(1,1,1)} \right)_{L^2 ( T ) }
&=
C_{\boldsymbol{\alpha}}
\delta_{\boldsymbol{\alpha}, \boldsymbol{\alpha}^{\prime}},
\end{align*}
where $W_T ={\prod_{\mathbf{z} \in \mathcal{V} (T) }} \lambda_{T, \mathbf{z}}$ and $C_{\boldsymbol{\alpha}} > 0$ is some constant depending on $\boldsymbol{\alpha}$, see \cite[(3)]{Proriol-orthoPolyTriangle}. The lifted versions $P_{K,\boldsymbol{\alpha}}^{(1,1,1)} \in \mathbb{P}_{\vert \boldsymbol{\alpha} \vert} ( K)$ for $K \in \mathcal{T}$ read
\begin{align}
P_{K,\boldsymbol{\alpha}}^{(1,1,1)}
&:=
P_{T, \boldsymbol{\alpha}}^{(1,1,1)} \circ \chi_K^{-1}
\label{Eq:Bivariate ortho poly - physical}
\end{align}
where $\chi_K : T \to K$ is some affine transformation.
Setting $W_K = W_T \circ \chi_K^{-1}$, we compute via an affine pullback to the reference element $T$, that 
\begin{align}
\left( W_K  P_{K,\boldsymbol{\alpha}}^{(1,1,1)}, P_{K,\boldsymbol{\alpha}^{\prime }}^{(1,1,1)} \right)_{L^2 ( K ) }
&=
C_{\boldsymbol{\alpha}}
\frac{\vert K \vert}{\vert T \vert}
\delta_{\boldsymbol{\alpha}, \boldsymbol{\alpha}^{\prime}}.
\label{Eq:Pairwise orthogonality}
\end{align}

\section{Discrete reliability}
\label{Sec:dRel}

In this section, we introduce the Poisson model problem and define the corresponding non-conforming residual error estimator. We state and prove the main theorem of this paper: the Crouzeix--Raviart finite element method is discretely reliable for any odd polynomial degree $k \geq 1$.

\subsection{The model problem and the estimator}
Let $a : H_0^1 ( \Omega ) \times H_0^1 ( \Omega ) \to \mathbb{R}$ be the bilinear form:
$a ( u, v ) := ( \nabla u, \nabla v )_{\mathbf{L}^2 ( \Omega )}$ for all $u,v \in H_0^1 ( \Omega )$. We are considering the following problem: for given $f \in L^2 ( \Omega )$, find $u \in H_0^1 ( \Omega )$ such that
\begin{align}
a ( u , v ) 
&=  (f, v )_{L^2 ( \Omega )}
\qquad \forall v \in H_0^1 ( \Omega ).
\label{Eq:VariationalPoisson}
\end{align}
As a discretization method, we use the non-conforming Crouzeix--Raviart finite element method ($\CR_k$ FEM) based on the choice $\CR_{k,0} (\mathcal{T} )$ as a finite dimensional subspace in a Galerkin method: find $u_h \in \CR_{k,0} ( \mathcal{T} )$ which satisfies
\begin{align}
a_{\mathcal{T}} ( u_h, v_h ) =  ( f,  v_h )_{L^2 ( \Omega )}
\qquad \forall v_h \in \CR_{k,0} ( \mathcal{T} ),
\label{Eq:DiscretePoisson}
\end{align}
where the discrete bilinear form $a_{\mathcal{T}} : H_{k,0}^{\CR} ( \mathcal{T} ) \times H_{k,0}^{\CR} ( \mathcal{T} ) \to \mathbb{R}$ is given by
$
a_{\mathcal{T}} (u_h, v_h )
:= 
( \nabla_{\mathcal{T}} u_h, \nabla_{\mathcal{T}} v_h )_{\mathbf{L}^2 ( \Omega )}.
$ 

For a given right-hand side function $f \in L^2 ( \Omega )$ and triangulation $\mathcal{T}$, the residual error estimator $\eta ( \cdot ; \mathcal{T},f) : \mathbb{P}_k ( \mathcal{T} ) \to \mathbb{R}$ is defined by:
\begin{subequations} \label{Eq:A posteriori}
\begin{align}
\mu (v_h ;  K,f )^2
&:=
\vert K \vert \Vert f + \Delta v_h \Vert_{L^2 ( K )}^2
+
\vert K \vert^{1/2}
\sum_{E \in \mathcal{E}_{\Omega} ( K )}
\Vert \langle \llbracket \nabla_{\mathcal{T}} u_h \rrbracket_E , \mathbf{n}_E \rangle \Vert_{L^2 ( E )}^2
\quad \forall v_h \in \mathbb{P}_k ( \mathcal{T}),
\label{SubEq:A posteriori - conformity estimator}
\\
\nu ( v_h ; K )^2
&:=
\vert K \vert^{1/2}
\sum_{E \in \mathcal{E} ( K )}
\Vert \langle \llbracket \nabla_{\mathcal{T}} u_h \rrbracket_E , \mathbf{t}_E \rangle \Vert_{L^2 ( E )}^2
\qquad \forall v_h \in \mathbb{P}_k ( \mathcal{T} ),
\label{SubEq:A posteriori - nonconformity error}
\\
\eta ( v_h ; K,f )^2
&:=
\mu (v_h ;  K,f )^2
+
\nu ( v_h ; K )^2
\qquad \forall v_h \in \mathbb{P}_k ( \mathcal{T}),
\label{SubEq:A posteriori - local}
\\
\eta ( v_h ; \mathcal{S}, f )^2
&:= 
\sum_{K \in \mathcal{S}} 
\eta ( v_h ; K,f )^2 
\qquad \forall v_h \in \mathbb{P}_k ( \mathcal{T} ), \
\mathcal{S} \subseteq \mathcal{T}.
\label{SubEq:A posteriori - global}
\end{align}
\end{subequations}


\subsection{Recap: Theory of adaptive methods by the axioms of adaptivity}

The paper \cite{CFPP-AxiomsOfAdaptivity} provides a framework for a proof of optimal rates of an adaptive finite element method with a given a posteriori error estimator. The core assumptions for the theoretical framework consists of four fundamental properties called the axioms.  
Given an initial triangulation $\mathcal{T}_0$ and an odd polynomial degree $k \geq 1$, the axioms for the $\operatorname{CR}_k$ AFEM with the residual estimator $\eta$ from \eqref{Eq:A posteriori} read as follows: Assume that $\mathcal{T}_0$ is some initial triangulation.

\textbf{Stability:} There exists a constant $\Lambda_1 > 0$ such that for all $\mathcal{T}, \widehat{\mathcal{T}} \in \mathbb{T} ( \mathcal{T}_0 )$, where $\widehat{\mathcal{T}}$ is an admissible refinement of $\mathcal{T}$, any  $v \in \CR_{k,0} ( \mathcal{T} )$ and $\widehat{v} \in \CR_{k,0} ( \widehat{\mathcal{T}})$ satisfy
\begin{align}
\vert \eta ( v ; \mathcal{T} \cap \widehat{\mathcal{T}}, f) 
- 
\eta ( \widehat{v}; \widehat{\mathcal{T}} \cap \mathcal{T}, f) \vert
& \leq \Lambda_1
\Vert \nabla_{\mathcal{T}} v 
-  
\nabla_{\widehat{\mathcal{T}}} \widehat{v} \Vert_{\mathbf{L}^{2}( \Omega )}.
\tag{A1}
\label{Eq:Stability}
\end{align} 

\textbf{Reduction:} There exist constants $\Lambda_2 > 0$ and $0 < \rho_2 < 1$ such that for all $\mathcal{T}, \widehat{\mathcal{T}} \in \mathbb{T} ( \mathcal{T}_0 )$, where $\widehat{\mathcal{T}}$ is an admissible refinement of $\mathcal{T}$, any $v \in \CR_{k,0} ( \mathcal{T} )$ and $\widehat{v} \in \CR_{k,0} ( \widehat{\mathcal{T}})$ satisfy
\begin{align}
\eta ( \widehat{v}; \widehat{\mathcal{T}} \setminus \mathcal{T}, f)^2
& \leq
\rho_2 \eta ( v; \mathcal{T} \setminus \widehat{\mathcal{T}}, f)^2 
+
\Lambda_2
\Vert \nabla_{\mathcal{T}} v 
-  
\nabla_{\widehat{\mathcal{T}}} \widehat{v} \Vert_{\mathbf{L}^{2}( \Omega )}^2.
\tag{A2}
\label{Eq:Reduction}
\end{align}

\textbf{Discrete reliability: } There exists a constant $\Lambda_3 > 0$ such that for all $\mathcal{T}, \widehat{\mathcal{T}} \in \mathbb{T} ( \mathcal{T}_0 )$, where $\widehat{\mathcal{T}}$ is an admissible refinement of $\mathcal{T}$, the  coarse/fine discrete solution $u_h \in \CR_{k,0} (\mathcal{T} )$ and $\widehat{u}_h \in \CR_{k,0} ( \widehat{\mathcal{T} })$ of \eqref{Eq:DiscretePoisson} satisfy
\begin{align}
\Vert \nabla_{\widehat{\mathcal{T}}} \widehat{u}_h - \nabla_{\mathcal{T}} u_h \Vert_{\mathbf{L}^2 ( \Omega )}
& \leq \Lambda_3
\eta ( u_h ; \mathcal{R}^1, f).
\tag{A3}
\label{Eq:Discrete reliability}
\end{align}

\textbf{Quasi orthogonality:} There exists a constant $\Lambda_4 > 0$ such that the sequence of meshes $\mathcal{T}_j$ and discrete solutions $u_j \in \CR_{k,0} ( \mathcal{T}_j)$ of \eqref{Eq:DiscretePoisson} for $j \in \mathbb{N}_0$ generated by \eqref{Eq:Adaptive Algorithm} satisfy
\begin{align}
\sum_{j = \ell }^{\infty}
\Vert \nabla_{\mathcal{T}_{j+1}} u_{j+1} - \nabla_{\mathcal{T}_j} u_j \Vert_{\mathbf{L}^{2}( \Omega )}
& \leq \Lambda_4
\eta ( u_{\ell}; \mathcal{T}_{\ell} , f ),
\tag{A4}
\label{Eq:Quasi orthogonoality - general}
\end{align}

The requirements for the refinement strategy are given in \cite[Sec. 2.4, (2.7) -- (2.10)]{CFPP-AxiomsOfAdaptivity}.

\begin{proposition}[{\cite[Thm. 4.1]{CFPP-AxiomsOfAdaptivity}}] \label{Prop:OptimalRates}
Given a refinement strategy that satisfies \cite[(2.7) -- (2.10)]{CFPP-AxiomsOfAdaptivity}, \eqref{Eq:Stability} -- \eqref{Eq:Quasi orthogonoality - general} imply optimal convergence of the error estimator if Dörfler marking \cite{Doerfler_convergent} is used.
\end{proposition}

\subsection{The main theorem}

The main focus of this paper is on  \eqref{Eq:Discrete reliability}, i.e., discrete reliability. This is summarized in the main theorem of this paper.

\begin{theorem}[Discrete reliability] \label{Thm:Discrete reliabilty}
Let $k \geq 1$ be odd and let $\mathcal{T}$ be some coarse triangulation. 
Consider any $\widehat{\mathcal{T}} \in \mathbb{T} ( \mathcal{T} )$.
For a given $f \in L^2 ( \Omega )$, we denote the fine/coarse discrete solution of \eqref{Eq:DiscretePoisson} by $\widehat{u}_h \in \CR_{k,0} ( \widehat{\mathcal{T}} )$ and $u_h \in \CR_{k,0} ( \mathcal{T} )$ respectively. Then
\begin{align*}
\Vert \nabla_{\widehat{\mathcal{T}}} \widehat{u}_h - \nabla_{\mathcal{T}} u_h \Vert_{\mathbf{L}^2 ( \Omega )}
& \lesssim
\eta ( u_h ; \mathcal{R}^1, f),
\end{align*}
where $\mathcal{R}$ and $\mathcal{R}^1$ are as in \eqref{Eq:Refined elements} and \eqref{Eq:Layers}.
\end{theorem}

The proof of Theorem \ref{Thm:Discrete reliabilty} relies on a series of auxiliary results. We start with a non-conforming quasi-interpolation operator. An explicit construction of this operator is given in Section \ref{Sec:NonConfApprox}.

\begin{lemma} \label{Lem:Non-Conf Approx Operator}
For any $k \geq 1$ odd and triangulation $\mathcal{T}$ there exists a bounded quasi-interpolation operator $I_k^{\mathcal{T}, \CR} : H_k^{\CR} ( \mathcal{T} ) \to \CR_k ( \mathcal{T} )$ satisfying the following properties: for all $u \in H_k^{\CR} ( \mathcal{T})$, we have that,
\begin{subequations} \label{Eq:Non-Conf Approx Operator}
\begin{align}
\int_E q  I_k^{\mathcal{T}, \CR} u 
&=
\int_E q u 
\qquad \qquad \forall q \in \mathbb{P}_{k-1} ( E ), \ E \in \mathcal{E} (\mathcal{T}),
\label{SubEq:Non-Conf Approx Operator - edge moment}
\\
\int_K p  I_k^{\mathcal{T}, \CR} u
&=
\int_K p u
\qquad \qquad \forall p \in \mathbb{P}_{k-3} ( K ), \ K \in \mathcal{T},
\label{SubEq:Non-Conf Approx Operator- vol moment}
\\
\Vert \nabla I_k^{\mathcal{T}, \CR} u \Vert_{\mathbf{L}^2 ( K )} 
& \lesssim 
\Vert \nabla u \Vert_{\mathbf{L}^2 ( K )}
\qquad
\forall K \in \mathcal{T}.
\label{SubEq:Non-Conf Approx Operator - loc bound}
\end{align}
\end{subequations}
\end{lemma} 

The existence of $I_k^{\mathcal{T},\CR}$ is proven in \cite[Lem. 2.1]{ArnoldBrezzi-mixedNonConvFEM85} via a dimension counting argument and by showing that \eqref{SubEq:Non-Conf Approx Operator - edge moment} -- \eqref{SubEq:Non-Conf Approx Operator- vol moment} are locally unisolvent for $\mathbb{P}_k ( K )$. Here, we provide an explicit construction of $I_k^{\mathcal{T}, \CR}$ using local edge- and triangle oriented basis functions and their associated functionals in Section \ref{Sec:NonConfApprox}. The construction of $I_k^{\mathcal{T},\CR}$ generates valuable insight for all subsequent operator construction in this paper.


Apart from the operator $I_k^{\mathcal{T}, \CR}$ we need a second approximation operator. This operator will also be explicitly constructed in Section \ref{Sec:PartConfOp}.

\begin{lemma} \label{Lem:Partially conforming}
Let $k \geq 1$ be odd and $\mathcal{T}$ be some triangulation of $\Omega$. For every sub-mesh $\mathcal{S} \subseteq \mathcal{T}$, there exists bounded linear operator $M_{k,0}^{\mathcal{S}} : H_{k,0}^{\CR} ( \mathcal{T} ) \to \CR_{k,0} ( \mathcal{T} )$ such that
\begin{subequations} \label{Eq:Mixed Operator properties}
\begin{align}
\left. (M_{k,0}^{\mathcal{S}} u) \right\vert_{\omega_{\mathcal{S}}} 
&\in S_k ( \mathcal{S})
\qquad \qquad \qquad \ \ \,
\forall u \in  H_{k,0}^{\CR} ( \mathcal{T} ),
\label{SubEq:Mixed Operator properties - conforming on S}
\\
\left. (  M_{k,0}^{\mathcal{S}} u ) \right\vert_{\omega_{\mathcal{T} \setminus \mathcal{S}^{1/2}} }
&= 
\left. u \right\vert_{\omega_{\mathcal{T} \setminus \mathcal{S}^{1/2}} }
\qquad \qquad \quad \,
\forall u \in \CR_{k,0} ( \mathcal{T} ),
\label{ṢubEq:Mixed Operator properties - identity on TmSplus}
\\
\Vert \nabla M_{k,0}^{\mathcal{S}} u \Vert_{\mathbf{L}^2 ( K )}
& \lesssim 
\Vert \nabla_{\mathcal{T}} u \Vert_{\mathbf{L}^2 ( D_K )}
\qquad \quad \ \forall u \in  H_{k,0}^{\CR} ( \mathcal{T} ),
\label{SubEq:Mixed Operator properties - loc bounded}
\\
\Vert ( \operatorname{Id} - M_{k,0}^{\mathcal{S}}) u \Vert_{L^{2} (  K )}
& \lesssim
h_K \Vert \nabla_{\mathcal{T}} u \Vert_{\mathbf{L}^2 ( D_K )}
\qquad \forall K \in \mathcal{T},
\label{SubEq:Mixed Operator properties - approx}
\end{align}
\end{subequations} 
where $D_K \subseteq \overline{\Omega}$ is some local triangle neighborhood of $K$ and $\operatorname{Id}$ signifies the identity operator.
Due to \eqref{SubEq:Mixed Operator properties - conforming on S}, we refer to $M_{k,0}^{\mathcal{S}}$ as the \emph{partially conforming} operator.
\end{lemma}

Lemma \ref{Lem:Non-Conf Approx Operator} allows us to prove the existence of a discrete right-inverse operator for the operator $I_k^{\mathcal{T}, \CR}$. This is done via a \emph{conforming companion} operator $J_{k,0} : H_{k,0}^{\CR} ( \mathcal{T} ) \to H_0^1 ( \Omega )$. The details of the construction are given in Section \ref{Sec:RightInverse}.

\begin{lemma} \label{Lem:Upper bound by tangential jumps}
Let $\widehat{\mathcal{T}} \in \mathbb{T} ( \mathcal{T} )$ be given. Then for every coarse $\CR_k$ function $v \in \CR_{k,0} ( \mathcal{T} )$ there exists a fine $\CR_k$ function $\widehat{v}^{\ast} \in \CR_{k,0} ( \widehat{\mathcal{T}} )$ such that
\begin{subequations} \label{Eq:Non-Conf}
\begin{align}
I_k^{\mathcal{T}, \CR} \widehat{v}^{\ast} 
&= 
v,
\label{SubEq:Non-Conf - right inverse}
\\
\Vert \nabla_{\mathcal{T}} v - \nabla_{\widehat{\mathcal{T}}}\widehat{v}^{\ast} \Vert_{\mathbf{L}^2 ( \Omega )} &\lesssim
\nu (v; \mathcal{R}^1),
\label{SubEq:Non-Conf -  error estimate}
\end{align}
\end{subequations}
where $\mathcal{R}$ and $\mathcal{R}^1$ are as in \eqref{Eq:Refined elements} and \eqref{Eq:Layers}. 
\end{lemma}

Lemma \ref{Lem:Upper bound by tangential jumps} allows us to quantify the distance between the coarse and fine $\CR_k$ spaces, and we use the operator $M_{k,0}^{\mathcal{S}}$ to construct another quasi-interpolation operator $\widehat{P}$ which maps the fine $\CR_k$ space into the intersection of the coarse and fine $\CR_k$ space. 

\begin{lemma}\label{Lem:Existence Phat}
For any $\widehat{\mathcal{T}} \in \mathbb{T} (\mathcal{T})$ there exists an operator $\widehat{P} : \CR_{k,0} ( \widehat{\mathcal{T}} ) \to \CR_{k,0} ( \mathcal{T} ) \cap \CR_{k,0} ( \widehat{\mathcal{T}} )$ such that
for any function $\widehat{v} \in \CR_{k,0} ( \widehat{\mathcal{T}} )$
\begin{subequations} \label{Eq:Conditions Phat}
\begin{align}
\left. ( \widehat{v}- \widehat{P}  \widehat{v}) \right\vert_K 
&= 
0
\qquad \qquad \qquad \qquad \qquad \quad \ \,
\forall K \in \mathcal{T} \setminus \mathcal{R}^{1/2},
\label{SubEq:Conditions Phat - Identiy}
\\
\Vert ( \operatorname{Id} - \widehat{P} ) \widehat{v} \Vert_{L^2 ( K )}
& \lesssim
h_K
\Vert \nabla \widehat{v} \Vert_{\mathbf{L}^2 ( Z_K )}
\qquad \qquad \quad  \ \;
\forall K \in \mathcal{R}^{1/2},
\label{SubEq:Conditions Phat - approx propterties vol}
\\
\Vert \llbrace (\operatorname{Id} - \widehat{P} ) \widehat{v} \rrbrace_E \Vert_{L^2 ( E )}
& \lesssim 
\sum_{K \in \mathcal{T}_E} 
h_K ^{1/2}
\Vert \nabla_{\widehat{\mathcal{T}}} \widehat{v} \Vert_{\mathbf{L}^2 ( Z_K )}
\qquad \forall E \in \mathcal{E} ( \mathcal{R} ),
\label{SubEq:Conditions Phat - approx propterties edge}
\end{align}
\end{subequations}
where $Z_K \subseteq \overline{\Omega}$ is some local triangle neighborhood of $K$.
\end{lemma}

The following result is a consequence of the Lemmata \ref{Lem:Upper bound by tangential jumps} and \ref{Lem:Existence Phat} and is used in the proof of the main theorem.

\begin{lemma} \label{Lem:Conformity estimation}
Let $\widehat{\mathcal{T}} \in \mathbb{T} ( \mathcal{T} )$ and assume that $u_h \in \CR_{k,0} ( \mathcal{T} )$ and $\widehat{u}_h \in \CR_{k,0} (  \widehat{\mathcal{T}} )$ to be the coarse and fine solutions to the discrete Poisson problem \eqref{Eq:DiscretePoisson}. It holds that
\begin{align}
a_{\widehat{\mathcal{T}}} ( \widehat{u}_h - u_h, \widehat{u}_h - \widehat{u}_h^{\ast} )
& \lesssim
\mu ( u_h; \mathcal{R}^{1/2}, f)
\Vert \nabla_{\widehat{\mathcal{T}}} ( \widehat{u}_h - \widehat{u}_h^{\ast} )
\Vert_{\mathbf{L}^2 ( \Omega )},
\end{align}
where $\widehat{u}_h^{\ast} \in \CR_{k,0} ( \widehat{\mathcal{T}} )$ satisfies $I_k^{\mathcal{T}, \CR} \widehat{u}_h^{\ast} = u_h$ as in Lemma \ref{Lem:Upper bound by tangential jumps}.
\end{lemma}

With these auxiliary results at hand, we are able to prove the main theorem.

\begin{proof}[Proof of Theorem \ref{Thm:Discrete reliabilty}]
Let the mesh $\widehat{\mathcal{T}} \in \mathbb{T} ( \mathcal{T} )$ be an admissible refinement of $\mathcal{T}$ and let $u_h \in \CR_{k,0} ( \mathcal{T} )$ and $\widehat{u}_h \in \CR_{k,0} ( \widehat{\mathcal{T}} )$ be the coarse and fine solution to the discrete Poisson problem \eqref{Eq:DiscretePoisson}. We denote the discrete error by $\widehat{e}_h := \widehat{u}_h - u_h$. Let $\widehat{u}_h^{\ast} \in \CR_{k,0} ( \widehat{\mathcal{T}} )$ be given by Lemma \ref{Lem:Upper bound by tangential jumps} and estimate that
\begin{align}
\Vert \nabla_{\widehat{\mathcal{T}}} \widehat{u}_h - \nabla_{\mathcal{T}} u_h \Vert_{\mathbf{L}^2 ( \Omega )}^2
&=
a_{\widehat{\mathcal{T}}} ( \widehat{e}_h, \widehat{u}_h - u_h )
=
a_{\widehat{\mathcal{T}}}  ( \widehat{e}_h , \widehat{u}_h - \widehat{u}_h^{\ast} )
+
a_{\widehat{\mathcal{T}}}  ( \widehat{e}_h , \widehat{u}_h^{\ast} - u_h )
\nonumber
\\
& \leq
a_{\widehat{\mathcal{T}}}  ( \widehat{e}_h , \widehat{u}_h - \widehat{u}_h^{\ast} )
+ 
\Vert \nabla_{\widehat{\mathcal{T}}} \widehat{e}_h \Vert_{\mathbf{L}^2 ( \Omega )}
\Vert \nabla_{\widehat{\mathcal{T}}} \widehat{u}_h^{\ast} - \nabla_{\mathcal{T}} u_h \Vert_{\mathbf{L}^2 ( \Omega )}.
\label{Eq:Discrete error - first estimate}
\end{align}
Applying Lemma \ref{Lem:Conformity estimation} to $a_{\widehat{\mathcal{T}}}  ( \widehat{e}_h , \widehat{u}_h - \widehat{u}_h^{\ast} )$ we obtain that
\begin{align*}
\Vert \nabla_{\widehat{\mathcal{T}}}  \widehat{e}_h \Vert_{\mathbf{L}^2 ( \Omega )}^2
\lesssim &
\mu ( u_h ; \mathcal{R}^{1/2}, f )
\Vert \nabla_{\widehat{\mathcal{T}}} ( \widehat{u}_h - \widehat{u}_h^{\ast} )
\Vert_{\mathbf{L}^2 ( \Omega )}
+
\Vert \nabla_{\widehat{\mathcal{T}}} \widehat{e}_h \Vert_{\mathbf{L}^2 ( \Omega )}
\Vert \nabla_{\widehat{\mathcal{T}}} \widehat{u}_h^{\ast} - \nabla_{\mathcal{T}} u_h \Vert_{\mathbf{L}^2 ( \Omega )}
\nonumber
\\
\lesssim &
\mu ( u_h ; \mathcal{R}^{1/2}, f )
\left(\Vert \nabla_{\widehat{\mathcal{T}}} \widehat{e}_h \Vert_{\mathbf{L}^2 ( \Omega )}
+
\Vert \nabla_{\widehat{\mathcal{T}}} \widehat{u}_h^{\ast} - \nabla_{\mathcal{T}} u_h \Vert_{\mathbf{L}^2 ( \Omega )} \right)
+
\Vert \nabla_{\widehat{\mathcal{T}}} \widehat{e}_h \Vert_{\mathbf{L}^2 ( \Omega )}
\Vert \nabla_{\widehat{\mathcal{T}}} \widehat{u}_h^{\ast} - \nabla_{\mathcal{T}} u_h \Vert_{\mathbf{L}^2 ( \Omega )}.
\end{align*}
By setting $a := \Vert \nabla_{\widehat{\mathcal{T}}} \widehat{e}_h \Vert_{\mathbf{L}^2 ( \Omega )}$, $b := \Vert \nabla_{\widehat{\mathcal{T}}} \widehat{u}_h^{\ast} - \nabla_{\mathcal{T}} u_h \Vert_{\mathbf{L}^2 ( \Omega )}$ and  $c := \mu ( u_h ; \mathcal{R}^{1/2} , f )$, we have shown in the previous estimate that $a^2 \lesssim ab + c(a+b)$, for $a,b,c \geq 0$. This implies $a \lesssim b+c$ by some elementary computation which is omitted here. Consequently, we have shown that
\begin{align}
\Vert \nabla_{\widehat{\mathcal{T}}} \widehat{e}_h \Vert_{\mathbf{L}^2 ( \Omega )}
&\lesssim
\mu ( u_h ; \mathcal{R}^{1/2} ,f )
+
\Vert \nabla_{\widehat{\mathcal{T}}} \widehat{u}_h^{\ast} - \nabla_{\mathcal{T}} u_h \Vert_{\mathbf{L}^2 ( \Omega )}.
\label{Eq:Almost final estimate}
\end{align}
Lemma \ref{Lem:Upper bound by tangential jumps} \eqref{SubEq:Non-Conf -  error estimate} is used to estimate $\Vert \nabla_{\widehat{\mathcal{T}}} \widehat{u}_h^{\ast} - \nabla_{\mathcal{T}} u_h \Vert_{\mathbf{L}^2 ( \Omega )}$ in \eqref{Eq:Almost final estimate} so that
\begin{align*}
\Vert \nabla_{\widehat{\mathcal{T}}} \widehat{e}_h \Vert_{\mathbf{L}^2 ( \Omega )}
&\lesssim
\mu ( u_h ; \mathcal{R}^{1/2} ,f )
+
\nu (u_h ; \mathcal{R}^1 ).
\end{align*}
Combining $\mathcal{R}^{1/2} \subseteq \mathcal{R}^1$, the fact that $\mu ( u_h ; \mathcal{R}^{1/2} ,f ), \nu (u_h ; \mathcal{R}^1 ) \geq 0$ by construction and the monotonicity of the square root function, we conclude that
$
\Vert \nabla_{\widehat{\mathcal{T}}} \widehat{u}_h - \nabla_{\mathcal{T}} u_h \Vert_{\mathbf{L}^2 ( \Omega )}
\lesssim
\eta ( u_h ; \mathcal{R}^1 , f)$ concluding the proof.
\end{proof}

Since the proof of Theorem \ref{Thm:Discrete reliabilty} is built upon Lemmata \ref{Lem:Non-Conf Approx Operator} -- \ref{Lem:Conformity estimation}, we provide proofs of these results in Sections \ref{Sec:NonConfApprox} -- \ref{Sec:Unseen}.

\section{The non-conforming approximation operator}
\label{Sec:NonConfApprox}

In this section, we provide the explicit construction of the non-conforming approximation operator $I_k^{\mathcal{T}, \CR} : H_k^{\CR} ( \mathcal{T} ) \to \CR_k ( \mathcal{T} )$. We prove Lemma \ref{Lem:Non-Conf Approx Operator} and derive some additional properties. In \cite[Sec. 6]{BohneSauterCiarletdD} a general construction of quasi-interpolation operators $I_{\mathcal{T},k} : H_0^1 ( \Omega ) + \CR_{k,0} ( \mathcal{T} ) \to \CR_{k,0} ( \mathcal{T} )$ of the form 
\begin{align}
I_{\mathcal{T},k} u
&=
\mathcal{I}_{\mathcal{T},k}^{\mathcal{E}} u 
+ 
\overset{\bullet \bullet}{\Pi}_{\mathcal{T},k} (u - \mathcal{I}_{\mathcal{T},k}^{\mathcal{E}} u )
\qquad \forall u \in H_0^1 ( \Omega ) + \CR_{k,0} ( \mathcal{T} )
\label{Eq:Class of apprximation}
\end{align}
is given.
The operator $\mathcal{I}_{\mathcal{T},k}^{\mathcal{E}}$ is explicitly constructed in \cite[(70)]{BohneSauterCiarletdD}, and $\overset{\bullet \bullet}{\Pi}_{\mathcal{T},k}: H_0^1 ( \Omega ) + \CR_{k,0} ( \mathcal{T} ) \to \overset{\bullet \bullet}{S}_k ( \mathcal{T})$ is some bounded linear operator with $\overset{\bullet \bullet}{S}_k ( \mathcal{T}) := \{v \in S_{k,0} ( \mathcal{T}) \; \vert \; \forall E \in \mathcal{E} ( \mathcal{T} ): \ v \vert_E = 0\}$. The operator $I_k^{\mathcal{T}, \CR}$ will be an instance of \eqref{Eq:Class of apprximation}.

For any $\mathbf{z} \in \mathcal{V} ( K )$ let $\varphi_{\mathbf{z}}^1 \in S_1 ( \mathcal{T} )$ be defined by $\varphi_{\mathbf{z}} ( \mathbf{y}) = \delta_{\mathbf{z},\mathbf{y}}$ for all $\mathcal{V} ( T )$ so that $ \varphi_{\mathbf{z}} \vert_K = \lambda_{K, \mathbf{z} }$ holds for $K \in \mathcal{T}_{\mathbf{z}}$ and $ \operatorname{supp} \varphi_{\mathbf{z}} = \omega_{\mathcal{T}_{\mathbf{z}}}$. 
Then for $K \in \mathcal{T}$ and $E \in \mathcal{E} ( \mathcal{T} )$, we define $W_K \in S_{3,0} ( \mathcal{T} )$ and $W_E \in S_2 ( \mathcal{T} )$ by $W_K
:=
\prod_{\mathbf{z} \in \mathcal{V} ( K )} \varphi_{\mathbf{z}}^1$ and $W_E 
:=
\prod_{\mathbf{z} \in \mathcal{V} ( E )} \varphi_{\mathbf{z}}^1$. We define the basis functions $b_{E,j}, b_{K, \boldsymbol{\alpha}} \in \CR_k ( \mathcal{T} )$ as
\begin{subequations} \label{Eq:Basis functions}
\begin{align}
\left. b_{E,k-1} \right\vert_K
&:=
\begin{cases}
P_k^{(0,0)} ( 1 - 2 \lambda_{K, E } )
&
\quad \; K \in \mathcal{T}_E,
\\
0
& \quad \;
K \in \mathcal{T} \setminus \mathcal{T}_E,
\end{cases}
\label{SubEq:Basis functions - nonconf}
\\
b_{E,j}
&:=
W_E
P_j^{(1,1)} ( 2 \varphi_{\mathbf{z}_0}^1 - 1)
\qquad
E = [\mathbf{z}_0, \mathbf{z}_1]\in \mathcal{E} ( \mathcal{T} ), \ 
j \in \range{k-2},
\label{SubEq:Basis functions - conf edge}
\\
b_{K, \boldsymbol{\alpha}} 
&:=
W_K
P_{K,\boldsymbol{\alpha}}^{(1,1,1)}
\qquad \qquad \quad \; \; \,
K \in \mathcal{T}, \ \boldsymbol{\alpha} \in \multirange{2}{k-3},
\label{SubEq:Basis functions - conf vol}
\end{align}
\end{subequations}
where $P_{K,\boldsymbol{\alpha}}^{( 1,1,1) }$ is given by \eqref{Eq:Bivariate ortho poly - physical} and $\lambda_{K, E } = \lambda_{K, \mathbf{z} }$ where $\mathbf{z} \in \mathcal{V} ( K )$ is the vertex of $K$ opposite of the edge $E$. We collect the functions \eqref{Eq:Basis functions} in the set
\begin{align*}
\mathcal{B}_k^{\CR} ( \mathcal{T} )
&:=
\{ \left. b_{K,\boldsymbol{\alpha}} 
\; \right\vert \;
K \in \mathcal{T}, \ \boldsymbol{\alpha} \in \multirange{2}{k-3}\}
\cup
\{ \left. b_{E,j} 
\; \right\vert \; 
E \in \mathcal{E} ( \mathcal{T} ), \ 
j \in \range{k-1}\}.
\end{align*}
From \cite[Lem. 19]{BohneSauterCiarletdD} we know that $\mathcal{B}_k^{\CR} ( \mathcal{T} )$ forms a basis of $\CR_k ( \mathcal{T} )$ and \eqref{SubEq:Basis functions - nonconf} implies $\operatorname{supp} b_{E,k-1} = \omega_{\mathcal{T}_E}$ for all edges $E \in \mathcal{E} ( \mathcal{T} )$. Since the edge and triangle bubbles are globally continuous, we observe the following properties of the functions $b_{E,j}$ and $b_{K,\boldsymbol{\alpha}}$: 
\begin{subequations} \label{Eq:Conf basis prop}
\begin{align}
\operatorname{supp} b_{E,j} 
= 
\omega_{\mathcal{T}_E} \quad
& \text{and} \quad
\operatorname{supp} b_{K, \boldsymbol{\alpha} }
= 
K
\label{SubEq:Conf basis prop - support}
\\
b_{E,j } 
\in S_k ( \mathcal{T} )
\quad &\text{and} \quad
b_{K, \boldsymbol{\alpha}} \in S_{k,0} ( \mathcal{T})
\label{SubEq:Conf basis prop - conf}
\end{align}
\end{subequations}
for all $E \in \mathcal{E}( \mathcal{T})$, $K \in \mathcal{T}$, $j \in \range{k-2}$ and $\boldsymbol{\alpha} \in \multirange{2}{k-3}$. Furthermore, the basis functions $b_{E,k-1}$ for $E \in \mathcal{E} ( \mathcal{T} )$ enjoy a useful orthogonality property.

\begin{lemma} \label{A:Lem:Basis volume orthogonality}
Let $\mathcal{T}$ and $k \geq 3$ odd be given. Then $b_{E,k-1} \perp_{L^2 (\Omega)} \mathbb{P}_{k-2} ( \mathcal{T} )$ for all $E \in \mathcal{E} ( \mathcal{T} )$.
\end{lemma}

\begin{proof}
Due to the support property of $b_{E,k-1}$ we only need to consider $K \in \mathcal{T}_E$. 
For any $p \in \mathbb{P}_{k-2} ( K )$, a pullback to the reference element $T$ via an affine linear transformation $\chi_K : T \to K$ leads to
\begin{align*}
\int_K p b_{E,k-1} 
&=
\frac{\vert K \vert}{\vert T \vert}
\int_{T} P_k^{( 0,0)} ( 1 - 2 \lambda_{T, \mathbf{Z}_1 } )
p_{K}
=
\frac{\vert K \vert}{\vert T \vert}
\int_0^1 
P_k^{(0,0)} ( 1-2x) 
 \left(\int_{0}^{1-x}
p_{K} (x,y) dy \right) dx,
\end{align*}
with $p_K := p \circ \chi_K$.
Since
$
\int_0^{1-x} p_{K} ( x, y) dy \in \mathbb{P}_{k-1} ( [0,1 ])
$,
we use an affine pullback to the interval $[-1,1]$ and the orthogonality property of $P_k^{( 0,0)}$, see \cite[18.2.1 \& Table 18.3.1]{NIST:DLMF}, to conclude.
\end{proof}

We exploit the splitting of $\mathcal{B}_k^{\CR} ( \mathcal{T} )$ into triangle and edge associated basis functions in our construction. For any $E \in \mathcal{E} ( \mathcal{T} )$ and $j \in \range{k-1}$, recall the definition of $g_{E,j} \in \mathbb{P}_{k-1} ( E )$ from \cite[(68b)]{BohneSauterCiarletdD}:
\begin{subequations}\label{Eq:Moment functions}
\begin{align}
g_{E,j} 
&:=
\frac{2 \gamma_j}{\vert E \vert} \left( P_j^{(1,1)} ( 2 \varphi_{\mathbf{z}_0,E}^1 -1 )- c_j P_{k-1}^{(1,1)} ( 2 \varphi_{\mathbf{z}_0,E}^1 -1 )  \right),
\label{SubEq:Moment functions - ege}
\end{align}
with constants $\gamma_j := (2j + 3) ( j +2 )/(8 j + 8)$, $c_j := (1 + (-1)^j) (k+1)/(2j + 4)$ for $j \in \range{k-2}$ and $c_{k-1} := 1/(2k+1)$ given in \cite[(68c)]{BohneSauterCiarletdD}. For $K \in \mathcal{T}$ and $\boldsymbol{\alpha} \in \multirange{2}{k-3}$, we define $g_{K,\boldsymbol{\alpha}} \in \mathbb{P}_{k-3} ( K )$ as
\begin{align}
g_{K,\boldsymbol{\alpha}}
&:=
\gamma_{K,\boldsymbol{\alpha}}
P_{K,\boldsymbol{\alpha}}^{( 1,1,1 )},
\label{SubEq:Moment functions - vol}
\end{align}
\end{subequations}
with $\gamma_{K,\boldsymbol{\alpha}}^{-1} = \Vert W_K^{1/2} P_{\boldsymbol{\alpha}}^{( 1,1,1 )} \Vert_{L^2 ( K )}^2$. 
The edge and volume functionals $F_{E,j}, F_{K,\boldsymbol{\alpha}} : H_k^{\CR} ( \mathcal{T} ) \to \mathbb{R}$ are given by
\begin{align}
\begin{matrix}
F_{E,j} (u) 
:=
( g_{E,j}  , u )_{L^2 ( E )}
& \text{and}&
F_{K,\boldsymbol{\alpha}} (u) 
:=
( g_{K,\boldsymbol{\alpha}}, u)_{L^2 ( K )}
\end{matrix}
\label{Eq:NC-Functionals}
\end{align}
for all $u \in H_k^{\CR} ( \mathcal{T} )$, $E \in \mathcal{E} ( \mathcal{T} )$, $j \in \range{k-1}$, $K \in \mathcal{T}$ and $\boldsymbol{\alpha} \in \multirange{2}{k-3}$. Remark \ref{Rem:Jump condition} implies that the edge functionals in \eqref{Eq:NC-Functionals} are well-defined and coincide with $J_{E,j}$ from \cite[Def. 28]{BohneSauterCiarletdD}. The following properties of the functionals in \eqref{Eq:NC-Functionals} and the basis $\mathcal{B}_k^{\CR} ( \mathcal{T})$ follow from \cite[Lem. 30 (69)]{BohneSauterCiarletdD}.

\begin{proposition}
\label{Prop:Quasi biduality properties}
The functionals $F_{E,j}, F_{K, \boldsymbol{\alpha}}: H_k^{\CR} ( \mathcal{T} ) \to \mathbb{R}$ satisfy
\begin{subequations} \label{Eq:Quasi biduality}
\begin{align}
F_{E,j} ( b_{E^{\prime},j^{\prime}} )
&=
\delta_{E,E^{\prime}} \delta_{j,j^{\prime}}
\qquad \ \;
\forall E, E^{\prime} \in \mathcal{E} ( \mathcal{T}), \
j \in \range{{k-1}},
\label{SubEq:Quasi biduality - edge}
\\
F_{E,j} ( b_{K,\boldsymbol{\alpha}} ) 
&= 
0
\qquad \qquad \qquad
\forall E \in \mathcal{E} ( \mathcal{T} ) , \ 
j \in \range{k-1}, \
K \in \mathcal{T}, \ 
\boldsymbol{\alpha} \in \multirange{2}{k-3},
\label{SubEq:Quasi biduality - mixed}
\\
F_{K,\boldsymbol{\alpha}} ( b_{K^{\prime}, \boldsymbol{\alpha}^{\prime}} )
&=
\delta_{K,K^{\prime}} \delta_{\boldsymbol{\alpha}, \boldsymbol{\alpha}^{\prime}}
\qquad
\forall K, K^{\prime} \in \mathcal{T}, \ \boldsymbol{\alpha}, \boldsymbol{\alpha}^{\prime} \in \multirange{2}{k-3}.
\label{SubEq:Quasi biduality - vol}
\end{align}
\end{subequations}
\end{proposition}

The functionals $F_{E,j}$ allow us to define three operators $\Pi_k^{\mathcal{E} (\mathcal{S}), \operatorname{nc}}, \Pi_k^{\mathcal{E} (\mathcal{S}), c}, \Pi_k^{\mathcal{E} ( \mathcal{S})} : H_k^{\CR} ( \mathcal{T} ) \to \CR_k ( \mathcal{T} )$ for any sub-mesh $\mathcal{S} \subseteq \mathcal{T}$:
\begin{subequations} \label{Eq:Edge Interpolation}
\begin{align}
\Pi_k^{\mathcal{E} ( \mathcal{S} ),\operatorname{nc}} u
&:=
\sum_{E \in \mathcal{E} ( \mathcal{S} )}
F_{E,k-1} ( u ) b_{E,k-1}
\qquad \forall u \in H_k^{\CR} ( \mathcal{T} ),
\label{SubEq:Edge Interpolation - nonconf}
\\
\Pi_k^{\mathcal{E} ( \mathcal{S} ),\operatorname{c}} u
&:=
\sum_{E \in \mathcal{E} ( \mathcal{S} )} \sum_{j = 0}^{k-2} 
F_{E,j} ( u ) b_{E,j}
\qquad \forall u \in H_k^{\CR} ( \mathcal{T} ),
\label{SubEq:Edge Interpolation - conf}
\\
\Pi_k^{\mathcal{E} ( \mathcal{S} )} u
&:=
\Pi_k^{\mathcal{E} ( \mathcal{S} ),\operatorname{nc}} u
+
\Pi_k^{\mathcal{E} ( \mathcal{S} ),\operatorname{c}} u
\qquad \forall u \in H_k^{\CR} ( \mathcal{T} ).
\label{SubEq:Edge Interpolation}
\end{align}
\end{subequations}

Furthermore, we set $\Pi_{k,0}^{\mathcal{S}} : H_k^{\CR} ( \mathcal{T} ) \to S_{k,0} ( \mathcal{T} )$, as
\begin{align}
\Pi_{k,0}^{\mathcal{S}}  u
&:=
\sum_{K \in \mathcal{S}}
\sum_{\boldsymbol{\alpha} \in \multirange{2}{k-3}}
F_{K,\boldsymbol{\alpha}} ( u ) b_{K,\boldsymbol{\alpha}}
\qquad \forall u \in H_k^{\CR} ( \mathcal{T} ).
\label{Eq:Volume Interpolation}
\end{align}
We have collected all operators needed to define the operator $I_k^{\mathcal{S}, \CR} : H_k^{\CR} ( \mathcal{T} ) \to \CR_k ( \mathcal{T} ) $ by combining \eqref{SubEq:Edge Interpolation} and \eqref{Eq:Volume Interpolation} into
\begin{align}
I_k^{\mathcal{S},\CR} u
:=
\Pi_k^{\mathcal{E} ( \mathcal{S} )} u
+
\Pi_{k,0}^{\mathcal{S}} \left( u - \Pi_k^{\mathcal{E} ( \mathcal{S})} u \right)
\qquad 
\forall u \in H_k^{\CR} ( \mathcal{T} ).  
\label{Eq:CR Interpolation}
\end{align}

\begin{notation} \label{Not:Boundary conditions}
The operators $\Pi_{k,0}^{\mathcal{E} (\mathcal{S}), \operatorname{nc}}, \ \Pi_{k,0}^{\mathcal{E} (\mathcal{S}), \operatorname{c}}, \ \Pi_{k,0}^{\mathcal{E} ( \mathcal{S})}$ and $I_{k,0}^{\mathcal{S}, \CR}$ are defined by restricting the sums in \eqref{Eq:Edge Interpolation} to \emph{inner} edges $\mathcal{E}_{\Omega} ( \mathcal{S})$.
\end{notation}

\begin{remark} \label{Rem:Non-conf quasi interpolation}
\begin{enumerate}
\item 
If $\mathcal{S} = \mathcal{T}$, the operator $\Pi_{k,0}^{\mathcal{T}}$ is a specific choice for $\overset{\bullet \bullet}{\Pi}_{\mathcal{T},k}$ in \eqref{Eq:Class of apprximation} and $\Pi_{k,0}^{\mathcal{E} ( \mathcal{T} )}$ coincides with $\mathcal{I}_{\mathcal{T},k}^{\mathcal{E}}$. Consequently, $I_{k,0}^{\mathcal{T},\CR}$ is an operator of type \eqref{Eq:Class of apprximation}.

\item In this section we only consider the case $\mathcal{S} = \mathcal{T}$. For the construction of the partially conforming operator $M_{k,0}^{\mathcal{S}}$ in Section \ref{Sec:PartConfOp}, we will consider the operator $I_{k,0}^{\mathcal{S},\CR}$ on (proper) sub-meshes $\mathcal{S} \subseteq \mathcal{T}$.
\end{enumerate}
\end{remark}
 
We use Proposition \ref{Prop:Quasi biduality properties} to derive some useful properties of the operator $I_k^{\mathcal{T}, \CR}$.
The following Lemma \ref{Lem:Functional preservation} provides some tools to prove Lemma \ref{Lem:Non-Conf Approx Operator}.
\begin{lemma} \label{Lem:Functional preservation}
The operator $I_k^{\mathcal{T},\CR} : H_k^{\CR} ( \mathcal{T} ) \to  \CR_k ( \mathcal{T} )$ in \eqref{Eq:CR Interpolation} satisfies
\begin{subequations} \label{Eq:Functional perservation}
\begin{align}
F_{E,j} ( I_k^{\mathcal{T}, \CR} u )
&=
F_{E,j} ( u )
\qquad \; \forall E \in \mathcal{E} ( \mathcal{T} ), \
j \in \range{k-1},
\label{SubEq:Functional preservation - edge}
\\
F_{K,\boldsymbol{\alpha}} ( I_k^{\mathcal{T}, \CR} u ) 
&=
F_{K,\boldsymbol{\alpha}} ( u )
\qquad \forall K \in \mathcal{T}, \
\boldsymbol{\alpha} \in \multirange{2}{k-3},
\label{SubEq:Functional perservation - vol}
\end{align}
\end{subequations}
for all $u \in H_k^{\CR} ( \mathcal{T} )$.
\end{lemma}

\begin{proof}
\textbf{@\labelcref{SubEq:Functional preservation - edge}:} For any $u \in H_k^{\CR} ( \mathcal{T} )$, $E \in \mathcal{E} ( \mathcal{T} )$ and $j \in \range{k-1}$, we apply Proposition \ref{Prop:Quasi biduality properties} for
\begin{align*}
F_{E,j} \left( I_k^{\mathcal{T}, \CR} u \right)
& \overset{\eqref{SubEq:Quasi biduality - mixed}}{=}
F_{E,j} \left( \Pi_k^{\mathcal{E} ( \mathcal{T})} u \right)
\overset{\eqref{SubEq:Edge Interpolation}}{=} 
\sum_{E^{\prime} \in \mathcal{E} ( \mathcal{T} )}
\sum_{j^{\prime} = 0}^{k-1}
F_{E^{\prime}, j^{\prime}} (u ) 
F_{E,j} ( b_{E^{\prime}, j^{\prime}})
\overset{\eqref{SubEq:Quasi biduality - edge}}{=}
F_{E,j} ( u ).
\end{align*}

\textbf{@\labelcref{SubEq:Functional perservation - vol}:} By using Proposition \ref{Prop:Quasi biduality properties} \eqref{SubEq:Quasi biduality - vol} and a similar reasoning as in \eqref{SubEq:Functional preservation - edge}, we obtain that $F_{K,\boldsymbol{\alpha}} ( \Pi_k^{\mathcal{T}} u ) = F_{K,\boldsymbol{\alpha}} ( u )$ for all $K \in \mathcal{T}$, $\boldsymbol{\alpha} \in \multirange{2}{k-3}$ and $u \in H_k^{\CR} ( \mathcal{T} )$. Then the linearity of the functional $F_{K,\boldsymbol{\alpha}}$ applied to $I_k^{\mathcal{T}, \CR} u$ concludes the proof.
\end{proof}

\begin{proof}[Proof of Lemma \ref{Lem:Non-Conf Approx Operator}]
\textbf{@\labelcref{SubEq:Non-Conf Approx Operator - edge moment}:} For any edge $E \in \mathcal{E} ( \mathcal{T} )$ we collect the functions $g_{E,j}  \in \mathbb{P}_{k-1} ( E )$, $j \in \range{k-1}$ in the set $\mathcal{B}_{k-1} ( E ) := \{ \left. g_{E,j} \in \mathbb{P}_k ( E ) \; \right\vert \; j \in \range{k-1} \}$. It is easy to check that $\mathcal{B}_{k-1} ( E )$ is a basis for $\mathbb{P}_{k-1} ( E )$. Indeed, this is a consequence of Proposition \ref{Prop:Quasi biduality properties} \eqref{SubEq:Quasi biduality - edge} and a dimension counting argument. Consequently, for any polynomial $q \in \mathbb{P}_{k-1} ( E )$ there exist constants $q_j \in \mathbb{R}$ for $j \in \range{k-1}$ such that $q = \sum_{j = 0}^{k-1} q_j g_{E,j}$. For any $u \in H_k^{\CR} ( \mathcal{T} )$ we apply Lemma \ref{Lem:Functional preservation} \eqref{SubEq:Functional preservation - edge} and obtain
\begin{align*}
\int_E
q I_k^{\mathcal{T}, \CR} u
&=
\sum_{j = 0}^{k-1}
q_j 
\int_E
g_{E,j} I_k^{\mathcal{T}, \CR} u
=
\sum_{j = 0}^{k-1}
q_j 
F_{E,j} \left( I_k^{\mathcal{T}, \CR} u \right)
=
\sum_{j = 0}^{k-1}
q_j 
F_{E,j} ( u )
=
\sum_{j = 0}^{k-1}
q_j 
\int_E
g_{E,j} u
=
\int_E q u.
\end{align*}
\textbf{@\labelcref{SubEq:Non-Conf Approx Operator- vol moment}:} For any triangle $K \in \mathcal{T}$ and $\mathcal{B}_{k-3} ( K ) := \{ b_{K,\boldsymbol{\alpha}} \; \vert  \; \boldsymbol{\alpha} \in \multirange{2}{k-3}\}$, the orthogonality property of the polynomials $P_{K, \boldsymbol{\alpha}}^{( 1,1,1)}$ with respect to the weight function $W_K$ (cf., \eqref{Eq:Pairwise orthogonality}) combined with a counting argument implies that $\mathcal{B}_{k-3} ( K )$ is a basis of $\mathbb{P}_{k-3} ( K )$. Consequently, for every polynomial $p \in \mathbb{P}_{k-3} ( K )$ there exist constants $p_{\boldsymbol{\alpha}} \in \mathbb{R}$, for $\boldsymbol{\alpha} \in \multirange{2}{k-3}$, such that $p = \sum_{\boldsymbol{\alpha} \in \multirange{2}{k-3}} p_{\boldsymbol{\alpha}} g_{K, \boldsymbol{\alpha}}$.
As before, we compute 
\begin{align*}
\int_K
p I_k^{\mathcal{T}, \CR} u
&=
\sum_{\boldsymbol{\alpha} \in \multirange{2}{k-3}} 
p_{\boldsymbol{\alpha}}
\int_K 
g_{K,\boldsymbol{\alpha}}
I_k^{\mathcal{T},\CR} u
=
\sum_{\boldsymbol{\alpha} \in \multirange{2}{k-3}}
p_{\boldsymbol{\alpha}}
F_{K,\boldsymbol{\alpha}} \left( I_k^{\mathcal{T},\CR} u \right)
\overset{\eqref{SubEq:Functional perservation - vol}}{=}
\sum_{\boldsymbol{\alpha} \in \multirange{2}{k-3}}
p_{\boldsymbol{\alpha}}
F_{K,\boldsymbol{\alpha}} (  u)
=
\int_K
p  u
\end{align*}
for any $u \in H_k^{\CR} ( \mathcal{T} )$.

\textbf{@\labelcref{SubEq:Non-Conf Approx Operator - loc bound}:} 
According to \cite[Lem. 2.1]{ArnoldBrezzi-mixedNonConvFEM85} for $K \in \mathcal{T}$, the local operator $I_{K,k}^{\operatorname{loc}} : H^1 ( K) \to  \mathbb{P}_k ( K )$ is well defined by
\begin{align}
\begin{matrix}
\int_K p ( u - I_{K,k}^{\operatorname{loc}} u )
=
0
& \text{and} &
\int_E q (u - I_{K,k}^{\operatorname{loc}} u) 
=
0
\end{matrix}
\label{Eq:Local moment preservation}
\end{align}
for all $u \in H^1 ( K )$,  $p \in \mathbb{P}_{k-3} ( K )$ and $q \in \mathbb{P}_{k-1} ( E )$ where $E \in \mathcal{E} ( K )$. 
This has two immediate consequences. First, since \eqref{Eq:Local moment preservation} is unisolvent for every polynomial $p \in \mathbb{P}_k ( K )$ as shown in the proof of \cite[Lem. 2.1]{ArnoldBrezzi-mixedNonConvFEM85}, we have that $I_{K,k}^{\operatorname{loc}}$ is a projection onto $\mathbb{P}_k ( K )$ and second, we have an estimate of the $L^2$ norm, i.e.,
\begin{subequations} \label{Eq:Prop LkK}
\begin{align}
I_{K,k}^{\operatorname{loc}} p 
&= 
p
\qquad \qquad \qquad \qquad \qquad \qquad \qquad \quad \ \ \forall p \in \mathbb{P}_k ( K ),
\label{SubEq:Prop LkK - Local projection}
\\
\Vert I_{K,k}^{\operatorname{loc}} u \Vert_{L^2 ( K )} 
&\lesssim 
\Vert u \Vert_{L^2 ( K )} + h_K^{1/2} \sum_{E \in \mathcal{E} ( K )} \Vert u \Vert_{L^2 ( E )}
\qquad
\forall u \in H_{k,0}^{\CR} ( \mathcal{T} ).
\label{SubEq:Prop LkK - L^2 estimate}
\end{align}
\end{subequations}
Furthermore, \eqref{Eq:Local moment preservation} is the restriction of \eqref{SubEq:Non-Conf Approx Operator - edge moment} and \eqref{SubEq:Non-Conf Approx Operator- vol moment} to a single triangle $K \in \mathcal{T}$ which implies
\begin{align}
\left. ( I_k^{\mathcal{T}, \CR} u ) \right\vert_K 
&=
I_{K,k}^{\operatorname{loc}} (\left. u \right\vert_K)
\qquad \forall u \in H_k^{\CR} ( \mathcal{T} ).
\label{Eq:Local desciption}
\end{align}
Hence $I_{K,k}^{\operatorname{loc}}$ is considered as the local version of $I_k^{\mathcal{T}, \CR}$.

For any function $u \in H_k^{\CR} ( \mathcal{T} )$, let $\overline{u}_K$ denote the integral mean of $u$ over $K$ as in \eqref{Eq:Integreal mean}. We use an inverse- and a multiplicative trace inequality \cite[(12.1) \& (12.16)]{ErnGuermondI} so that
\begin{align*}
\Vert \nabla I_k^{\mathcal{T}, \CR} u \Vert_{\mathbf{L}^2 ( K )}
& \overset{\eqref{Eq:Local desciption}}{=}
\Vert \nabla I_{K,k}^{loc} u \Vert_{\mathbf{L}^2 ( K )}
\overset{\eqref{SubEq:Prop LkK - Local projection}}{=} \Vert \nabla I_{K,k}^{loc} ( u - \overline{u}_K ) \Vert_{\mathbf{L}^2 ( K )}
\lesssim
h_K^{-1} \Vert  I_{K,k}^{loc} ( u - \overline{u}_K ) \Vert_{L^2 ( K )}
\\
&\overset{\eqref{SubEq:Prop LkK - L^2 estimate}}{\lesssim} 
h_K^{-1} 
( \Vert u -  \overline{u}_K \Vert_{L^2 ( K )}
+
h_K^{1/2} 
\Vert u -  \overline{u}_K \Vert_{L^2 ( E )} ) 
\\
& \, \ \ \lesssim
h_K^{-1} 
( \Vert u -  \overline{u}_K \Vert_{L^2 ( K )}
+
h_K^{1/2} 
\Vert u -  \overline{u}_K \Vert_{L^2 ( K )}^{1/2}
\Vert \nabla u \Vert_{L^2 ( K )}^{1/2} ) .
\end{align*}
A final application of a Poincaré inequality (see, e.g., \cite[Thm. 3.3]{CarstensenHellwig-DiscretePoissonFreidrichs}) concludes the proof.
\end{proof}

From Lemma \ref{Lem:Non-Conf Approx Operator} and its proof we derive two additional properties of $I_k^{\mathcal{T},\CR}$.

\begin{corollary}
\label{Cor:Global Projection}
\begin{enumerate}
\item \label{CorItem:Projection} For all functions $u \in \CR_k ( \mathcal{T} )$, it holds that $I_k^{\mathcal{T}, \CR} u = u$.

\item \label{CorItem:Exact bc} For all $u_0 \in H_{k,0}^{\CR} ( \mathcal{T} )$, it holds that $I_k^{\mathcal{T}, \CR} u_0 \in \CR_{k,0} ( \mathcal{T} )$, i.e.,
$
\left. I_k^{\mathcal{T},\CR} \right\vert_{H_{k,0}^{\CR} ( \mathcal{T} )} 
=
I_{k,0}^{\mathcal{T}, \CR}
$.
\end{enumerate}
\end{corollary}

\begin{proof}
\textbf{@1.:} This property follows from \eqref{SubEq:Prop LkK - Local projection}.

\textbf{@2.:} By definition of the space $H_{k,0}^{\CR} ( \mathcal{T} )$, we know that $F_{E,j} ( u_0) = 0$ for all $E \in \mathcal{E}_{\Omega} ( \mathcal{T} )$, $j \in \range{k-1}$ and $u_0 \in H_{k,0}^{\CR} ( \mathcal{T} )$. This yields $\Pi_k^{\mathcal{E} ( \mathcal{T} )} u_0 = \Pi_{k, 0}^{\mathcal{E} ( \mathcal{T} )}$ and therefore $I_k^{\mathcal{T}, \CR} u_0 = I_{k,0}^{\mathcal{T},\CR} u_0$. Furthermore, \eqref{SubEq:Non-Conf Approx Operator - edge moment} gives $\int_E q I_k^{\mathcal{T}, \CR} u_0 = \int_E q u_0 = 0$ for all boundary edges $E \in \mathcal{E}_{\partial \Omega} ( \mathcal{T} )$ and $q \in \mathbb{P}_{k-1} ( E )$ which implies $I_k^{\mathcal{T}, \CR} u_0 \in \CR_{k,0} ( \mathcal{T} )$.
\end{proof}

\section{The partially conforming operator}
\label{Sec:PartConfOp}

This section is devoted to the proof of Lemma \ref{Lem:Partially conforming}. As for the non-conforming quasi-interpolation operator $I_k^{\mathcal{T}, \CR}: H_k^{\CR} ( \mathcal{T})   \to \CR_k ( \mathcal{T} )$ from the previous section, 
we provide an explicit construction of the partially conforming operator $M_{k,0}^{\mathcal{S}} : H_{k,0}^{\CR} ( \mathcal{T} ) \to \CR_{k,0} ( \mathcal{T} )$ for some sub-mesh $\mathcal{S} \subseteq \mathcal{T}$.

In the previous section, we used the fact that the basis $\mathcal{B}_k^{\CR} ( \mathcal{T} )$ can be decomposed into edge associated basis functions and triangle associated basis functions. However, for the present construction, we exploit a different decomposition of $\CR_{k,0} ( \mathcal{T} )$. From \cite[Thm. 13]{BohneSauterCiarletdD}, we know that
\begin{align}
\begin{matrix}
\CR_k ( \mathcal{T} )
=
\overset{\bullet}{S}_k  ( \mathcal{T} )
\oplus 
V_k^{\operatorname{nc}} ( \mathcal{E} ( \mathcal{T}) )
& \text{and} &
\CR_{k,0}( \mathcal{T} )
=
\overset{\bullet}{S}_{k,0}  ( \mathcal{T} )
\oplus 
V_k^{\operatorname{nc}} ( \mathcal{E}_{\Omega} ( \mathcal{T} ) ),
\end{matrix}
\label{Eq:CR space decomposition}
\end{align}
where the spaces $\overset{\bullet}{S}_k  ( \mathcal{T} )$, $\overset{\bullet}{S}_{k,0}  ( \mathcal{T} )$ and $V_k^{\operatorname{nc}} ( \mathcal{E} ( \mathcal{T} ) )$ and $V_k^{\operatorname{nc}} ( \mathcal{E}_{\Omega} ( \mathcal{T} ) )$ are given by
\begin{subequations} \label{Eq:CR space decomposition - spaces}
\begin{align}
\overset{\bullet}{S}_k  ( \mathcal{T} )
&:=
\{ \left. v \in S_k ( \mathcal{T} )
\; \right\vert \; 
\forall \mathbf{z} \in \mathcal{V} ( \mathcal{T} ) : \
v ( \mathbf{z}) = 0 \}
\quad \text{and} \quad
\overset{\bullet}{S}_{k,0} ( \mathcal{T} )
:=
\overset{\bullet}{S}_k  ( \mathcal{T} ) 
\cap
S_{k,0} ( \mathcal{T} ),
\\
V_k^{\operatorname{nc}} ( \mathcal{E} ( \mathcal{T} ) )
&:=
\operatorname{span} 
\{ \left. b_{E,k-1} 
\; \right\vert \; 
E \in \mathcal{E} ( \mathcal{T} ) \}
\quad \text{and} \quad
V_k^{\operatorname{nc}} ( \mathcal{E}_{\Omega} ( \mathcal{T} ) )
:=
\operatorname{span} 
\{ \left. b_{E,k-1} 
\; \right\vert \; 
E \in \mathcal{E}_{\Omega}( \mathcal{T} ) \}.
\label{SubEq:CR space decomposition - spaces - non-conf}
\end{align}
\end{subequations}
From the proof of \cite[Lem. 19]{BohneSauterCiarletdD} we know that the sets
\begin{subequations} \label{Eq:Conf vertex basis}
\begin{align}
\overset{\bullet}{\mathcal{B}}_k ( \mathcal{T} )
&:=
\{ \left. b_{K,\boldsymbol{\alpha}} 
\; \right\vert \;
K \in \mathcal{T}, \
\boldsymbol{\alpha} \in \multirange{2}{k-3} \}
\cup
\{ \left. b_{E,j} 
\; \right\vert \;
E \in \mathcal{E} ( \mathcal{T} ), \
j \in \range{k-2} \}
\subseteq
\overset{\bullet}{S}_k ( \mathcal{T} ),
\label{SubEq:Conf vertex basis - full}
\\
\overset{\bullet}{\mathcal{B}}_{k,0} ( \mathcal{T} )
&:=
\{ \left. b_{K,\boldsymbol{\alpha}} 
\; \right\vert \;
K \in \mathcal{T}, \
\boldsymbol{\alpha} \in \multirange{2}{k-3} \}
\cup
\{ \left. b_{E,j} 
\; \right\vert \;
E \in \mathcal{E}_{\Omega} ( \mathcal{T} ), \
j \in \range{k-2} \}
\subseteq 
\overset{\bullet}{S}_{k,0}  ( \mathcal{T} )
\label{SubEq:Conf vertex basis - 0bc}
\end{align}
\end{subequations}
form a basis for $\overset{\bullet}{S}_k ( \mathcal{T} )$ and $\overset{\bullet}{S}_{k,0} ( \mathcal{T} )$ respectively. Since both $\Pi_{k,0}^{\mathcal{E} (\mathcal{T}), \operatorname{c}}$ and $\Pi_{k,0}^{\mathcal{T}}$ are linear combinations of the functions in $\overset{\bullet}{\mathcal{B}}_{k,0} ( \mathcal{T} )$, they map $H_{k,0}^{\CR} ( \mathcal{T} )$ to $\overset{\bullet}{S}_{k,0} ( \mathcal{T})$. We therefore combine the two operators into the operator $\overset{\bullet}{\Pi}_{k,0} : H_k^{\CR} ( \mathcal{T} ) \to \overset{\bullet}{S}_{k,0} ( \mathcal{T} )$, which reads
\begin{align}
\overset{\bullet}{\Pi}_{k,0} u
&:=
\Pi_{k,0}^{\mathcal{E} ( \mathcal{T} ),\operatorname{c}} u
+
\Pi_{k,0}^{\mathcal{T}} \left( u - \Pi_{k,0}^{\mathcal{E} ( \mathcal{T} ),c} u \right)
\qquad \forall u \in H_{k,0}^{\CR} ( \mathcal{T} ).
\label{Eq:SV part}
\end{align}

\begin{lemma}\label{Lem:Prop Pi bullet}
For every $u \in H_{k,0}^{\CR} ( \mathcal{T} )$, $E \in \mathcal{E}_{\Omega} ( \mathcal{T} )$, $K \in \mathcal{T}$, $j \in \range{k-1}$ and $\boldsymbol{\alpha} \in \multirange{2}{k-3}$
\begin{align}
\begin{matrix}
F_{E,j} (\overset{\bullet}{\Pi}_{k,0} u)
=
\begin{cases}
F_{E,j} (u)
& 
j \in \range{k-2},
\\
0 
& j= k-1
\end{cases}
& \text{and} &
F_{K, \boldsymbol{\alpha}} ( \overset{\bullet}{\Pi}_{k,0} u )
= F_{K, \boldsymbol{\alpha}} ( u )
\end{matrix}
\label{Eq:Prop Pibullet}
\end{align}
holds.
Furthermore,
\begin{align}
\overset{\bullet}{\Pi}_{k,0} v
&= 
v
\qquad 
\forall v \in \overset{\bullet}{S}_{k,0} ( \mathcal{T} ).
\label{Eq:Pibullet projection}
\end{align}
\end{lemma}

\begin{proof}
Property \eqref{Eq:Prop Pibullet} follows from the same the arguments given in the proof of Lemma \ref{Lem:Functional preservation}, while \eqref{Eq:Pibullet projection} is a consequence of \eqref{Eq:Prop Pibullet} and Corollary \ref{Cor:Global Projection} \eqref{CorItem:Projection}.
\end{proof}

Note that properties \eqref{Eq:Prop Pibullet} and \eqref{Eq:Pibullet projection} are not needed for the proof of Lemma \ref{Lem:Partially conforming} but will be used for construction of a conforming companion operator in Section \ref{Sec:RightInverse}.

Before we define the operator $M_{k,0}^{\mathcal{S}}$ we introduce an auxiliary operator $\Pi_{k,0}^{\mathcal{V} ( \mathcal{S} )} : H_{k,0}^{\CR} ( \mathcal{T} ) \to \CR_{k,0} ( \mathcal{T} )$. For any vertex $\mathbf{z} \in \mathcal{V} ( \mathcal{S} )$ let $\psi_{\mathcal{S},k}^{\mathbf{z}} \in \CR_k ( \mathcal{T} )$ be given by
\begin{align}
\psi_{\mathcal{S},k}^{\mathbf{z}}
&:=
\frac{1}{2}
\sum_{E \in \mathcal{E}_{\mathbf{z}} ( \mathcal{S} )}
b_{E,k-1} .
\label{Eq:S-conforming vertex basis}
\end{align}
Recall that $\mathcal{E}_{\mathbf{z}} ( \mathcal{S} )$ is the set of all edges in $\mathcal{S}$ that have $\mathbf{z}$ as an endpoint (cf. \eqref{Eq:S contained edge spider}) and that the summation of discontinuous functions is defined in Convention \ref{Not:Sumation convention in H1 broken}.
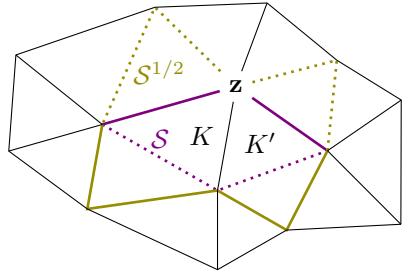
\begin{figure}
\centering
\begin{tikzpicture}[scale=.7]
\coordinate (zero) at (0,0);
\coordinate (ref) at (1,0);
\draw ($(zero)!2!80:(ref)$) node[] (z2) {$\mathbf{z}$};
\coordinate (z3) at ($(zero)!2.5!150:(ref)$);
\coordinate (z4) at ($(zero)!2.2!20:(ref)$);
\coordinate (z5) at ($(z2)!0.93!50:(z4)$);
\coordinate (z6) at ($(z3)!0.93!50:(z2)$);
\coordinate (z7) at ($(z3)!0.65!290:(zero)$);
\coordinate (z8) at ($(zero)!1.5!330:(ref)$);

\coordinate (z41) at ($(z4)!0.93!170:(z2)$);
\coordinate (z42) at ($(z4)!.8!250:(z2)$);
\coordinate (z21) at ($(z2)!.8!105:(z4)$);

\coordinate (z31) at ($(z3)!.8!125:(z2)$);
\coordinate (z32) at ($(z3)!.7!180:(z2)$);

\coordinate (zero1) at ($(zero)!1.5!270:(ref)$);

\draw[line width = 1, violet]  (z4) -- (z2) -- (z3);
\draw[line width = 1, violet, dotted]  (z3) -- (zero) -- (z4);
\draw [line width = 1, olive]  (z3) -- (z7) -- (zero) --(z8) --(z4);
\draw [line width = 1, olive, dotted] (z4) -- (z5) -- (z2) -- (z6) -- (z3);
\draw (z41) -- (z42) -- (z5) -- (z21) -- (z6) -- (z31) -- (z32) -- (z7) -- (zero1) -- (z8) -- (z41);

\draw (z2) -- (zero);
\draw (z4) -- (z41);
\draw (z4) -- (z42);
\draw (z2) -- (z21);
\draw (z3) -- (z31);
\draw (z3) -- (z32);
\draw (zero) -- (zero1);

\node at (barycentric cs:zero=1,z2=1,z3=.5) {$K$};
\node at (barycentric cs:zero=1,z2=1,z4=1) {$K^{\prime}$};

\node at (barycentric cs:z6=1,z2=.8,z3=1) {$\mathcolor{olive}{\mathcal{S}^{1/2}}$};
\node at (barycentric cs:z2=.31,zero=.6,z3=1) {$\mathcolor{violet}{\mathcal{S}}$};

\end{tikzpicture}
\caption{We set $\mathcal{S} := \{ K, K^{\prime} \}$. Then the purple line denotes the internal boundary $\Gamma$. The set $\mathcal{S}^{1/2}$ is outlined by the olive line. Then $\operatorname{supp} \psi_{\mathcal{S},k}^{\mathbf{z}}$ is encompassed by the dotted line.}
\label{Fig:S-conforming supprot}
\end{figure}
The following proposition is a consequence of \cite[Lem. 8]{BohneSauterCiarletdD}.

\begin{proposition}[{\cite[Lem. 8]{BohneSauterCiarletdD}}] \label{Prop:S-conforming prop}
For any vertex $\mathcal{V} ( \mathcal{S} )$ consider $\psi_{\mathcal{S},k}^{\mathbf{z}} \in \CR_k ( \mathcal{T} )$. Then for all vertices $\mathbf{z}, \mathbf{y} \in \mathcal{V} ( \mathcal{S} )$ the following three properties hold:
\begin{align}
\begin{matrix}
\supp \psi_{\mathcal{S},k}^{\mathbf{z}} 
= 
\omega_{(\mathcal{S}^{1/2})_{\mathbf{z}}}
,
&
\left. \psi_{\mathcal{S},k}^{\mathbf{z}} \right\vert_{\omega_{\mathcal{S}}} 
\in 
S_k ( \mathcal{S} )
& \text{and}&
\left. \psi_{\mathcal{S},k}^{\mathbf{z}} \right\vert_{\omega_{\mathcal{S}}} ( \mathbf{y} ) 
= 
\delta_{\mathbf{z}, \mathbf{y}}.
\end{matrix}
\label{Eq:S-conforming prop}
\end{align}
\end{proposition}

The support property of $\psi_{\mathcal{S},k}^{\mathbf{z}}$ is visualized in Figure \ref{Fig:S-conforming supprot}. Next, we define the functional $F_{\mathcal{S}, \mathbf{z}} : H_k^{\CR} ( \mathcal{T} ) \to \mathbb{R}$ for any vertex $\mathbf{z} \in \mathcal{V} ( \mathcal{T})$ as
\begin{align}
F_{\mathcal{S},\mathbf{z}} ( u )
&:=
\frac{1}{\operatorname{card} \mathcal{S}_{\mathbf{z}}}
\sum_{K \in \mathcal{S}_{\mathbf{z}}} 
( g_K^{\mathbf{z}}, u )_{L^2 ( K )}
\qquad \forall u \in H_k^{\CR} ( \mathcal{T} ),
\label{Eq:S-conforming vertex average}
\end{align}
where $g_K^{\mathbf{z}} \in  \mathbb{P}_k ( K )$ for $K \in \mathcal{S}_{\mathbf{z}}$ is 
given by
\begin{align}
g_K^{\mathbf{z}}
:= \frac{-1}{\vert K \vert}
\binom{k+2}{2}
P_k^{(0,2)} ( 1 - 2 \lambda_{K, \mathbf{z} } ).
\end{align}
Note that $g_k^{\mathbf{z}}$ is the same polynomial used in \cite[Def. 4 (27)]{GBS-PWStokesNumMat} leading to the following Proposition.

\begin{proposition}[{\cite[Lem 6. (30)]{GBS-PWStokesNumMat}, \cite[Prop 3.4 (3.16b)]{BohneGrässleSauter-PressureImprovedStokes}}] \label{Prop:Vertex value}
For every vertex $\mathbf{z} \in \mathcal{V} ( \mathcal{S} )$, it holds that,
\begin{subequations} \label{Eq:Point eval}
\begin{align}
(g_K^{\mathbf{z}}, p )_{L^2 ( K )}
&=
\left. q \right\vert_K ( \mathbf{z} )
\qquad \forall p \in \mathbb{P} ( \mathcal{T} ),
\label{SubEq:Point eval}
\\
\Vert g_K^{\mathbf{z} } \Vert_{L^2 ( K )} 
& =
h_K^{-1}.
\label{SubEq:Point eval - wheight estimate}
\end{align}
\end{subequations}
\end{proposition}
 
The combination of \eqref{SubEq:Point eval} and \eqref{Eq:S-conforming prop} leads to
\begin{align}
F_{\mathcal{S}, \mathbf{z}} ( \psi_{\mathcal{S},k}^{\mathbf{y}})
&=
\delta_{\mathbf{z}, \mathbf{y}}
\qquad \forall \mathbf{z}, \mathbf{y} \in \mathcal{V}  ( \mathcal{S} ).
\label{Eq:Biuality psi}
\end{align}
The nodal operator $\Pi_{k,0}^{\mathcal{V} ( \mathcal{S} )} : H_{k,0}^{\CR} ( \mathcal{T} ) \to \CR_{k,0} ( \mathcal{T} )$ is defined by
\begin{align}
\Pi_{k,0}^{\mathcal{V}( \mathcal{S} ) } u
&:=
\sum_{\mathbf{z} \in \mathcal{V}_{\Omega} ( \mathcal{S} ) }
F_{\mathcal{S},\mathbf{z}} ( u )
\psi_{\mathcal{S},k}^{\mathbf{z}}
\qquad \forall u \in H_{k,0}^{\CR} ( \mathcal{T} ).
\label{Eq:S-conforming vertex operator}
\end{align}
Finally, we need the operator $\Pi_{k,0}^{\mathcal{E} ( \mathcal{T} \setminus \mathcal{S}^{1/2}), \operatorname{nc}} : H_{k,0}^{\CR} ( \mathcal{T} ) \to \CR_{k,0} ( \mathcal{T} )$ given by \eqref{SubEq:Edge Interpolation - nonconf}.

\begin{lemma} \label{Lem:Interaction PiV PiEnc Pbullet}
For any sub-mesh $\mathcal{S} \subseteq \mathcal{T}$ it holds that:
\begin{subequations} \label{Eq:Interaction PiV PiEnc Pbullet}
\begin{align}
\Pi_{k,0}^{\mathcal{V} ( \mathcal{S} )} ( H_{k,0}^{\CR} ( \mathcal{T})),
\Pi_k^{\mathcal{E} ( \mathcal{T} \setminus \mathcal{S}^{1/2}),\operatorname{nc}} ( H_{k,0}^{\CR} ( \mathcal{T} ))
& \subseteq 
\operatorname{ker} \overset{\bullet}{\Pi}_{k,0},
\label{SubEq:Interaction PiV PiEnc Pbullet - conf}
\\
\Pi_k^{\mathcal{E} ( \mathcal{T} \setminus \mathcal{S}^{1/2}),\operatorname{nc}} ( H_{k,0}^{\CR} ( \mathcal{T} )) 
\subseteq  
\operatorname{ker} \Pi_{k,0}^{\mathcal{V} ( \mathcal{S} )} 
\text{ and } \
\Pi_{k,0}^{\mathcal{V} ( \mathcal{S} )} ( H_{k,0}^{\CR} ( \mathcal{T} )) 
&\subseteq  
\operatorname{ker} \Pi_k^{\mathcal{E} ( \mathcal{T} \setminus \mathcal{S}^{1/2}),\operatorname{nc}},
\label{SubEq:Interaction PiV PiEnc Pbullet - nc}
\end{align}
\end{subequations}
where $\operatorname{ker}$ denotes the kernel.
\end{lemma}

\begin{proof}
\textbf{@\labelcref{SubEq:Interaction PiV PiEnc Pbullet - conf}:} By construction we have
\begin{align}
\begin{matrix}
\Pi_{k,0}^{\mathcal{V} ( \mathcal{S})} u 
\in 
\operatorname{span} \{b_{E,k-1} \; \vert \; E \in \mathcal{E}_{\Omega} ( \mathcal{S} )\}
& \text{and} &
\Pi_k^{\mathcal{E} ( \mathcal{T} \setminus \mathcal{S}^{1/2}),\operatorname{nc}} u
\in
\operatorname{span}
\{b_{E,k-1} \; \vert \; E \in \mathcal{E}_{\Omega} ( \mathcal{T} \setminus \mathcal{S}^{1/2} )\}
\end{matrix}
\label{Eq:PiV PiEnc image span}
\end{align}
for all $u \in H_{k,0}^{\CR} ( \mathcal{T} )$.
Thus, Proposition \ref{Prop:Quasi biduality properties} \eqref{SubEq:Quasi biduality - edge} and Lemma \ref{A:Lem:Basis volume orthogonality} imply 
\begin{align*}
\overset{\bullet}{\Pi}_{k,0} (\Pi_{k,0}^{\mathcal{V} ( \mathcal{T} )} u )
&=
\Pi_{k,0}^{\mathcal{E} ( \mathcal{T}), c} ( \Pi_{k,0}^{\mathcal{V} ( \mathcal{T} )} u )
+
\Pi_{k,0}^{\mathcal{T}} ( \Pi_{k,0}^{\mathcal{V} ( \mathcal{T} )} u - \Pi_{k,0}^{\mathcal{E} ( \mathcal{T}), c} ( \Pi_{k,0}^{\mathcal{V} ( \mathcal{T} )} u ))
\overset{\eqref{SubEq:Quasi biduality - edge}}{=}
\Pi_{k,0}^{\mathcal{T}} ( \Pi_{k,0}^{\mathcal{V} ( \mathcal{T} )} u )
\overset{\text{Lem.} \ref{A:Lem:Basis volume orthogonality}}{=}
0
\end{align*}
and similarly for $\Pi_k^{\mathcal{E} ( \mathcal{T} \setminus \mathcal{S}^{1/2}),\operatorname{nc}} u$.

\textbf{@\labelcref{SubEq:Interaction PiV PiEnc Pbullet - nc}:} Recall that $\operatorname{supp} b_{E,k-1} = \omega_{\mathcal{T}_E}$ for all $E \in \mathcal{E}_{\Omega} ( \mathcal{T} \setminus \mathcal{S}^{1/2})$ and $\operatorname{supp} \psi_{\mathcal{S},k}^{\mathbf{z}} \subseteq \omega_{S^{1/2}}$ for every $\mathbf{z} \in \mathcal{V}_{\Omega} ( \mathcal{S} )$ (cf. \eqref{Eq:S-conforming prop}). We therefore conclude that
\begin{align}
\begin{matrix}
\displaystyle{\operatorname{supp} \Pi_{k,0}^{\mathcal{E} ( \mathcal{T} \setminus \mathcal{S}^{1/2}),\operatorname{nc}} u
\subseteq \bigcup_{E \in \mathcal{E} ( \mathcal{T} \setminus \mathcal{S}^{1/2} )} \omega_{\mathcal{T}_E}
= \omega_{\mathcal{T} \setminus \mathcal{S}}}
& \text{and} & 
\displaystyle{\operatorname{supp} \Pi_{k,0}^{V (\mathcal{T})} u 
\subseteq \bigcup_{\mathbf{z} \in \mathcal{V}_{\Omega} ( \mathcal{S} )} \operatorname{supp} \psi_{\mathcal{S},k}^{\mathbf{z}}
= \omega_{\mathcal{S}^{1/2}}}
\end{matrix}.
\label{Eq:Vert&NC - support}
\end{align}
Observe that the operator $\Pi_{k,0}^{\mathcal{V}( \mathcal{S})}$ only uses information from  $u \vert_{\omega_{\mathcal{S}}}$ through the functionals $F_{\mathcal{S},\mathbf{z}}$. This and \eqref{Eq:Vert&NC - support} yields $\Pi_{k,0}^{\mathcal{V} ( \mathcal{S} )} (\Pi_{k,0}^{\mathcal{E}( \mathcal{T} \setminus \mathcal{S}^{1/2})} u) =0$. Similarly, as the operator $\Pi_{k,0}^{\mathcal{E} ( \mathcal{T} \setminus \mathcal{S}^{1/2}), \operatorname{nc}}$ solely considers information on the edges $\mathcal{E} ( \mathcal{T} \setminus \mathcal{S}^{1/2} )$. Thus \eqref{Eq:PiV PiEnc image span} and Proposition \ref{Prop:Quasi biduality properties} \eqref{SubEq:Quasi biduality - edge} yield $\Pi_{k,0}^{\mathcal{E}( \mathcal{T} \setminus \mathcal{S}^{1/2}), \operatorname{nc}} \circ \Pi_{k,0}^{\mathcal{V} ( \mathcal{S} )} \equiv 0$.
\end{proof}

We define partially-conforming operator $M_{k,0}^{\mathcal{S}}: H_{k,0}^{\CR} ( \mathcal{T} ) \to \CR_{k,0} ( \mathcal{T} )$ by combining the operators $\Pi_{k, 0}^{\mathcal{V} ( \mathcal{S})}, \Pi_{k,0}^{\mathcal{E} ( \mathcal{T} \setminus \mathcal{S}^{1/2}), \operatorname{nc}}$ and $\overset{\bullet}{\Pi}_{k,0}$:
\begin{align}
M_{k,0}^{\mathcal{S}} u 
&:=
\Pi_{k,0}^{\mathcal{V}( \mathcal{S})} u
+
\Pi_{k,0}^{\mathcal{E}( \mathcal{T} \setminus \mathcal{S}^{1/2}),\operatorname{nc}} \left(u - \Pi_{k,0}^{\mathcal{V}( \mathcal{S})} u \right)
+
\overset{\bullet}{\Pi}_{k,0} \left(u - \Pi_{k,0}^{\mathcal{V}( \mathcal{S})} u  - \Pi_{k,0}^{\mathcal{E}( \mathcal{T} \setminus \mathcal{S}^{1/2}),\operatorname{nc}} \left(u - \Pi_{k,0}^{\mathcal{V}( \mathcal{S})} u \right) \right)
\nonumber
\\
&\overset{\eqref{Eq:Interaction PiV PiEnc Pbullet}}{=}
\Pi_{k,0}^{\mathcal{V}( \mathcal{S})} u
+
\Pi_{k,0}^{\mathcal{E}( \mathcal{T} \setminus \mathcal{S}^{1/2}),\operatorname{nc}} u
+
\overset{\bullet}{\Pi}_{k,0} u
\label{Eq:Mixed interpolation operator}
\end{align}
for all $u \in H_{k,0}^{\CR} ( \mathcal{T} )$.

Our approach for the proof of Lemma \ref{Lem:Partially conforming} \eqref{SubEq:Mixed Operator properties - loc bounded} is similar as for the non-conforming quasi-interpolation operator $I_k^{\mathcal{T}}$. We derive an $L^2$ estimate of the form $\Vert M_{k,0}^{\mathcal{S}} u \Vert_{L^2 ( K )} \lesssim \Vert u \Vert_{L^2 ( D_K )} + h_K \Vert \nabla_{\mathcal{T}} u \Vert_{\mathbf{L}^2 ( D_K )}$ for all $u \in H_{k,0}^{\CR} ( \mathcal{T} )$ and that $M_{k,0}^{\mathcal{S}}$ locally reproduces constant functions for triangles that do not touch the boundary $\partial \Omega$. Those two properties, combined with an inverse and a Poincaré-type inequality implies \eqref{SubEq:Mixed Operator properties - loc bounded} for interior triangles. The boundary case will be treated separately. We need four auxiliary results, beginning with the $L^2$ estimate.

\begin{lemma}\label{Lem:L2 boundedness partially conf}
For every $u \in H_{k,0}^{\CR} ( \mathcal{T})$ we have that $\Vert M_{k,0}^{\mathcal{S}} u \Vert_{L^2 ( K )} \lesssim \Vert u \Vert_{L^2 ( D_K )} + h_K \Vert \nabla_{\mathcal{T}} u \Vert_{\mathbf{L}^2 ( K )}$, where 
\begin{align}
D_K
:=
\begin{cases}
\mathcal{T}_K 
&
K \in \mathcal{S}^{1/2},
\\
K
& 
K \in \mathcal{T} \setminus \mathcal{S}^{1/2}.
\end{cases}
\label{Eq:L2 stability region}
\end{align}
\end{lemma}

The proof of this lemma relies on the estimates of the functionals $F_{E,j}$, $F_{K, \boldsymbol{\alpha}}$ and $F_{\mathcal{S}, \mathbf{z}}$, namely
\begin{subequations} \label{A:Eq:Functional estimates}
\begin{align}
\vert F_{E,j} ( u ) \vert
& \lesssim
\begin{cases}
h_K^{-1}
\Vert u \Vert_{L^2 ( K )} 
+ 
 \Vert \nabla u \Vert_{\mathbf{L}^2( K)} 
&
j = k-1,
\\
 \Vert \nabla u \Vert_{\mathbf{L}^2 ( K )}
&
j \in \range{k-2},
\end{cases}
\label{A:SubEq:Functional estimates - edge}
\\
\vert F_{K, \boldsymbol{\alpha}} ( u ) \vert
& \lesssim
h_K^{-1}
\Vert u \Vert_{L^2 ( K )} ,
\label{A:SubEq:Functional estimates - volume}
\\
\vert F_{\mathcal{S},\mathbf{z}} ( u ) \vert
&\lesssim
h_{\mathcal{T}_{\mathbf{z}}}^{-1} 
\Vert u \Vert_{L^2 ( \omega_{\mathcal{T}_{\mathbf{z}}} )},
\label{SubEq:Point eval - functional estimate}
\end{align}
\end{subequations}
for all $u \in H_{k,0}^{\CR} ( \mathcal{T} )$.
Since the proof of \eqref{A:Eq:Functional estimates} follows from standard arguments and it is postponed to Lemma \ref{A:Lem:Functional estimates} in the appendix. Recall the definition of the basis functions $b \in \CR_{k,0} ( \mathcal{T} )$ from \eqref{Eq:Basis functions}. By construction, we have that $\Vert b \Vert_{L^{\infty} (  K )} \lesssim 1$ for all $K \in \mathcal{T}$ with $K \subseteq \operatorname{supp} b$. Hence, by using affine pullbacks and an inverse inequality, e.g., \cite[(12.1)]{ErnGuermondI}, we conclude that
\begin{align}
\begin{matrix}
\Vert b \Vert_{L^2 ( K )} 
 \lesssim 
h_K
& \text{and} &
\Vert \nabla b \Vert_{\mathbf{L}^2 ( K )} 
& \lesssim 
1
\end{matrix}.
\label{A:Eq:Upperbound on basis functions}
\end{align}
for all $b \in \mathcal{B}_k^{\CR} ( \mathcal{T} )$ and $K \in \mathcal{T}$ such that $K \subseteq \operatorname{supp} b$.

\begin{proof}[Proof of Lemma \ref{Lem:L2 boundedness partially conf}]
From the definition of the partially-conforming operator $M_{k,0}^{\mathcal{S}}$ in \eqref{Eq:Mixed interpolation operator} we estimate that
\begin{align*}
\Vert M_{k,0}^{\mathcal{S}} u \Vert_{L^2 ( K )}
& \leq
\Vert \Pi_{k,0}^{\mathcal{V} ( \mathcal{S})} u \Vert_{L^2 ( K )}
+
\Vert \Pi_{k,0}^{\mathcal{E} ( \mathcal{T} \setminus \mathcal{S}^{1/2}), \operatorname{nc}} u \Vert_{L^2 ( K )}
+
\Vert \overset{\bullet}{\Pi}_{k,0} u \Vert_{L^2 ( K )}.
\end{align*}
We treat these contributions individually and start with $\Vert \Pi_{k,0}^{\mathcal{V} ( \mathcal{S})} u \Vert_{L^2 ( K )}$. Recalling the support properties of the operator $\Pi_{k,0}^{\mathcal{V} ( \mathcal{S} )}$ from \eqref{Eq:Vert&NC - support}, it is sufficient to consider $K \in \mathcal{S}^{1/2}$. Using \eqref{SubEq:Point eval - functional estimate} reveals
\begin{align*}
\Vert \Pi_{k,0}^{\mathcal{V} ( \mathcal{S})} u \Vert_{L^2 ( K )}
& \leq
\sum_{\mathbf{z} \in \mathcal{V} ( K ) \cap \mathcal{V} ( \mathcal{S} )}
\vert F_{\mathcal{S} , \mathbf{z}} ( u ) \vert
\Vert \psi_{\mathcal{S},k}^{\mathbf{z}} \Vert_{L^2 ( K )}
\lesssim
\Vert u \Vert_{L^2 ( \mathcal{T}_K )}
\sum_{\mathbf{z} \in \mathcal{V} ( K ) \cap \mathcal{V} ( \mathcal{S} )}
h_{\mathcal{T}_{\mathbf{z}}}^{-1}
\Vert \psi_{\mathcal{S},k}^{\mathbf{z}} \Vert_{L^2 ( K )}.
\end{align*}
Since $\psi_{\mathcal{S},k}^{\mathbf{z}} \vert_K$ is a linear combination of the functions $b_{E,k-1} \in \mathcal{B}_k^{\CR} ( \mathcal{T} )$ for at most two edges $E \in \mathcal{E} ( K )$, we combine \eqref{A:Eq:Upperbound on basis functions} and the equivalence of local mesh sizes and obtain 
\begin{align}
\Vert \Pi_{k,0}^{\mathcal{V} ( \mathcal{S})} u \Vert_{L^2 ( K )} 
&\lesssim 
\Vert u \Vert_{L^2 ( \mathcal{T}_K )}
=
\Vert u \Vert_{L^2 ( D_K )}.
\label{Eq:L2 estimate PiV}
\end{align}
For $\Vert \Pi_{k,0}^{\mathcal{E} ( \mathcal{T} \setminus \mathcal{S}^{1/2}), \operatorname{nc}} u \Vert_{L^2 ( K )}$ we use the support property of the operator $ \Pi_{k,0}^{\mathcal{E} ( \mathcal{T} \setminus \mathcal{S}^{1/2}), \operatorname{nc}}$ from \eqref{Eq:Vert&NC - support}. For $K \in \mathcal{T} \setminus \mathcal{S}$ it holds that
\begin{align}
\Vert \Pi_{k,0}^{\mathcal{E} ( \mathcal{T} \setminus \mathcal{S}^{1/2}), \operatorname{nc}} u \Vert_{L^2 ( K )}
& \leq
\sum_{E \in \mathcal{E} ( K ) \cap \mathcal{E} ( \mathcal{T} \setminus \mathcal{S}^{1/2)}}
\vert F_{E,k-1} ( u ) \vert
\Vert b_{E,k-1} \Vert_{L^2 ( K )} 
\nonumber
\\
&\overset{\eqref{A:SubEq:Functional estimates - edge}}{\lesssim}
(h_K^{-1} 
 \Vert u \Vert_{L^2 ( K )} 
+
\Vert \nabla u \Vert_{\mathbf{L}^2 ( K )})
\sum_{E \in \mathcal{E} ( K )}
\Vert b_{E,k-1} \Vert_{L^2 ( K )}
\nonumber
\\
& 
\overset{\eqref{A:Eq:Upperbound on basis functions}}{\lesssim}
\Vert u \Vert_{L^2 ( D_K )} 
+
h_K 
\Vert \nabla u \Vert_{\mathbf{L}^2 ( K )},
\label{Eq:L^2 estimate PiEnc}
\end{align}
using $K \subseteq D_K$ in the last step. 
Recalling the definition of $\overset{\bullet}{\Pi}_{k,0}$ from \eqref{Eq:SV part} we observe that
$
\Vert \overset{\bullet}{\Pi}_{k,0} u \Vert_{L^2 ( K )}
\leq 
\Vert \Pi_{k,0}^{\mathcal{E} ( \mathcal{T}), c} u\Vert_{L^2 ( K )}
+
\Vert \Pi_{k,0}^{\mathcal{T}} ( u -  \Pi_{k,0}^{\mathcal{E} ( \mathcal{T} ),c} u  ) \Vert_{L^2 ( K )}.
$ 
Therefore, a combination of \eqref{A:SubEq:Functional estimates - edge} and \eqref{A:Eq:Upperbound on basis functions} leads to
\begin{align}
\Vert \Pi_{k,0}^{\mathcal{E} ( \mathcal{T}),c} u\Vert_{L^2 ( K )}
& \leq
\sum_{E \in \mathcal{E} ( K )}
\sum_{j = 0}^{k-2}
\vert F_{E,j} ( u ) \vert
\Vert b_{E,j} \Vert_{L^2 ( K )}
\lesssim
h_K
\Vert \nabla u \Vert_{\mathbf{L}^2 ( K )}.
\label{Eq:L^2 estimate PE}
\end{align}
Similarly, combining \eqref{SubEq:Point eval - functional estimate} with \eqref{A:Eq:Upperbound on basis functions}, \eqref{Eq:L2 estimate PiV} and \eqref{Eq:L^2 estimate PE} reveals that
\begin{align*}
\Vert \Pi_{k,0}^{\mathcal{T}} ( u - \Pi_{k,0}^{\mathcal{E} ( \mathcal{T} ),c} u  ) \Vert_{L^2 ( K )}
& \leq
\sum_{\boldsymbol{\alpha} \in \multirange{2}{k-3}}
\vert F_{K, \boldsymbol{\alpha}} ( u - \Pi_{k,0}^{\mathcal{E} ( \mathcal{T} ),c} u ) \vert
\Vert b_{K,\boldsymbol{\alpha}}  \Vert_{L^2 ( K )}
\lesssim
\Vert u - \Pi_{k,0}^{\mathcal{E} ( \mathcal{T} ),c } u   \Vert_{L^2 ( K )}
\\
&\lesssim
\Vert u \Vert_{L^2 ( D_K )} 
+
h_K \Vert \nabla u \Vert_{\mathbf{L}^2 ( K )}.
\end{align*}
Consequently we have shown that $\Vert \overset{\bullet}{\Pi}_{k,0} u\Vert_{L^2 ( K )} \lesssim \Vert u \Vert_{L^2 ( D_K )} 
+
h_K \Vert \nabla u \Vert_{\mathbf{L}^2 ( K )}$ and thus combining this with \eqref{Eq:L2 estimate PiV} and \eqref{Eq:L^2 estimate PiEnc} concludes the proof.
\end{proof}

We us turn our attention to the preservation of locally constant functions.

\begin{lemma} \label{A:Lem:Mixed constatnt identity}
For all triangles $K \in \mathcal{T}$ such that $K \cap \partial \Omega = \emptyset$, we have that
\begin{align}
\left. M_{k,0}^{\mathcal{S}} ( c ) \right\vert_K 
&=
\left. c \right\vert_{K}
\label{Eq:Mixed constant identity}
\end{align}
for all functions $c \in H_{k,0}^{\CR} ( \mathcal{T} )$ such that $c \vert_{\omega_{\mathcal{T}_K}} \in \mathbb{R}$ is constant.
\end{lemma}

The idea of the proof is as follows: We show that $\left. (M_{k,0}^{\mathcal{S}} c )\right\vert_K = I_{K,k}^{\operatorname{loc}} (c \vert_K) $ for all $c \in H_{k,0}^{\CR} ( \mathcal{T} )$ with $c \vert_{\omega_{\mathcal{T}_K}} \in \mathbb{R}$ constant. The local projection property \eqref{SubEq:Prop LkK - Local projection} of $I_{K,k}^{\operatorname{loc}}$ then implies \eqref{Eq:Mixed constant identity}. This approach requires two ingredients:

First, recall from \eqref{Eq:Local desciption} that the operator $I_{K,k}^{\operatorname{loc}} : H^1 ( K ) \to \mathbb{P}_k ( K )$ satisfies $\left. ( I_k^{\mathcal{T}, CR} u ) \right\vert_K = I_{K,k}^{loc} ( u \vert_K)$ for all $u \in H_k^{\CR} ( \mathcal{T} )$. From this we deduce that $I_{K,k}^{\operatorname{loc}}$ can be written as
\begin{align}
I_{K,k}^{\operatorname{loc}} u 
&:=
\Pi_{K,k}^{\operatorname{edg}} u 
+
\Pi_{K,k}^{\operatorname{vol}} ( u - \Pi_{K,k}^{\operatorname{edg}} u),
\label{Eq:Local Interpolation - explicit}
\end{align}
where the two operators $\Pi_{K,k}^{\operatorname{edg}}, \Pi_{K,k}^{\operatorname{vol}} : H^1 ( K ) \to \mathbb{P}_k ( K )$ are given by 
\begin{align*}
\begin{matrix}
\Pi_{K,k}^{\operatorname{edg}} u 
:= 
\displaystyle{\sum_{E \in \mathcal{E} ( K )} 
\sum_{j = 0}^{k-1}
F_{E,j}  (u )
\left. b_{E,j} \right\vert_K}
& \text{and} &
\Pi_{K,k}^{\operatorname{vol}} u 
:=
\displaystyle{\sum_{\boldsymbol{\alpha} \in \multirange{2}{k-3}}
F_{K, \boldsymbol{\alpha}} ( u )
\left. b_{K,\boldsymbol{\alpha}} \right\vert_K}
\end{matrix}
\end{align*}
for all $u \in H^1 ( K )$. Rearranging the terms in \eqref{Eq:Local Interpolation - explicit} and applying Proposition \ref{Prop:Quasi biduality properties} \eqref{SubEq:Quasi biduality - edge} and Lemma \ref{A:Lem:Basis volume orthogonality} leads to
\begin{align}
I_{K,k}^{\operatorname{loc}} (u \vert_K)
&=
\Pi_{K,k}^{\operatorname{edg}} (u\vert_K )
+
\Pi_{K,k}^{\operatorname{vol}} \left( u \vert_K - \Pi_{K,k}^{\operatorname{edg}} (u\vert_K ) \right)
=
( \Pi_k^{\mathcal{E} ( \mathcal{T} ) ,\operatorname{nc}} u ) \vert_K
+
( \overset{\bullet}{\Pi}_{k,0} u ) \vert_K,
\label{Eq:Local Interpolation - equivalent descrp}
\end{align}
for $u \in H_k^{\CR} ( \mathcal{T} )$.

Second, we use that the functions $\psi_{\mathcal{S},k}^{\mathbf{z}}$ are linear combinations of the non-conforming basis functions $b_{E,k-1}$ and the fact that $F_{\mathcal{S},\mathbf{z}} ( c ) = F_{E,k-1} ( c )$ for all edges $E \in \mathcal{E}_{\mathbf{z}} ( \mathcal{S} )$. A proof of the latter equality is given in Lemma \ref{A:Lem:Equivalence Jz JEk-1 for constants} of the appendix.

\begin{proof}[Proof of Lemma \ref{A:Lem:Mixed constatnt identity}]
Assume $K \in \mathcal{T}$ satisfies $K \cap \partial \Omega = \emptyset$. For any $u \in H_{k,0}^{\CR} ( \mathcal{T} )$, restricting $M_{k,0}^{\mathcal{S}} u$ to $K$ reveals
\begin{align}
( M_{k,0}^{\mathcal{S}} v ) \vert_K
=&
( \Pi_{k,0}^{\mathcal{V}( \mathcal{S})} v  ) \vert_K 
+ 
( \Pi_{k,0}^{\mathcal{E} ( \mathcal{T} \setminus \mathcal{S}^{1/2}), \operatorname{nc}} v ) \vert_K
+
( \overset{\bullet}{\Pi}_{k,0} ) \vert_K.
\label{Eq:Mixed operator - complete local description}
\end{align}
Recalling the definitions of the operators $ \Pi_{k,0}^{\mathcal{V}( \mathcal{S})}$ and $\Pi_{k,0}^{\mathcal{E} ( \mathcal{T} \setminus \mathcal{S}^{1/2}) ,\operatorname{nc}}$ we compute that
\begin{subequations} \label{Eq:Vert&NC - local}
\begin{align}
( \Pi_{k,0}^{\mathcal{V}( \mathcal{S})} v  ) \vert_K 
&=
\sum_{\mathbf{z} \in \mathcal{V} ( K ) \cap \mathcal{V} ( \mathcal{S} )}
F_{\mathcal{S}, \mathbf{z}} ( v )
\psi_{\mathcal{S},k}^{\mathbf{z}} \vert_K,
\label{SubEq:Vert&NC - local- vert}
\\
( \Pi_{k,0}^{\mathcal{E} ( \mathcal{T} \setminus \mathcal{S}^{1/2}) , \operatorname{nc}} v ) \vert_K
&=
\sum_{E \in \mathcal{E} ( K ) \cap \mathcal{E} ( \mathcal{T} \setminus \mathcal{S}^{1/2} )}
F_{E,k-1} ( v )
b_{E,k-1}  \vert_K.
\label{SubEq:Vert&NC - local - nc}
\end{align}
\end{subequations}
Thus  $c \in H_{k,0}^{\CR} ( \mathcal{T} )$ such that $c \vert_{\omega_{\mathcal{T}_K}} \in \mathbb{R}$ is constant, we distinguish between three case based on the location of the triangle $K \in \mathcal{T}$.

\textbf{Case 1: $K \in \mathcal{T} \setminus \mathcal{S}^{1/2}$.} The support property \eqref{Eq:Vert&NC - support} of the operator $\Pi_{k,0}^{\mathcal{V} ( \mathcal{S}) }$ implies that
\begin{align}
( M_{k,0}^{\mathcal{S}} u ) \vert_K
&=
( \Pi_{k,0}^{\mathcal{E} ( \mathcal{T} \setminus \mathcal{S}^{1/2})} u ) \vert_K
+
( \overset{\bullet}{\Pi}_{k,0}u ) \vert_K
\overset{\eqref{Eq:Local Interpolation - equivalent descrp}}{=}
I_{K,k}^{loc} u
\overset{\eqref{SubEq:Prop LkK - Local projection}}{=}
u \vert_K
\label{Eq:Recover poly}
\end{align}
for every function $u \in \CR_{k,0} (\mathcal{T} )$.

\textbf{Case 2: $K \in \mathcal{S}$.} Applying the same arguments as in Case 1 to the operator $\Pi_{k,0}^{\mathcal{E} ( \mathcal{T} \setminus \mathcal{S}^{1/2}) , \operatorname{nc}}$ yields
\begin{align}
 ( M_{k,0}^{\mathcal{S}} c ) \vert_K
&=
( \Pi_{k,0}^{\mathcal{V}( \mathcal{S})} c  ) \vert_K 
+
( \overset{\bullet}{\Pi}_{k,0} c ) \vert_K.
\label{Eq:Recover constants - conforming first comp}
\end{align}
By assumption we know that $\mathcal{V} ( K ) \cap \mathcal{V} ( \mathcal{S} ) = \mathcal{V} ( K )$. Combining \eqref{SubEq:Vert&NC - local- vert} with Lemma \ref{A:Lem:Equivalence Jz JEk-1 for constants} and recalling the definition of $\psi_{\mathcal{S},k}^{\mathbf{z}}$ from \eqref{Eq:S-conforming vertex basis} yields
\begin{align*}
( \Pi_{k,0}^{\mathcal{V}( \mathcal{S})} c  ) \vert_K 
& \overset{\text{Lem. \ref{A:Lem:Equivalence Jz JEk-1 for constants}}}{=}
c
\sum_{\mathbf{z} \in \mathcal{V} ( K )}
\psi_{\mathcal{S},k} \vert_K
=
c
\sum_{\mathbf{z} \in \mathcal{V} ( K )}
\sum_{E \in \mathcal{E} ( K ) \atop
\mathbf{z} \in E }
\frac{1}{2} 
b_{E,k-1} \vert_K
=
c
\sum_{E \in \mathcal{E} ( K ) }
\sum_{\mathbf{z} \in \mathcal{V} ( K ) \atop
\mathbf{z} \in E}
\frac{1}{2}
b_{E,k-1} \vert_K
\\
&=
c
\sum_{E \in \mathcal{E} ( K )}
b_{E,k-1} \vert_K
\overset{\text{Lem. \ref{A:Lem:Equivalence Jz JEk-1 for constants}}}{=}
\sum_{E \in \mathcal{E} ( K )}
F_{E,k-1} ( c )
b_{E,k-1} \vert_K
\overset{\eqref{SubEq:Vert&NC - local - nc}}{=}
( \Pi_{k,0}^{\mathcal{E} ( \mathcal{T})} c ) \vert_K.
\end{align*}
We therefore have reduced  \eqref{Eq:Recover constants - conforming first comp} to \eqref{Eq:Recover poly}, concluding {Case 2}.
\begin{figure}
\centering
\begin{tikzpicture}[scale=.8]
\coordinate (zero) at (0,0);
\coordinate (ref) at (1,0);
\coordinate (z2) at ($(zero)!2!80:(ref)$);
\coordinate (z3) at ($(zero)!2.5!150:(ref)$);
\coordinate (z2+) at ($(zero)!1.2!0:(z2)$);
\coordinate (zero+) at ($(zero)!1!0:(ref)$);
\coordinate (zero-) at ($(zero)!-1!0:(ref)$);
\draw[dashed] (zero) -- (z2) --(z3) -- (zero);

\node at (barycentric cs:zero=-0.4,z2=.7,z3=1) {$( \mathbf{a} )$};
\node at (barycentric cs:zero=1,z2=1,z3=1) {$K$};
\node at (barycentric cs:zero=1,z2=1,z3=-.2)  {$\Gamma$};
\end{tikzpicture}
\hfil
\begin{tikzpicture}[scale=.8]
\coordinate (zero) at (0,0);2
\coordinate (ref) at (1,0);
\coordinate (z2) at ($(zero)!2!80:(ref)$);
\coordinate (z3) at ($(zero)!2.5!150:(ref)$);
\coordinate (z3+) at ($(zero)!2.7!150:(ref)$);
\coordinate (z2+) at ($(zero)!1.2!0:(z2)$);
\coordinate (zero+) at ($(zero)!1!0:(ref)$);
\coordinate (zero-) at ($(zero)!-1!0:(ref)$);
\draw  (z2) --(z3);
\draw[dashed] (z2+) -- (zero) -- (z3+);

\node at (barycentric cs:zero=-0.4,z2=.7,z3=1) {$( \mathbf{b} )$};
\node at (barycentric cs:zero=1,z2=1,z3=1) {$K$};
\node at (barycentric cs:zero=-0.2,z2=1,z3=.5) {$E$};
\node at (barycentric cs:zero=.6,z2=-.2,z3=1)  {$E^{\prime}$};
\node at (barycentric cs:zero=1,z2=-.2,z3=.6)  {$\Gamma$};
\node at (barycentric cs:zero=.6,z2=1,z3=-.2)  {$E^{\prime}$};
\node[anchor=north east] at (z3)  {$\mathbf{z}$};
\node[anchor=north] at (zero)  {$\mathbf{z}_E$};
\node[anchor=west] at (z2)  {$\mathbf{z}$};
\end{tikzpicture}
\hfil
\begin{tikzpicture}[scale=.8]
\coordinate (zero) at (0,0);2
\coordinate (ref) at (1,0);
\coordinate (z2) at ($(zero)!2!80:(ref)$);
\coordinate (z3) at ($(zero)!2.5!150:(ref)$);
\coordinate (z2+) at ($(zero)!1.2!0:(z2)$);
\coordinate (zero-) at ($(zero)!-0.2!0:(z2)$);

\draw  (z2) --(z3) -- (zero);
\draw[dashed] (z2+) -- (zero-);
\node at (barycentric cs:zero=-0.4,z2=.7,z3=1) {$( \mathbf{c} )$};
\node (T1) at (barycentric cs:zero=1,z2=1,z3=1) {$K$};
\node at (barycentric cs:z3=-.2,z2=1,zero=1)  {$\Gamma$};
\end{tikzpicture}
\caption{The three ways a triangle $K \in \mathcal{S}^{1/2} \setminus \mathcal{S}$ with $K \cap \partial \Omega = \emptyset$ intersects the interface $\Gamma$ (dashed line).}
\label{A:Fig:Gamma interaction}
\end{figure}
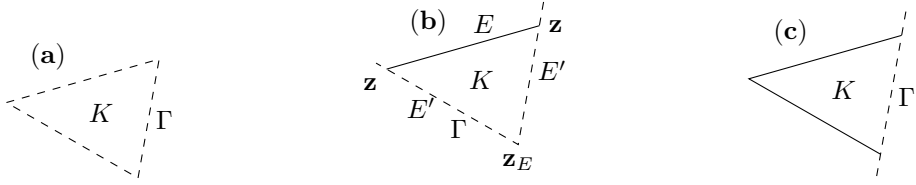 

\textbf{Case 3: $K \in \mathcal{S}^{1/2} \setminus \mathcal{S}$.} By construction, the supports of $ ( \Pi_{k,0}^{\mathcal{V}( \mathcal{S})} c  )$ and $ ( \Pi_{k,0}^{\mathcal{E} ( \mathcal{T} \setminus \mathcal{S}^{1/2})} c ) $ overlap on $K$. Furthermore we know by assumption that $\mathcal{E} ( K ) \cap \mathcal{E} ( \Gamma ) \neq \emptyset$. This transforms \eqref{Eq:Mixed operator - complete local description} into
\begin{align}
( \Pi_{k,0}^{\mathcal{V}( \mathcal{S})} c  ) \vert_K 
+
( \Pi_{k,0}^{\mathcal{E} ( \mathcal{T} \setminus \mathcal{S}^{1/2})} c ) \vert_K
&=
\sum_{\mathbf{z} \in \mathcal{V} ( K ) \cap\mathcal{V} (\Gamma)}
F_{\mathcal{S}, \mathbf{z}} ( c )
\psi_{\mathcal{S},k}^{\mathbf{z}} \vert_K
+
\sum_{E \in \mathcal{E} ( K ) \setminus \mathcal{E} ( \Gamma )}
F_{E,k-1} ( c )
 b_{E,k-1}  \vert_K.
\label{Eq:Mixed in mixing zone - complete description}
\end{align}
There are only three ways that $K$ interacts with the interface $\Gamma$ (see Figure \ref{A:Fig:Gamma interaction}). Consequently, we distinguish between three further cases.

\textit{Case 3(a): $\partial K \subseteq \Gamma$.} This implies that $\mathcal{E} ( K ) \setminus \mathcal{E} ( \Gamma ) = \emptyset$ and thus 
$  ( \Pi_{k,0}^{\mathcal{V}( \mathcal{S})} c  ) \vert_K 
+
 ( \Pi_{k,0}^{\mathcal{E} ( \mathcal{T} \setminus \mathcal{S}^{1/2})} c ) \vert_K
= \sum_{\mathbf{z} \in \mathcal{V} ( K ) }
F_{\mathcal{S}, \mathbf{z}} (c )
\psi_{\mathcal{S},k}^{\mathbf{z}} \vert_K$. Plugging this identity into \eqref{Eq:Mixed operator - complete local description} reveals that $ ( M_{k,0}^{\mathcal{S}} c ) \vert_K
=
( \Pi_{k,0}^{\mathcal{V}( \mathcal{S})} c  ) \vert_K 
+
(\overset{\bullet}{\Pi}_{k,0} c) \vert_K$. 
Following the arguments of {Case 2} concludes \textit{Case 3(a)}.

\textit{Case 3(b): there exists one and only one edge $E \in \mathcal{E} ( K )$ such that $K \cap \Gamma = \partial K \setminus \overset{\circ}{E}$.} This implies $\mathcal{E} ( K ) \setminus \mathcal{E} ( \Gamma ) = \{ E \}$ which yields $ ( \Pi_{k,0}^{\mathcal{E} ( \mathcal{T} \setminus \mathcal{S}^{1/2}), \operatorname{nc}} c ) \vert_K = F_{E,k-1} (c ) \left. b_{E,k-1} \right\vert_K$. Furthermore, we know that $\mathcal{V} ( K ) \cap \mathcal{V} ( \Gamma ) = \mathcal{V} ( K )$ as in \textit{Case 3(a)} and therefore we compute that $ ( \Pi_{k,0}^{\mathcal{V}( \mathcal{S})} c  ) \vert_K  = c 
\sum_{\mathbf{z} \in \mathcal{V} ( K ) }
\psi_{\mathcal{S},k}^{\mathbf{z}} \vert_K$ by an application of Lemma \ref{A:Lem:Equivalence Jz JEk-1 for constants}. Let $\mathbf{z}_E \in \mathcal{V} ( K )$ be the vertex opposite to $E$. Using \eqref{Eq:S-conforming vertex basis}, we observe that 
\begin{align*}
\begin{matrix}
\displaystyle{\left. \psi_{\mathcal{S},k}^{\mathbf{z}_E} \right\vert_K
=
\frac{1}{2} 
\sum_{E^{\prime} \in \mathcal{E} ( K ) \setminus \{E\}}
\left. b_{E^{\prime},k-1} \right\vert_K}
& \text{and} &
\displaystyle{\left. \psi_{\mathcal{S},k}^{\mathbf{z}} \right\vert_K = \frac{1}{2} \left. b_{E^{\prime},k-1}  \right\vert_K}
\end{matrix}
\end{align*}
where $E^{\prime} = [\mathbf{z}, \mathbf{z}_E]$. By combining everything and applying Lemma \ref{A:Lem:Equivalence Jz JEk-1 for constants}, we conclude 
\begin{align*}
( \Pi_{k,0}^{\mathcal{V}( \mathcal{S})} c  ) \vert_K 
+
( \Pi_{k,0}^{\mathcal{E} ( \mathcal{T} \setminus \mathcal{S}^{1/2})} c ) \vert_K
=
\sum_{E \in \mathcal{E} ( K )} 
F_{E, k-1} ( c) 
b_{E,k-1} \vert_K
\end{align*}
and therefore $( M_{k,0}^{\mathcal{S}} c ) \vert_K = I_{K,k}^{\operatorname{loc}} (c \vert_K) = c \vert_K$ follows.

\textit{Case 3(c): there exists a unique edge $E \in \mathcal{E} ( K )$ such that $E \subseteq \Gamma$.} The assertion for this case follows by similar arguments as in \textit{Case 3(b)}.
\end{proof}

As mentioned in the beginning of this section, we also need a Poincaré inequality on triangle patches for functions in $H_{k,0}^{\CR} ( \mathcal{T} )$. 

\begin{proposition}[{\cite[Thm 8.1]{Vohralik-DiscretePoincare}}] \label{A:Prop:Discrete Poincare}
Let $\mathcal{T}$ and $k \geq 1$ odd be given. Then the following estimates holds for all triangles $K \in \mathcal{T}$:
\begin{align} \label{A:Eq:Discrete Poincare}
\Vert u - \overline{u}_{\mathcal{T}_K}\Vert_{\mathbf{L}^2 ( \omega_{\mathcal{T}_K})}
& \lesssim
h_{\mathcal{T}_K} 
\Vert \nabla_{\mathcal{T}} u \Vert_{\mathbf{L}^2 ( \omega_{\mathcal{T}_K})}
\qquad \forall u \in H_k^{\CR} ( \mathcal{T} )
\end{align}
where $\overline{u}_{\mathcal{T}_K}$ is as in \eqref{Eq:Integreal mean}.
\end{proposition}

Proposition \ref{A:Prop:Discrete Poincare} allows us to provide an estimate for the integral mean of a function $u \in H_{k,0}^{\CR} ( \mathcal{T} )$ in a neighborhood of the boundary.

\begin{corollary} \label{Cor:Boundary integral mean estimate}
Let $K \in \mathcal{T}$ be given such that $K \cap \partial \Omega \neq \emptyset$. Then for all functions $u \in H_{k,0}^{\CR} ( \mathcal{T} )$
\begin{align}
\left\vert \overline{u}_{\mathcal{T}_K} \right\vert
& \lesssim
\Vert \nabla_{\mathcal{S}} u \Vert_{\mathbf{L}^2 ( \omega_{\mathcal{T}_K} )}.
\label{Eq:Boundary integral mean estimate}
\end{align}
\end{corollary}

\begin{proof}
Since $K \cap \partial \Omega \neq \emptyset$, we know that there exists a triangle $K^{\prime} \in \mathcal{T}_K$ with an edge $E \in \mathcal{E} ( K )$ that lies on the boundary, i.e., $E \in \mathcal{E}_{\partial \Omega} (\mathcal{T} )$. Since $\llbracket u \rrbracket_E \perp_{L^2 (E)} \mathbb{P}_{k-1} ( E )$ by assumption, a trace inequality leads to
\begin{align*}
\vert \overline{u}_{\mathcal{T}_K}  \vert
&=
\frac{1}{\vert E \vert} 
\left\vert\int_E \overline{u}_{\mathcal{T}_K} \right\vert
=
\frac{1}{\vert E \vert} 
\left\vert \int_E u - \overline{u}_{\mathcal{T}_K} \right\vert
\lesssim
h_{K^{\prime}}^{-1}
\Vert u  - \overline{u}_{\mathcal{T}_K} \Vert_{L^2 ( K^{\prime} )} 
+ 
\Vert \nabla u \Vert_{\mathbf{L}^2 ( K^{\prime} )} .
\end{align*}
Enlarging the support of the norms to $\omega_{\mathcal{T}_K}$ and applying Proposition \ref{A:Prop:Discrete Poincare} \eqref{A:Eq:Discrete Poincare} reveals
\begin{align*}
\vert \overline{u}_{\mathcal{T}_K} \vert
&\lesssim
\left(\frac{h_{\mathcal{T}_K}}{h_{K^{\prime}}} +1 \right)
\Vert \nabla_{\mathcal{T}} u \Vert_{\mathbf{L}^2 ( \omega_{\mathcal{T}_K} )}
\overset{\eqref{Eq:Equivalence local mesh size and volume}}{\lesssim}
\Vert \nabla_{\mathcal{T}} u \Vert_{\mathbf{L}^2 ( \omega_{\mathcal{T}_K} )}.
\qedhere
\end{align*}
\end{proof}

With these four results, we are able to proof Lemma \ref{Lem:Partially conforming}.

\begin{proof}[Proof of Lemma \ref{Lem:Partially conforming}]
\textbf{@\labelcref{SubEq:Mixed Operator properties - conforming on S}:}
Following the same arguments as for {Case 1} of the proof of Lemma \ref{A:Lem:Mixed constatnt identity} we compute that
\begin{align*}
( M_{k,0}^{\mathcal{S}} u ) \vert_{\omega_{\mathcal{S}}} 
&=
( \Pi_{k,0}^{\mathcal{V}( \mathcal{S})} u  ) \vert_{\omega_{\mathcal{S}}} 
+
(\overset{\bullet}{\Pi}_{k,0} u) \vert_{\omega_{\mathcal{S}}}
\qquad 
\forall u \in H_{k,0}^{\CR} ( \mathcal{T} ).
\end{align*}
We know that $(\overset{\bullet}{\Pi}_{k,0} u )\vert_{\omega_{\mathcal{S}}} \in S_k ( \mathcal{S} )$ by definition. Since $ \Pi_{k,0}^{\mathcal{V}( \mathcal{S})} u $ consists of a linear combination of the functions $\{ \psi_{\mathcal{S},k}^{\mathbf{z}} \; \vert \; \mathbf{z} \in \mathcal{V}_{\Omega} ( \mathcal{S} ) \}$, it follows that $( \Pi_{k,0}^{\mathcal{V}( \mathcal{S})} u  ) \vert_{\omega_{\mathcal{S}}} \in S_k ( \mathcal{S}) $ via an application of Proposition \ref{Prop:S-conforming prop} \eqref{Eq:S-conforming prop} and therefore that $( M_{k,0}^{\mathcal{S}} u ) \vert_{\omega_{\mathcal{S}}} \in S_k ( \mathcal{S} ) $.

\textbf{@\labelcref{ṢubEq:Mixed Operator properties - identity on TmSplus}:} This was already proven with \eqref{Eq:Recover poly}.

\textbf{@\labelcref{SubEq:Mixed Operator properties - loc bounded}:} Let $K \in \mathcal{T}$ be any triangle. We distinguish two cases based on the interaction of $K$ with $\partial \Omega$.

\textbf{Case 1: $K \cap \partial \Omega = \emptyset$.} Let $\overline{u}_{D_K}$ be the integral mean of $u$ over the region $D_K$ which is given by \eqref{Eq:L2 stability region}. Using an inverse inequality \cite[(12.1)]{ErnGuermondI} yields
\begin{align*}
\Vert \nabla M_{k,0}^{\mathcal{S}} u \Vert_{\mathbf{L}^2 ( K )}
&  \overset{\text{Lem.} \ref{A:Lem:Mixed constatnt identity}}{=}
\Vert \nabla M_{k,0}^{\mathcal{S}} ( u - \overline{u}_{D_K}) \Vert_{\mathbf{L}^2 ( K )}
\lesssim 
h_K^{-1} 
\Vert  M_{k,0}^{\mathcal{S}} ( u - \overline{u}_{D_K}) \Vert_{L^2 ( K )}
\\
& \overset{\text{Lem.} \ref{Lem:L2 boundedness partially conf}}{\lesssim}
h_K^{-1}
\Vert u - \overline{u}_{D_K} \Vert_{L^2 ( D_K )} 
+
\Vert \nabla_{\mathcal{T}} u \Vert_{\mathbf{L}^2 ( K )}.
\end{align*}
If $K \in \mathcal{T} \setminus \mathcal{S}^{1/2}$ then we know that $D_K = K$ and we employ the Poincaré inequality on triangles, see e.g. \cite[Thm. 3.3]{CarstensenHellwig-DiscretePoissonFreidrichs}, and obtain that $\Vert \nabla M_{k,0}^{\mathcal{S}} u \Vert_{\mathbf{L}^2 ( K )} \lesssim \Vert \nabla u \Vert_{\mathbf{L}^2 ( K )}$. Otherwise, we have that $D_K = \mathcal{T}_K$ and we apply Proposition \ref{A:Prop:Discrete Poincare} and the equivalence of locals mesh-sizes \eqref{Eq:Equivalence local mesh size and volume}, resulting in $\Vert \nabla M_{k,0}^{\mathcal{S}} u \Vert_{\mathbf{L}^2 ( K )} \lesssim \Vert \nabla_{\mathcal{T}} u \Vert_{\mathbf{L}^2 ( \omega_{\mathcal{T}_K} )}$.

\textbf{Case 2: $K \cap \partial \Omega \neq \emptyset$.} We distinguish three sub-cases, based on the location of $K$.

\textit{Case 2.1: $K \in \mathcal{T} \setminus \mathcal{S}^{1/2}$.} From Corollary \ref{Cor:Global Projection} \eqref{CorItem:Exact bc} and the explicit description of $I_{K,k}^{\operatorname{loc}}$ in \eqref{Eq:Local Interpolation - equivalent descrp}, it follows that $ \llbracket I_{K,k}^{\operatorname{loc}} u \rrbracket_E \perp_{L^2 (E)} \mathbb{P}_{k-1} ( E )$, for all $u \in H_{k,0}^{\CR} ( \mathcal{T} )$ and $E \in \mathcal{E}_{\partial \Omega} ( K )$. This implies that $( M_{k,0}^{\mathcal{S}} u ) \vert_K = I_{K,k}^{\operatorname{loc}} u$ by following the same reasoning as in \eqref{Eq:Recover poly} and therefore conclude that
$\Vert \nabla M_{k,0}^{\mathcal{S}} u \Vert_{\mathbf{L}^2 ( K )} = \Vert I_k^{\mathcal{T}, \CR} u \Vert_{\mathbf{L}^2 ( K )} \lesssim \Vert \nabla u \Vert_{\mathbf{L}^2 ( K )}$ through an application of Lemma \ref{Lem:Non-Conf Approx Operator} \eqref{SubEq:Non-Conf -  error estimate}.
\begin{figure}
\centering
\begin{tikzpicture}[scale=.8]
\coordinate (zero) at (0,0);2
\coordinate (ref) at (1,0);
\coordinate (z2) at ($(zero)!2!80:(ref)$);
\coordinate (z3) at ($(zero)!2.5!150:(ref)$);
\coordinate (z2+) at ($(zero)!1.2!0:(z2)$);
\coordinate (zero+) at ($(zero)!1!0:(ref)$);
\coordinate (zero-) at ($(zero)!-1!0:(ref)$);
\draw  (z2) --(z3) -- (zero);
\draw[dashed] (z2+) -- (zero);
\draw (zero+) -- (zero-);
\fill[pattern=north east lines, ] (zero+) -- (zero-) --([yshift=-.5em]zero-) -- ([yshift=-.5em]zero+);
\node at (barycentric cs:zero=-0.4,z2=1,z3=1) {$( \mathbf{a} )$};
\node at (barycentric cs:zero=1,z2=1,z3=1) {$K$};
\node at (barycentric cs:zero+=1,z2=-0.4,zero-=1)  {$\partial \Omega$};
\node at (barycentric cs:zero+=1,z2=1,zero=1)  {$\Gamma$};
\end{tikzpicture}
\hfil
\begin{tikzpicture}[scale=.8]
\coordinate (zero) at (0,0);2
\coordinate (ref) at (1,0);
\coordinate (z2) at ($(zero)!2!80:(ref)$);
\coordinate (z3) at ($(zero)!2.5!150:(ref)$);
\coordinate (z3+) at ($(zero)!2.7!150:(ref)$);
\coordinate (z2+) at ($(zero)!1.2!0:(z2)$);
\coordinate (zero+) at ($(zero)!1!0:(ref)$);
\coordinate (zero-) at ($(zero)!-1!0:(ref)$);
\draw  (z2) --(z3);
\draw[dashed] (z2+) -- (zero) -- (z3+);
\draw (zero+) -- (zero-);
\fill[pattern=north east lines, ] (zero+) -- (zero-) --([yshift=-.5em]zero-) -- ([yshift=-.5em]zero+);

\node at (barycentric cs:zero=-0.4,z2=1,z3=1) {$( \mathbf{b} )$};
\node at (barycentric cs:zero=1,z2=1,z3=1) {$K$};
\node at (barycentric cs:zero+=1,z2=-0.4,zero-=1)  {$\partial \Omega$};
\node at (barycentric cs:zero+=1,z2=1,zero=1)  {$\Gamma$};
\end{tikzpicture}
\hfil
\begin{tikzpicture}[scale=.8]
\coordinate (zero) at (0,0);
\coordinate (ref) at (1,0);
\coordinate (z2) at ($(zero)!2!80:(ref)$);
\coordinate (z3) at ($(zero)!2.5!150:(ref)$);
\coordinate (z2+) at ($(zero)!1.2!0:(z2)$);
\coordinate (zero-) at ($(zero)!-0.2!0:(z2)$);

\draw  (z2) --(z3) -- (zero);
\draw[dashed] (z2+) -- (zero);
\fill[pattern=north east lines, ] (zero) -- (z3) --([yshift=-.5em]z3.west) -- ([yshift=-.5em]zero);
\node at (barycentric cs:zero=-0.4,z2=1,z3=1) {$( \mathbf{c} )$};
\node at (barycentric cs:zero=-0.2,z2=1,z3=.5) {$E$};
\node (T1) at (barycentric cs:zero=1,z2=1,z3=1) {$K$};
\node at (barycentric cs:zero=1,z2=1,z3=-.2)  {$E_{\Gamma}$};
\node (T1) at (barycentric cs:zero=1,z2=-0.3,z3=1) {$E_{
\partial \Omega}$};

\node[anchor=west] (T1) at (zero) {$\mathbf{z}_{\partial \Omega}$};
\node[anchor=west] (T1) at (z2) {$\mathbf{z}$};
\node[anchor=south] (T1) at (z2+) {$\Gamma$};
\end{tikzpicture}
\hfil
\begin{tikzpicture}[scale=.8]
\coordinate (zero) at (0,0);2
\coordinate (ref) at (1,0);
\coordinate (z2) at ($(zero)!2!80:(ref)$);
\coordinate (z3) at ($(zero)!2.5!150:(ref)$);
\coordinate (z2+) at ($(zero)!1.2!0:(z2)$);
\coordinate (zero-) at ($(zero)!-0.2!0:(z2)$);

\draw  (z3) -- (zero);
\draw[dashed] (z3) -- (z2) -- (zero);
\fill[pattern=north east lines, ] (zero) -- (z3) --([yshift=-.5em]z3.west) -- ([yshift=-.5em]zero);
\node at (barycentric cs:zero=-0.4,z2=1,z3=1) {$( \mathbf{d} )$};
\node (T1) at (barycentric cs:zero=1,z2=1,z3=1) {$K$};
\node at (barycentric cs:zero+=1,z2=1,zero=1)  {$\Gamma$};
\node (T1) at (barycentric cs:zero=1,z2=-0.4,z3=1) {$\partial \Omega$};
\end{tikzpicture}
\caption{The four different ways a triangle $K \in \mathcal{S}^{1/2} \setminus \mathcal{S}$ can interact with the boundary. Here, the dashed line represents the interface $\Gamma$ between $\mathcal{S}$ and $\mathcal{T} \setminus \mathcal{S}$. The hached line segments represent the domain boundary $\partial \Omega$.}
\label{Fig:S+ boundary interaction}
\end{figure}
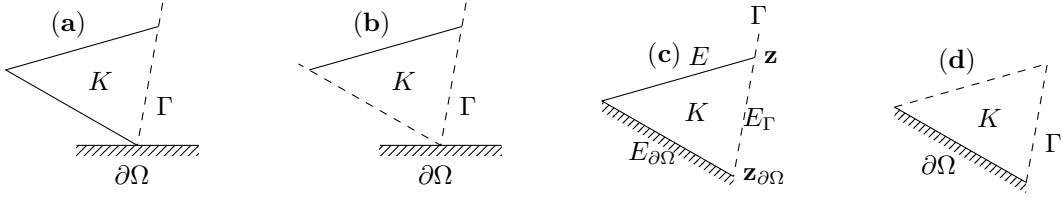

\textit{Case 2.2: $K \in \mathcal{S}^{1/2} \setminus \mathcal{S}$.} Since $K$ intersects with the interface $\Gamma$ (cf. \eqref{Eq:Gamma}), there are four ways that $K$ can interact with the boundary $\partial \Omega$, as illustrated in Figure \ref{Fig:S+ boundary interaction}. Since the proofs in all four cases are very similar, we only provide details for one of the four configurations. We choose Figure \ref{Fig:S+ boundary interaction}$( \mathbf{c})$ for this purpose. We write $\mathcal{E} ( K ) = \{ E, E_{\Gamma}, E_{\partial \Omega} \}$, where $E_{\Gamma} \subseteq \Gamma$ and $E_{\partial \Omega} \subseteq \partial \Omega$. Furthermore we consider $\mathbf{z}, \mathbf{z}_{\partial \Omega} \in \mathcal{V} ( K )$ where $\mathbf{z} , \mathbf{z}_{\partial \Omega} \in \Gamma$ and $\mathbf{z}_{\partial \Omega} \in \mathcal{V}_{\partial \Omega} ( K )$. Using this notation and recalling the definition of $M_{k,0}^{\mathcal{S}}$, we compute that
\begin{align}
\left. ( M_{k,0}^{\mathcal{S}}  u ) \right\vert_K
=
F_{E,k-1} ( u )
\left. b_{E,k-1} \right\vert_K
+
F_{\mathcal{S},\mathbf{z}} ( u )
\left. \psi_{\mathcal{S},k}^{\mathbf{z}} \right\vert_K
+
(\overset{\bullet}{\Pi}_{k,0} u) \vert_K.
\label{Eq:Mixed operator local boundary}
\end{align}
Let $\overline{u}_{\mathcal{T}_K}$ be the integral mean of $u$ over $\omega_{\mathcal{T}_K}$. The projection property \eqref{SubEq:Prop LkK - Local projection} of the local interpolation operator $I_{K,k}^{\operatorname{loc}}$ implies that $\overline{u}_{\mathcal{T}_K}  \vert_K = I_{K,k}^{\operatorname{loc}} \overline{u}_{\mathcal{T}_K}$. From Proposition \ref{Prop:Quasi biduality properties} \eqref{SubEq:Quasi biduality - edge}, we know that $F_{E^{\prime},j} ( \overline{u}_{\mathcal{T}} ) = 0$ for all $E \in \mathcal{E} ( K )$ and $j \in \range{k-2}$. Combining $F_{E^{\prime},k-1} ( \overline{u}_{\mathcal{T}_K}) = F_{\mathcal{S}, \mathbf{z}^{\prime}} ( \overline{u}_{\mathcal{T}_K})$ for all $E^{\prime}  \in \mathcal{E} ( K )$ and $\mathbf{z}^{\prime} \in \mathcal{V} ( K ) \cap \mathcal{V} ( \mathcal{S} )$ (cf. Lemma \ref{A:Lem:Equivalence Jz JEk-1 for constants}) with Lemma \ref{A:Lem:Basis volume orthogonality}, we rearrange $\overline{u}_{\mathcal{T}_K}  \vert_K = I_{K,k}^{\operatorname{loc}} \overline{u}_{\mathcal{T}_K}$ into
\begin{align}
\left. \overline{u}_{\mathcal{T}_K} \right\vert_K
&=
\sum_{\mathbf{z}^{\prime} \in \{ \mathbf{z}, \mathbf{z}_{\partial \Omega}\} }
F_{\mathcal{S},\mathbf{z}} ( \overline{u}_{\mathcal{T}_K} )
 \psi_{\mathcal{S},k}^{\mathbf{z}} \vert_K
+
\sum_{E^{\prime} \in \{ E , E_{\partial \Omega}\}}
F_{E^{\prime},k-1} ( \overline{u}_{\mathcal{T}_K} )
\left. b_{E^{\prime},k-1} \right\vert_K
+
\Pi_{K,k}^{\operatorname{vol}} \overline{u}_{\mathcal{T}_K},
\nonumber
\\
&=
F_{\mathcal{S}, \mathbf{z}_{\partial \Omega}} ( \overline{u}_{\mathcal{T}_K} ) 
\psi_{\mathcal{S},k}^{\mathbf{z}_{\partial \Omega}} \vert_K
+
F_{E_{\partial \Omega},k-1} ( \overline{u}_{\mathcal{T}_K})
b_{E_{\partial \Omega},k-1} \vert_K
+
(M_{k,0}^{\mathcal{S}} \overline{u}_{\mathcal{T}_K} ) \vert_K.
\label{Eq:Boundary mixing zone - description constant}
\end{align}
The combination of \eqref{Eq:Mixed operator local boundary}, \eqref{Eq:Boundary mixing zone - description constant}, and Lemma \ref{A:Lem:Equivalence Jz JEk-1 for constants} yields
\begin{align*}
\Vert \nabla M_{k,0}^{\mathcal{S}} u \Vert_{\mathbf{L}^2 ( K )}
& \leq
\Vert \nabla M_{k,0}^{\mathcal{S}} ( u - \overline{u}_{\mathcal{T}_K}) \Vert_{\mathbf{L}^2 ( K )}
+
\vert  \overline{u}_{\mathcal{T}_K}  \vert
\left(\Vert \nabla \psi_{\mathcal{S},k}^{\mathbf{z}_{\partial \Omega}} \Vert_{\mathbf{L}^2 ( K )}
+
\Vert \nabla b_{E_{\partial \Omega},k-1} \Vert_{\mathbf{L}^2 ( K )} \right).
\end{align*}
The term $\Vert \nabla M_{k,0}^{\mathcal{S}} ( u - \overline{u}_{\mathcal{T}_K}) \Vert_{\mathbf{L}^2 ( K )}$ is treated by following the arguments of {Case 1} and therefore, we have $\Vert \nabla M_{k,0}^{\mathcal{S}} ( u - \overline{u}_{\mathcal{T}_K}) \Vert_{\mathbf{L}^2 ( K )} \lesssim \Vert \nabla_{\mathcal{T}} u \Vert_{\mathbf{L}^2 ( \omega_{\mathcal{T}_K} )}$. 
Corollary \ref{Cor:Boundary integral mean estimate} leads to $\vert  \overline{u}_{\mathcal{T}_K}  \vert \lesssim \Vert \nabla_{\mathcal{T}} u \Vert_{\mathbf{L}^2 ( \omega_{\mathcal{T}_K})}$ and \eqref{A:Eq:Upperbound on basis functions} implies Do $\Vert \nabla \psi_{\mathcal{S},k}^{\mathbf{z}_{\partial \Omega}} \Vert_{\mathbf{L}^2 ( K )}
+
\Vert \nabla b_{E_{\partial \Omega},k-1} \Vert_{\mathbf{L}^2 ( K )} \lesssim 1$. Combining these two estimates yields
\begin{align*}
\vert  \overline{u}_{\mathcal{T}_K}  \vert 
(\Vert \nabla \psi_{\mathcal{S},k}^{\mathbf{z}_{\partial \Omega}} \Vert_{\mathbf{L}^2 ( K )}
+
\Vert \nabla b_{E_{\partial \Omega},k-1} \Vert_{\mathbf{L}^2 ( K )}) 
&\lesssim 
\Vert \nabla_{\mathcal{T}} u \Vert_{\mathbf{L}^2 ( \omega_{\mathcal{T}_K})}
\end{align*}
and, in turn, $\Vert \nabla M_{k,0}^{\mathcal{S}} u \Vert_{\mathbf{L}^2 ( K )} \lesssim \Vert \nabla_{\mathcal{T}} u \Vert_{\mathbf{L}^2 ( \omega_{\mathcal{T}_K} )}$ if $K \in \mathcal{S}^{1/2} \setminus \mathcal{S}$.

\textit{Case 2.3: $K \in \mathcal{S}$.} This case can be treated identically as \textit{Case 2.2} with some minor modification in the description of $I_{K,k}^{\operatorname{loc}} \overline{u}_{\mathcal{T}_K}$.
It follows that $\Vert \nabla M_{k,0}^{\mathcal{S}} u \Vert_{\mathbf{L}^2 ( K )} \lesssim \Vert \nabla_{\mathcal{T}} u \Vert_{\mathbf{L}^2 ( D_K )}$, where $D_K$ is as in \eqref{Eq:L2 stability region}.

\textbf{@\labelcref{SubEq:Mixed Operator properties - approx}:} For any $K \in \mathcal{T}$ let $\overline{u}_{D_K}$ be the integral mean over $D_K$. Based on the interaction of $K$ with the boundary, we distinguish two cases.

\textbf{Case 1: $K \cap \partial \Omega = \emptyset$.} Combining Lemmata \ref{Lem:L2 boundedness partially conf} and \ref{A:Lem:Mixed constatnt identity} with Proposition \ref{A:Prop:Discrete Poincare} leads to
\begin{align*}
\Vert ( \operatorname{Id} - M_{k,0}^{\mathcal{S}} ) u \Vert_{L^2 ( K )}
& \lesssim
\Vert u - \overline{u}_{D_K} \Vert_{L^2 ( D_K )}
+
h_K 
\Vert \nabla u \Vert_{\mathbf{L}^2 ( K )}
\lesssim
h_K 
\Vert \nabla u \Vert_{\mathbf{L}^2 ( D_K )}.
\end{align*}

\textbf{Case 2: $K \cap \partial \Omega \neq \emptyset$.} The approach is the same as in {Case 1} but one needs to incorporate the case distinctions from the proof of local boundedness \eqref{SubEq:Mixed Operator properties - loc bounded} {Case 2}. The details are omitted.
\end{proof}

\section{Right-inverse operator and non-conformity estimate}
\label{Sec:RightInverse}

In this section, we explicitly construct a conforming companion operator $J_{k,0} : H_{k,0}^{\CR} ( \mathcal{T} ) \to H_0^1 ( \Omega )$, such that $I_{k,0}^{\mathcal{T}, \CR} \circ J_{k,0} = \operatorname{Id}$ on $\CR_{k,0} ( \mathcal{T} )$. For this purpose, we adapt the ideas in \cite[Sec. 3.5]{CarstensenPuttkammer-howToProofDiscreteReliabiltiyNonConf} to any odd polynomial degree $k \geq 1$. 
The construction relies on two important properties: 
\begin{enumerate} 
\item \label{Item:Perservation DOF}  $J_{k,0}$ preserves the polynomial moments \eqref{SubEq:Non-Conf Approx Operator - edge moment} -- \eqref{SubEq:Non-Conf Approx Operator- vol moment} since they form a unisolvent set of degrees of freedom for $\CR_{k,0} ( \mathcal{T} )$.

\item \label{Item:Conforming identity} $J_{k,0} v = v$ for all conforming functions $v \in S_{k,0} ( \mathcal{T} )$.
\end{enumerate}

Given $E \in \mathcal{E} ( \mathcal{T} )$ we define the auxiliary function $\Phi_E \in S_{k+1} ( \mathcal{T} )$ as
\begin{align}
\Phi_E
&:= 
\gamma_E
W_E  P_{k-1}^{(1,1)} ( 2 \varphi_{\mathbf{z}_0}^1 - 1 ) 
\label{Eq:NonConf information carrier}
\end{align}
with the constant $\gamma_E  = 1 / F_{E,k-1} ( W_E  P_{k-1}^{(1,1)} ( 2 \varphi_{\mathbf{z}_0}^1 - 1 ) )$. 
Note that $\Phi_E$ is a scaled version of $b_{E,j}$ in \eqref{SubEq:Basis functions - conf edge} with $j \leftarrow k-1$.
Some elementary computations reveal that 
$
1/ \gamma_E
\gtrsim
\Vert W_E^{1/2} P_{k-1}^{( 1,1 )} (2 \varphi_{\mathbf{z}_0}^1 - 1 ) \Vert_{L^2 ( E )}^2
> 0,
$ 
and that $\operatorname{supp} \Phi_E = \omega_{\mathcal{T}_E}$. These two properties imply
\begin{align}
F_{E^{\prime},k-1} ( \Phi_E) 
&= 
\delta_{E,E^{\prime}}
\qquad \forall E^{\prime} \in \mathcal{E} ( \mathcal{T} )
\label{Eq:Biduality PhiE}
\end{align} 
and  $\Phi_E ( \mathbf{z} ) = 0$ for all $\mathbf{z} \in \mathcal{V} ( \mathcal{T} )$. 
The operator $\Pi_{k,0}^{\Phi} : \CR_{k,0} ( \mathcal{T} ) \to ·\overset{\bullet}{S}_{k+1,0} ( \mathcal{T} )$ is given by
\begin{align}
\Pi_{k,0}^{\Phi} u 
&:=
\sum_{E \in \mathcal{E}_{\Omega} ( \mathcal{T} )}
F_{E,k-1} ( u )
\Phi_E.
\label{Eq:NonConf information carrier - operator}
\end{align}
For every $u \in H_{k,0}^{\CR} ( \mathcal{T} )$ and $E \in  \mathcal{E}_{\Omega} ( \mathcal{T} )$ we compute that
\begin{align}
F_{E,j} ( \Pi_{k,0}^{\Phi} u  )
&=
\sum_{E^{\prime} \in \mathcal{E}_{\Omega} ( \mathcal{T} )}
F_{E^{\prime},k-1} ( u )
F_{E,k-1} ( \Phi_{E^{\prime}} )
\overset{\eqref{Eq:Biduality PhiE}}{=}
\sum_{E^{\prime} \in \mathcal{E}_{\Omega} ( \mathcal{T} )}
F_{E^{\prime},k-1} ( u )
\delta_{E,E^{\prime}}
=
F_{E,k-1} ( u ).
\label{Eq:Prop PiPhi}
\end{align}
We observe that both $\overset{\bullet}{\Pi}_{k,0}$ and $\Pi_{k,0}^{\Phi}$ map $H_{k,0}^{\CR} ( \mathcal{T} )$ into $\overset{\bullet}{S}_{k+1,0} ( \mathcal{T} )$, by construction. This means that, if we want \ref{Item:Conforming identity}.) to be satisfied, we need to preserve the vertex values of $v \in S_{k,0} ( \mathcal{T} )$.
To this end choose $\mathcal{S} = \mathcal{T}$ and recall the operator $\Pi_{k, 0}^{\mathcal{V} ( \mathcal{T} )} $ from \eqref{Eq:S-conforming vertex operator}. Thus $\psi_{\mathcal{T},k}^{\mathbf{z}} \in S_{k,0} ( \mathcal{T} )$ for all vertices $\mathbf{z} \in \mathcal{V}_{\Omega} ( \mathcal{T} )$ via an application of Proposition \ref{Prop:S-conforming prop}. Hence $\Pi_{k,0}^{\mathcal{V} ( \mathcal{T} )}$ maps $H_{k,0}^{\CR} ( \mathcal{T} )$ into $S_{k,0} ( \mathcal{T} )$. We combine the three operators $\Pi_{k,0}^{\mathcal{V} ( \mathcal{T} )}$, $\Pi_{k,0}^{\Phi}$, and $\overset{\bullet}{\Pi}_{k,0}$ in the definition of the conforming companion operator $J_{k,0} : H_{k,0}^{\CR} ( \mathcal{T} ) \to S_{k+1,0} ( \mathcal{T} ) \subseteq H_0^1 (\Omega)$:
\begin{align}
J_{k,0} u
&:=
\Pi_{k,0}^{\mathcal{V} ( \mathcal{T} )} u 
+ 
\Pi_{k,0}^{\Phi} ( u - \Pi_{k,0}^{\mathcal{V} ( \mathcal{T} )} u )
+
\overset{\bullet}{\Pi}_{k,0}  \left( u - \Pi_{k,0}^{\mathcal{V} ( \mathcal{T} )} u -  \Pi_{k,0}^{\Phi} ( u - \Pi_{k,0}^{\mathcal{V} ( \mathcal{T} )} u ) \right)
\nonumber
\\ &
\overset{\eqref{Eq:Interaction PiV PiEnc Pbullet}}{=}
\Pi_{k,0}^{\mathcal{V} ( \mathcal{T} )} u 
+ 
\Pi_{k,0}^{\Phi} ( u - \Pi_{k,0}^{\mathcal{V} ( \mathcal{T} )} u )
+
\overset{\bullet}{\Pi}_{k,0}  \left( u -  \Pi_{k,0}^{\Phi} ( u - \Pi_{k,0}^{\mathcal{V} ( \mathcal{T} )} u ) \right)
\label{Eq:Conforming companion}
\end{align}
for all $u \in H_{k,0}^{\CR} ( \mathcal{T} )$.

\begin{lemma} \label{Lem:Properties conforming companion}
The conforming companion operator $J_{k,0} : H_{k,0}^{\CR} ( \mathcal{T} ) \to S_{k+1,0} ( \mathcal{T} )$ satisfies
\begin{align}
J_{k,0} v 
&= 
v
\qquad \forall v \in S_{k,0} ( \mathcal{T} )
\label{Eq:Jkast conforming identity}
\end{align}
and for every function $u \in H_{k,0}^{\CR} ( \mathcal{T} )$, it holds that
\begin{subequations} \label{Eq:Jkast moment transport}
\begin{align}
\int_E q J_{k,0} u 
&=
\int_E q u 
\qquad \forall q \in \mathbb{P}_{k-1} ( E ), \ E \in \mathcal{E}_{\Omega} ( \mathcal{T} ),
\label{SubEq:Jkast moment transport - edge}
\\
\int_K p  J_{k,0} u
&=
\int_K p u
\qquad \forall q \in \mathbb{P}_{k-3} ( K ), \ K \in \mathcal{T}.
\label{SubEq:Jkast moment transport - vol}
\end{align}
\end{subequations}
\end{lemma}

\begin{proof}
\textbf{@\labelcref{Eq:Jkast conforming identity}:} Let $v \in S_{k,0} ( \mathcal{T} )$ be given. Recall from \eqref{Eq:Biuality psi} that $\psi_{\mathcal{T},k}^{\mathbf{z}}$ satisfies $F_{\mathcal{T}, \mathbf{y}} (\psi_{\mathcal{T},k}^{\mathbf{z}} ) = \delta_{\mathbf{z}, \mathbf{y}}$. Therefore, we can decompose $v$ into $v = v_0 + v_1$ where $v_1 = \sum_{\mathbf{z} \in \mathcal{V}_{\Omega} ( K )} v ( \mathbf{z} ) \psi_k^{\mathbf{z}}$ and $v_0 := v - v_1$. Now by construction we have that $v_1 ( \mathbf{z} ) = v ( \mathbf{z} )$ for all vertices $\mathbf{z} \in \mathcal{V}_{\Omega} ( \mathcal{T} )$ and thus $v_0 \in  \overset{\bullet}{S}_{k,0}  ( \mathcal{T} )$. Hence we obtain
\begin{align*}
\Pi_{k,0}^{\mathcal{V} ( \mathcal{T} )} v
&=
\Pi_{k,0}^{\mathcal{V} ( \mathcal{T} )} v_1
= 
\sum_{\mathbf{z} \in \mathcal{V}_{\Omega} ( \mathcal{T} ) }
F_{\mathcal{T}, \mathbf{z}} ( v_1 )
\psi_{\mathcal{T},k}^{\mathbf{z}}
=
\sum_{\mathbf{z}, \mathbf{y} \in \mathcal{V}_{\Omega} ( \mathcal{T} ) }
v ( \mathbf{y} )
F_{\mathcal{T}, \mathbf{z}} ( \psi_{\mathcal{T},k}^{\mathbf{y}})
\psi_{\mathcal{T},k}^{\mathbf{z}}
\overset{\eqref{Eq:Biuality psi}}{=}
\sum_{\mathbf{z} \in \mathcal{V}_{\Omega} ( \mathcal{T} ) }
v ( \mathbf{z} )
\psi_{\mathcal{T},k}^{\mathbf{z}}
=
v_1
\end{align*}
and therefore conclude that $v - \Pi_{k,0}^{\mathcal{V} ( \mathcal{T} )} v =  v_0 $ holds. Since $\overset{\bullet}{\mathcal{B}}_{k,0} ( \mathcal{T} )$ is a basis of $\overset{\bullet}{S}_{k,0} ( \mathcal{T} )$,  $F_{E,k-1} ( v_0 ) = 0$ for all inner edges $E \in \mathcal{E}_{\Omega} ( \mathcal{T} )$ via an application of Proposition \ref{Prop:Quasi biduality properties} \eqref{SubEq:Quasi biduality - edge} we have that $\Pi_{k,0}^{\Phi} (v_0) = 0$.
This allows us to compute that
\begin{align*}
J_{k,0} v
&= 
\Pi_{k,0}^{\mathcal{V} ( \mathcal{T} )} v 
+ 
\Pi_{k,0}^{\Phi} v_0
+
\overset{\bullet}{\Pi}_{k,0}  ( v - \Pi_{k,0}^{\Phi} v_0 )
=
\Pi_{k,0}^{\mathcal{V} ( \mathcal{T} )} v 
+ 
\overset{\bullet}{\Pi}_{k,0}  v 
\overset{\text{Lem. \ref{Lem:Interaction PiV PiEnc Pbullet}}}{=}
\Pi_{k,0}^{\mathcal{V} ( \mathcal{T} )} v 
+
\overset{\bullet}{\Pi}_{k,0} ( v
-
\Pi_{k,0}^{\mathcal{V} ( \mathcal{T} )} v ).
\end{align*}
Since $\overset{\bullet}{\Pi}_{k,0}$ is a projection onto $ \overset{\bullet}{S}_{k,0} (\mathcal{T})$ (cf. \eqref{Eq:Pibullet projection}) we conclude that 
\begin{align*}
J_{k,0} v 
&= 
\Pi_{k,0}^{\mathcal{V} ( \mathcal{T} )} v 
+
\overset{\bullet}{\Pi}_{k,0} ( v
-
\Pi_{k,0}^{\mathcal{V} ( \mathcal{T} )} v ) = \Pi_{k,0}^{\mathcal{V} ( \mathcal{T} )} v +  v
-
\Pi_{k,0}^{\mathcal{V} ( \mathcal{T} )} v = v.
\end{align*}

\textbf{@\labelcref{Eq:Jkast moment transport}:}  For any $u \in H_{k,0}^{\CR} ( \mathcal{T} )$ consider some $E \in \mathcal{E}_{\Omega} ( \mathcal{T} )$ and $j \in \range{k-1}$. For $j = k-1$ we combine \eqref{Eq:Prop Pibullet} and \eqref{Eq:Prop PiPhi} for
\begin{align*}
F_{E,k-1} ( J_{k,0} u ) 
&=
F_{E,k-1} ( \Pi_{k,0}^{\mathcal{V} ( \mathcal{T} )} u ) 
+
F_{E,k-1} ( \Pi_{k,0}^{\Phi} ( u - \Pi_{k,0}^{\mathcal{V} ( \mathcal{T} )} u ))
+
F_{E,k-1} ( \overset{\bullet}{\Pi}_{k,0} ( u - \Pi_{k,0}^{\Phi} ( u - \Pi_{k,0}^{\mathcal{V} ( \mathcal{T} )} u )))
\\
&=
F_{E,k-1} ( \Pi_{k,0}^{\mathcal{V} ( \mathcal{T} )} u ) 
+
F_{E,k-1} ( u - \Pi_{k,0}^{\mathcal{V} ( \mathcal{T} )} u )
= F_{E,k-1} ( u ).
\end{align*}
For $j \in \range{k-2}$ we apply Proposition \ref{Prop:Quasi biduality properties} \eqref{SubEq:Quasi biduality - edge} and \eqref{Eq:Prop Pibullet} to get
\begin{align*}
F_{E,j} ( J_{k,0} u ) 
&=
F_{E,j} ( \Pi_{k,0}^{\mathcal{V} ( \mathcal{T} )} u ) 
+
F_{E,j} ( \Pi_{k,0}^{\Phi} ( u - \Pi_{k,0}^{\mathcal{V} ( \mathcal{T} )} u ))
+
F_{E,j} ( \overset{\bullet}{\Pi}_{k,0} ( u - \Pi_{k,0}^{\Phi} ( u - \Pi_{k,0}^{\mathcal{V} ( \mathcal{T} )} u )))
\\
&=
F_{E,j} ( \Pi_{k,0}^{\Phi} ( u - \Pi_{k,0}^{\mathcal{V} ( \mathcal{T} )} u ))
+
F_{E,j} (u - \Pi_{k,0}^{\Phi} ( u - \Pi_{k,0}^{\mathcal{V} ( \mathcal{T} )} u ))
= F_{E,j} ( u ).
\end{align*}
For $K \in \mathcal{T}$ and $\boldsymbol{\alpha} \in \multirange{2}{k-3}$, we combine Lemma \ref{A:Lem:Basis volume orthogonality} and \eqref{Eq:Prop Pibullet} to deduce that
\begin{align*}
F_{K, \boldsymbol{\alpha}} ( J_{k,0} u ) 
&=
F_{K, \boldsymbol{\alpha}}  ( \Pi_{k,0}^{\mathcal{V} ( \mathcal{T} )} u ) 
+
F_{K, \boldsymbol{\alpha}}  ( \Pi_{k,0}^{\Phi} ( u - \Pi_{k,0}^{\mathcal{V} ( \mathcal{T} )} u ))
+
F_{K, \boldsymbol{\alpha}}  ( \overset{\bullet}{\Pi}_{k,0} ( u - \Pi_{k,0}^{\Phi} ( u - \Pi_{k,0}^{\mathcal{V} ( \mathcal{T} )} u )))
\\
&=
F_{K, \boldsymbol{\alpha}} ( \Pi_{k,0}^{\Phi} ( u - \Pi_{k,0}^{\mathcal{V} ( \mathcal{T} )} u ))
+
F_{K, \boldsymbol{\alpha}}  (u - \Pi_{k,0}^{\Phi} ( u - \Pi_{k,0}^{\mathcal{V} ( \mathcal{T} )} u ))
=
F_{K, \boldsymbol{\alpha}} ( u ).
\end{align*}
We thus have shown that the conforming companion operator $J_{k,0}$ satisfies both 
\begin{subequations} \label{Eq:Functional preservation Jk0}
\begin{align}
F_{E,j} ( J_{k,0} v ) 
&=
F_{E,j} ( u )
\qquad \forall E  \in \mathcal{E} ( \mathcal{T} ), \
j \in \range{k-1},
\label{SubEq:Functional preservation Jk0 - edge}
\\
F_{K,\boldsymbol{\alpha}} ( J_{k,0} u )
&=
F_{K,\boldsymbol{\alpha}} ( u )
\qquad \forall K \in \mathcal{T}, \
\boldsymbol{\alpha} \in \mathbb{N}_{k-3}^2.
\label{SubEq:Functional preservation Jk0 - vol}
\end{align}
\end{subequations}
Following the proof of Lemma \ref{Lem:Non-Conf Approx Operator} \eqref{SubEq:Non-Conf Approx Operator - edge moment} and \eqref{SubEq:Non-Conf Approx Operator- vol moment} verbatim concludes the proof.
\end{proof}

A consequence of Lemma \ref{Lem:Properties conforming companion} is that the conforming companion operator $J_{k,0}$ is a right-inverse operator to $I_{k,0}^{\mathcal{T}, \CR}$, if applied to a $\CR_k$ function.

\begin{lemma} \label{Lem:Jkast right invers}
For every $\CR_k$ function $u \in \CR_{k,0} ( \mathcal{T} )$ we have that
\begin{align}
I_{k,0}^{\mathcal{T}, \CR} ( J_{k,0} u ) = u
\qquad \forall u \in \CR_{k,0} ( \mathcal{T} ).
\label{Eq:Jkast right inverse}
\end{align}
\end{lemma}

\begin{proof}
From \eqref{Eq:Functional preservation Jk0}
$\Pi_{k,0}^{\mathcal{E} ( \mathcal{T} )} ( J_{k,0} v )
= \Pi_{k,0}^{\mathcal{E} ( \mathcal{T} )} v
$
and 
$\Pi_{k,0}^{\mathcal{T}} ( J_{k,0} v )
=
\Pi_{k,0}^{\mathcal{T}} v
$ follows directly. We therefore obtain
\begin{align*}
I_{k,0}^{\mathcal{T}, \CR} ( J_{k,0} v )
&=
\Pi_{k,0}^{\mathcal{E} ( \mathcal{T} )} ( J_{k,0} v )
+
\Pi_{k,0}^{\mathcal{T}} ( J_{k,0}v - \Pi_{k,0}^{\mathcal{E} ( \mathcal{T} )} ( J_{k,0} v ) )
=
\Pi_{k,0}^{\mathcal{E} ( \mathcal{T} )} v
+
\Pi_{k,0}^{\mathcal{T}} ( v - \Pi_{k,0}^{\mathcal{E} ( \mathcal{T} )} ( v ) )
=
I_{k,0}^{\mathcal{T}, \CR} v.
\end{align*}
Since the operator $I_{k,0}^{\mathcal{T}, \CR}$ is a projection onto $\CR_{k,0} ( \mathcal{T} )$ (cf. Corollary \ref{Cor:Global Projection}), we conclude that $I_{k,0}^{\mathcal{T},\CR} ( J_{k,0} v ) = v$ for all $v \in \CR_{k,0} ( \mathcal{T} )$.
\end{proof}

We now have collected all ingredients to prove Lemma \ref{Lem:Upper bound by tangential jumps}.

\begin{proof}[Proof of Lemma \ref{Lem:Upper bound by tangential jumps}]
For any given $u \in \CR_{k,0} ( \mathcal{T} )$ consider $\widehat{u}^{\ast} \in \CR_{k,0} ( \widehat{\mathcal{T}} )$, given by 
\begin{align}
\widehat{u}^{\ast}
:=
I_{k,0}^{\widehat{\mathcal{T}}, \CR} ( J_{k,0} u ),
\label{Eq:Fine right inverse}
\end{align}
where $I_{k,0}^{\widehat{\mathcal{T}}, \CR} : H_{k,0}^{\CR} ( \widehat{\mathcal{T}} ) \to \CR_{k,0} ( \widehat{\mathcal{T}} )$ is the standard non-conforming quasi-interpolation operator on the fine mesh $\widehat{\mathcal{T}}$ given by \eqref{Eq:CR Interpolation}.

\textbf{@\labelcref{SubEq:Non-Conf - right inverse}:}
By construction, we know $\widehat{u}^{\ast}$ satisfies \eqref{SubEq:Non-Conf Approx Operator - edge moment} and \eqref{SubEq:Non-Conf Approx Operator- vol moment} on $ \widehat{\mathcal{T}}$. This together with \eqref{Eq:Jkast moment transport} yields 
\begin{align*}
\int_E q \widehat{u}^{\ast}
&= 
\sum_{\widehat{E} \in \operatorname{succ}( E )}
\int_{\widehat{E}} q \widehat{u}^{\ast}
=
\sum_{\widehat{E} \in \operatorname{succ}( E )}
\int_{\widehat{E}} q 
I_{k,0}^{\widehat{\mathcal{T}}, \CR} ( J_{k,0} u )
=
\sum_{\widehat{E} \in \operatorname{succ}( E )}
\int_{\widehat{E}} q  J_{k,0} u
=
\int_E q  J_{k,0} u
=
\int_E q u
\end{align*}
and
\begin{align*}
\int_K p \widehat{u}^{\ast}
&= 
\sum_{\widehat{K} \in \operatorname{succ}( K )}
\int_{\widehat{K}} p \widehat{u}^{\ast}
=
\sum_{\widehat{K} \in \operatorname{succ}( K )}
\int_{\widehat{K}} p
I_{k,0}^{\widehat{\mathcal{T}}, \CR} ( J_{k,0} u )
=
\sum_{\widehat{K} \in \operatorname{succ}( K )}
\int_{\widehat{K}} p  J_{k,0} u
=
\int_K p  J_{k,0} u
=
\int_K p u
\end{align*}
for any $K \in \mathcal{T}$, $E \in \mathcal{E}_{\Omega} ( \mathcal{T})$ and $p \in \mathbb{P}_{k-3} ( K )$, $q \in \mathbb{P}_{k-1} ( E )$.
From this, it follows that $F_{E,j} ( \widehat{u}^{\ast} ) = F_{E,j} ( u)$ for $E \in \mathcal{E}_{\Omega} ( \mathcal{T} )$ and $j \in \range{k-1}$ as well as $F_{K,\boldsymbol{\alpha}} ( \widehat{u}^{\ast} ) = F_{K,\boldsymbol{\alpha}} ( u )$ for all $K \in \mathcal{T}$ and $\boldsymbol{\alpha} \in \multirange{2}{k-3}$. Therefore we have that $\Pi_k^{\mathcal{E} ( \mathcal{T} )}  \widehat{u}^{\ast} 
=
\Pi_k^{\mathcal{E} ( \mathcal{T} )} u$ and $\Pi_{k,0}^{\mathcal{T}}  \widehat{u}^{\ast}
=
\Pi_{k,0}^{\mathcal{T}} u$.
This and the linearity of the operator $\Pi_{k,0}^{\mathcal{T}}$ reveals
\begin{align*}
I_k^{\mathcal{T}, \CR} \widehat{u}^{\ast}
&=
\Pi_k^{\mathcal{E} ( \mathcal{T} )} \widehat{u}^{\ast} 
+
\Pi_{k,0}^{\mathcal{T}} ( \widehat{u}^{\ast}
- 
\Pi_k^{\mathcal{E} ( \mathcal{T} )} \widehat{u}^{\ast})
=
\Pi_k^{\mathcal{E} ( \mathcal{T} )} u
+
\Pi_{k,0}^{\mathcal{T}} ( \widehat{u}^{\ast}
- 
\Pi_k^{\mathcal{E} ( \mathcal{T} )} u)
\\
&=
\Pi_k^{\mathcal{E} ( \mathcal{T} )} u
+
\Pi_{k,0}^{\mathcal{T}} \widehat{u}^{\ast}
-  
\Pi_{k,0}^{\mathcal{T}} (\Pi_k^{\mathcal{E} ( \mathcal{T} )} u)
=
\Pi_k^{\mathcal{E} ( \mathcal{T} )} u
+
\Pi_{k,0}^{\mathcal{T}} u
-  
\Pi_{k,0}^{\mathcal{T}} (\Pi_k^{\mathcal{E} ( \mathcal{T} )}  u )
\\
&=
\Pi_k^{\mathcal{E} ( \mathcal{T} )} u
+
\Pi_{k,0}^{\mathcal{T}} ( u - \Pi_k^{\mathcal{E} ( \mathcal{T} )} u )
=
I_k^{\mathcal{T},\CR} u
\overset{\text{Cor. } \ref{Cor:Global Projection}}{=}
u. 
\end{align*}

\textbf{@\labelcref{SubEq:Non-Conf -  error estimate}:} This property is proved by following the arguments in \cite[Sec. 3.5 \& 3.6]{CarstensenPuttkammer-howToProofDiscreteReliabiltiyNonConf} which require three more properties. 

First, we verify that for any $u \in \CR_{k,0} ( \mathcal{T} )$, that
\begin{align}
\Vert \nabla_{\widehat{\mathcal{T}}} \widehat{u}^{\ast} - \nabla_{\mathcal{T}} u \Vert_{\mathbf{L}^2 ( \Omega )}
& \lesssim 
\Vert \nabla_{\mathcal{T}} ( J_{k,0}^{\ast} u - u ) \Vert_{\mathbf{L}^2 ( \omega_{\mathcal{T} \setminus \widehat{\mathcal{T}}} )}.
\label{Eq:Localized companion estimate}
\end{align}
Note that \eqref{Eq:Localized companion estimate} corresponds to \cite[(3.3)]{CarstensenPuttkammer-howToProofDiscreteReliabiltiyNonConf}. Since the proof follows from elementary properties of the operator $I_{k,0}^{\widehat{\mathcal{T}},\CR}$ it is postponed to  Appendix \ref{A:Sec:Check prop}.

Second, for every triangle $K \in \mathcal{T}$ let $J_{k,0}^{K} : H_{k,0}^{\CR} ( \mathcal{T}) \vert_{\omega_{\mathcal{T}_K}} \to \mathbb{P}_k ( K )$ be the localized version of $J_{k,0}$. 
We need to verify that
\begin{align}
J_{k,0}^K ( w ) 
& =
w \vert_K
\qquad \forall w \in \left. (\CR_k ( \mathcal{T} ) \cap H_0^1 ( \Omega )) \right\vert_{\omega_{\mathcal{T}_K}}.
\tag{C6}
\label{Eq:C6}
\end{align}
Since $J_{k,0}$ restricted to $S_{k,0} ( \mathcal{T} )$ is the identity operator (cf. Lemma \ref{Lem:Properties conforming companion} \eqref{Eq:Jkast conforming identity}) and $(\CR_k ( \mathcal{T} ) \cap H_0^1 ( \Omega )) \vert_{\omega_{\mathcal{T}_K}} = S_{k,0} ( \mathcal{T} ) \vert_{\omega_{\mathcal{T}_K}}$, property \eqref{Eq:C6} is inherited.

Third, let $\varPsi_K : H_{k,0}^{\CR} ( \mathcal{T} ) \to \mathbb{R}_{\geq 0}$ be given by
\begin{align}
\varPsi_K ( u )
&:=
\sum_{E \in  \mathcal{E} (\mathcal{T}_K)}
h_E
\Vert \langle \llbracket \nabla_{\mathcal{T}} u \rrbracket_E , \mathbf{t}_E  \rangle \Vert_{L^2 ( E )}^2.
\label{Eq:Local tangetial jumps}
\end{align}
We need to verify the following implication:
\begin{align}
\forall w \in \left. \CR_k ( \mathcal{T} ) \right\vert_{\omega_{\mathcal{T}_K}}, \
\varPsi_K ( w ) = 0
\implies
w \in \left. \left(\CR_k ( \mathcal{T}_K ) \cap H_0^1 ( \Omega )\right) \right\vert_{\omega_{\mathcal{T}_K}}.
\tag{C7}
\label{Eq:C7}
\end{align}
The proof of \eqref{Eq:C7} is technical but follows from standard arguments and is therefore postponed to Appendix \ref{A:Sec:Check prop}. We are now able to apply \cite[Thm. 3.2]{CarstensenPuttkammer-howToProofDiscreteReliabiltiyNonConf} in the following form:

If $\widehat{u}^{\ast}$ from \eqref{Eq:Fine right inverse} satisfies \eqref{SubEq:Non-Conf - right inverse}, \eqref{Eq:Localized companion estimate}, \eqref{Eq:C6}, \eqref{Eq:C7} and $\widehat{\mathcal{T}}$ was obtained using NVB, then \eqref{SubEq:Non-Conf -  error estimate} holds true.
\end{proof}

\section{A mapping into intersections of Crouzeix--Raviart spaces}
\label{Sec:Unseen}

Next, we present the explicit construction of an operator $\widehat{P} : H_{k,0}^{\CR} ( \widehat{\mathcal{T}} ) \to \CR_{k,0} ( \widehat{\mathcal{T}} ) \cap \CR_{k,0} ( \mathcal{T} )$, where $\widehat{\mathcal{T}} \in \mathbb{T} (\mathcal{T})$ is some admissible refinement. 
Recall the sets $\mathcal{R} = \mathcal{T} \setminus \widehat{\mathcal{T}}$ and $\widehat{\mathcal{R}} = \widehat{\mathcal{T}} \setminus \mathcal{T}$ from \eqref{Eq:Refined elements}, and the numbering convention from \eqref{Eq:Layers}. Let $M_{k,0}^{\mathcal{R}} : H_{k,0}^{\CR} ( \mathcal{T} ) \to \CR_{k,0} ( \mathcal{T} )$ and $M_{k,0}^{\widehat{\mathcal{R}}}: H_{k,0}^{\CR} ( \widehat{\mathcal{T}} ) \to \CR_{k,0} ( \widehat{\mathcal{T}} )$ be as in \eqref{Eq:Mixed interpolation operator}. Our goal is to investigate the composition $M_{k,0}^{\mathcal{R}} \circ M_{k,0}^{\widehat{\mathcal{R}}}$. First we need to derive that the range of the fine partially conforming operator $M_{k,0}^{\widehat{\mathcal{R}}}$ belongs to the domain of the coarse partially-conforming operator $M_{k,0}^{\mathcal{R}}$.
  
\begin{lemma} \label{Lem:FineToCoarse - successive application}
It holds that
$
M_{k,0}^{\widehat{\mathcal{R}}} ( H_{k,0}^{\CR} ( \widehat{\mathcal{T}} ) )
\subseteq
H_{k,0}^{\CR} ( \mathcal{T} ).
$
\end{lemma}

\begin{proof}
Consider any $\widehat{u} \in H_{k,0}^{\CR} ( \widehat{\mathcal{T}} )$. We now prove that $(M_{k,0}^{\widehat{\mathcal{R}}} \widehat{u}) \vert_K \in H^1 ( K )$ for every coarse triangle $K \in \mathcal{T}$ and that $\llbracket  M_{k,0}^{\widehat{\mathcal{R}}} \widehat{u} \rrbracket_E \perp_{L^2 (E)} \mathbb{P}_{k-1} ( E )$ for all $E \in \mathcal{E} ( \mathcal{T} )$. To this end we consider some coarse triangle $K \in \mathcal{T}$ and let $\Gamma_{\mathcal{R}}$ be the interface between $\omega_{\mathcal{R}}$ and $\omega_{\mathcal{T} \setminus \mathcal{R}}$ as in \eqref{Eq:Gamma}. We distinguish between two cases, based on the location of $K$. 

\textbf{Case 1: $K \in \mathcal{T} \setminus \mathcal{R}$.} In this case we know that $K \in \mathcal{T} \cap \widehat{\mathcal{T}}$ and thus $( M_{k,0}^{\widehat{\mathcal{R}}} \widehat{u} ) \vert_K \in \mathbb{P}_k ( K ) \subseteq H^1 ( K )$. Furthermore, we observe that, for all edges $E \in \mathcal{E} ( K )$ we also have that $E \in \mathcal{E} ( \mathcal{T} ) \cap \mathcal{E} ( \widehat{\mathcal{T}} )$. Thus $\llbracket M_{k,0}^{\widehat{\mathcal{R}}} \widehat{u} \rrbracket_E \perp_{L^2(E)} \mathbb{P}_{k-1} ( E )$ holds for all $E \in\mathcal{E} ( \mathcal{T} \setminus \mathcal{R} )$ by construction.

\textbf{Case 2: $K \in \mathcal{R}$:} Recall that $\omega_{\mathcal{R}} = \omega_{\widehat{\mathcal{R}}}$ from \eqref{Eq:Refined domain}. We now know from Lemma \ref{Lem:Partially conforming} \eqref{SubEq:Mixed Operator properties - conforming on S} that $( M_{k,0}^{\widehat{\mathcal{R}}} \widehat{u} ) \vert_{\omega_{\widehat{\mathcal{R}}}} \in S_k ( \widehat{\mathcal{R}} ) \subseteq C^0 ( \omega_{\widehat{\mathcal{R}}}) $ . This directly implies that $ ( M_{k,0}^{\widehat{\mathcal{R}}} \widehat{u} ) \vert_K \in H^1 ( K )$. Furthermore the continuity of $( M_{k,0}^{\widehat{\mathcal{R}}} \widehat{u} ) \vert_{\omega_{\widehat{\mathcal{R}}}}$ also implies that 
$
\llbracket M_{k,0}^{\widehat{\mathcal{R}}} \widehat{u} \rrbracket_E
= 
0$ for all $E \in \mathcal{E}_{\Omega} ( \mathcal{R} ) \setminus \mathcal{E} ( \Gamma_{\mathcal{R}} )$. Since $\Gamma_{\mathcal{R}}$ is the interface between $\omega_{\mathcal{R}}$ and $\omega_{\mathcal{T} \setminus \mathcal{R}}$, it is clear that $\mathcal{E} ( \Gamma_{\mathcal{R}}) = \mathcal{E} ( \mathcal{R} ) \cap \mathcal{E} ( \mathcal{T} \setminus \mathcal{R})$. However, since $\mathcal{R} = \mathcal{T} \setminus \widehat{\mathcal{T}}$, $\mathcal{E} ( \Gamma_{\mathcal{R}} ) \subseteq \mathcal{E} ( \mathcal{T} ) \cap \mathcal{E} ( \widehat{\mathcal{T}} )$ follows. Therefore $\llbracket M_{k,0}^{\widehat{\mathcal{R}}} \widehat{u} \rrbracket_E
\perp_{L^2 (E)}
\mathbb{P}_{k-1} ( E )$ for all $E \in \mathcal{E} ( \Gamma_{\mathcal{R}} )$ as shown in \textbf{Case 1} and thus we have shown that $M_{k,0}^{\widehat{\mathcal{R}}} \widehat{u} \in H_k^{\CR} ( \mathcal{T} )$.

Finally, since $M_{k,0}^{\widehat{\mathcal{R}}} \widehat{u} \in \CR_{k,0} ( \widehat{\mathcal{T}} )$ by construction, it is easy to see, that for all coarse boundary edges $E \in \mathcal{E}_{\partial \Omega} ( \mathcal{T} )$, we have that
$
( M_{k,0}^{\widehat{\mathcal{R}}} \widehat{u}, q )_{L^2 ( E )}
=
\sum_{\widehat{E} \in \operatorname{succ} ( E )}
( M_{k,0}^{\widehat{\mathcal{R}}} \widehat{u}, q )_{L^2 ( \widehat{E} )}
=
0
$ for all $q \in \mathbb{P}_{k-1} ( E )$ and hence $ M_{k,0}^{\widehat{\mathcal{R}}} \widehat{u} \in H_{k,0}^{\CR} ( \mathcal{T} )$.
\end{proof}

\begin{corollary} \label{Cor:Well-definedness Phat}
Let $\widehat{\mathcal{T}} \in \mathbb{T} (\mathcal{T})$ be an admissible refinement of the coarse mesh $\mathcal{T}$. Then the operator $\widehat{P} : H_{k,0}^{\CR} ( \widehat{\mathcal{T}} ) \to \CR_{k,0} ( \mathcal{T})$, given by
\begin{align}
\widehat{P} &:=
M_{k,0}^{\mathcal{R}} 
\circ
M_{k,0}^{\widehat{\mathcal{R}}},
\label{Eq:Def Phat}
\end{align}
is well-defined and it holds that
\begin{align}
\widehat{P} ( H_{k,0}^{\CR}  ( \widehat{\mathcal{T}}) ) 
&\subseteq 
\CR_{k,0} ( \widehat{\mathcal{T}} )  
\cap 
\CR_{k,0} ( \mathcal{T} ).
\label{Eq:PHat image}
\end{align}
\end{corollary}

\begin{proof}
Lemma \ref{Lem:FineToCoarse - successive application} yields that $\widehat{P}  \widehat{u}   =
M_{k,0}^{\mathcal{R}}  ( M_{k,0}^{\widehat{\mathcal{R}}} \widehat{u} )$ is well-defined for all functions $\widehat{u} \in H_{k,0}^{\CR} ( \widehat{\mathcal{T}} )$.

\textbf{@\labelcref{Eq:PHat image}:} Consider any $\widehat{u} \in H_{k,0}^{\CR} ( \widehat{\mathcal{T}} )$. Since $\widehat{P} \widehat{u} \in \CR_{k,0} ( \mathcal{T} )$ by construction, we only need to verify that $\widehat{P} \widehat{u} \in \CR_{k,0} ( \widehat{\mathcal{T}} )$. 
To this end, consider some $\widehat{K} \in \widehat{\mathcal{T}}$. We distinguish two cases based on the location of $\widehat{K}$.

\textbf{Case 1: $\widehat{K} \in \widehat{\mathcal{T}} \setminus \widehat{\mathcal{R}}$.} The definition of the set of refined triangles $\widehat{\mathcal{R}}$, implies $\widehat{K} \in \widehat{\mathcal{T}} \cap \mathcal{T}$ and 
consequently, $( \widehat{P} \widehat{u} ) \vert_{\widehat{K}} \in \mathbb{P}_k ( \widehat{K} )$.
Furthermore, $\widehat{K} \in \mathcal{T} \cap \widehat{\mathcal{T}}$ also implies that that $\mathcal{E} ( \widehat{K} ) \subseteq \mathcal{E} ( \mathcal{T} )$ and 
thus $\llbracket \widehat{P} \widehat{u} \rrbracket_E
\perp_{L^2 (E)} 
\mathbb{P}_{k-1} ( E )$ for all edges $E \in \mathcal{E} ( \widehat{K} )$ 
since $\widehat{P} \widehat{u} \in \CR_{k,0} ( \mathcal{T} )$.

\textbf{Case 2: $\widehat{K} \in \widehat{\mathcal{R}}$.} Then there exists a predecessor triangle $K \in \mathcal{R}$ such that $\widehat{K} \in \operatorname{succ} ( K )$. Now since $\widehat{P} \widehat{u} \in \CR_{k,0} ( \mathcal{T} )$, we have that $( \widehat{P} \widehat{u} ) \vert_K \in \mathbb{P}_k ( K )$ so that $( \widehat{P} \widehat{u} ) \vert_{\widehat{K}} \in \mathbb{P}_k ( \widehat{K} )$. Next, we investigate the behavior across fine edges $\widehat{E} \in \mathcal{E} ( \widehat{K} )$ and consider three sub-cases.

\textit{Case 2.1: $\widehat{E} \in \mathcal{E}_{\Omega} ( \widehat{K} )$ and $\widehat{E} \in \mathcal{E} ( \Gamma_{\mathcal{R}} )$.} This case is already covered by {Case 1} since $\mathcal{E} ( \Gamma_{\mathcal{R}} ) \subseteq \mathcal{E} ( \widehat{\mathcal{T}} \setminus \widehat{\mathcal{R}} )$.

\textit{Case 2.2: $\widehat{E} \in \mathcal{E}_{\Omega} ( \widehat{K} ) \setminus \mathcal{E} ( \Gamma_{\mathcal{R}} )$.} This implies that $\widehat{E} \subseteq \omega_{\widehat{\mathcal{R}}}$ so that either $\widehat{E} \cap \partial \omega_{\widehat{\mathcal{R}}} \in \mathcal{V} ( \widehat{\mathcal{T}} )$ or $\widehat{E} \cap \partial \omega_{\widehat{\mathcal{R}}} = \emptyset$. Lemma \ref{Lem:Partially conforming} \eqref{SubEq:Mixed Operator properties - conforming on S} implies $( \widehat{P} \widehat{u} ) \vert_{\omega_{\mathcal{R}}} 
\in 
S_k ( \mathcal{R} ) 
\subseteq 
C^0 ( \omega_{\mathcal{R}} )$  and $\llbracket \widehat{P} \widehat{u} \rrbracket_{\widehat{E}} = 0$ follows. 

\textit{Case 2.3: $\widehat{E} \in \mathcal{E}_{\partial \Omega} ( \widehat{K} )$.} This implies the existence of a coarse predecessor edge $E \in \mathcal{E}_{\partial \Omega} ( \mathcal{T} )$. Consequently the predecessor triangle $K \in \mathcal{T}$ of $\widehat{K}$ intersects with the boundary and we recall that
\begin{align*}
( \widehat{P} \widehat{u} ) \vert_K
\overset{\eqref{Eq:Def Phat}}{=}
\sum_{\mathbf{z} \in \mathcal{V}_{\Omega} ( K ) }
F_{\mathcal{R},\mathbf{z}}  ( M_{k,0}^{\widehat{\mathcal{R}}} \widehat{u} )
\psi_{\mathcal{R},k}^{\mathbf{z}}
+
\overset{\bullet}{\Pi}_{k,0} ( M_{k,0}^{\widehat{\mathcal{R}}} \widehat{u})
\end{align*}
By construction $\overset{\bullet}{\Pi}_{k,0} : H_{k,0}^{\CR} ( \mathcal{T} ) \to \overset{\bullet}{S}_{k,0} ( \mathcal{T} )$ so that $( \overset{\bullet}{\Pi}_{k,0} ( M_{k,0}^{\widehat{\mathcal{R}}} \widehat{u}) ) \vert_E = 0$ which in turn implies that $( \overset{\bullet}{\Pi}_{k,0} ( M_{k,0}^{\widehat{\mathcal{R}}} \widehat{u}) ) \vert_{\widehat{E}} = 0$. Since $\psi_{\mathcal{R},k}^{\mathbf{z}} \vert_{\omega_{\mathcal{R}}} \in S_k ( \mathcal{R} )$ and $\mathbf{z} \in \mathcal{V}_{\Omega} ( K )$, Proposition \ref{Prop:S-conforming prop} \eqref{Eq:S-conforming prop} yields that $\psi_{\mathcal{R},k}^{\mathbf{z}} \vert_E = 0$ and thus $\psi_{\mathcal{R},k}^{\mathbf{z}} \vert_{\widehat{E}} = 0$ for all $\mathbf{z} \in \mathcal{V}_{\Omega} ( K )$. Hence we have shown that $( \widehat{P} \widehat{u} ) \vert_{\widehat{E}} = 0$.

In summary, we have shown that for all refined triangles $\widehat{K} \in \widehat{\mathcal{T}}$, $( \widehat{P} \widehat{v} ) \vert_{\widehat{K}} \in \mathbb{P}_k ( \widehat{K} )$ and for all refined edges $\widehat{E} \in \mathcal{E} ( \widehat{\mathcal{T}} )$ the $\CR_k$ jump and boundary conditions are fulfilled, i.e., $\widehat{P} \widehat{u} \in \CR_{k,0}  ( \widehat{\mathcal{T}} )$.
\end{proof}

Next, we prove that the operator $\widehat{P}$ is locally bounded on the coarse mesh $\mathcal{T}$.
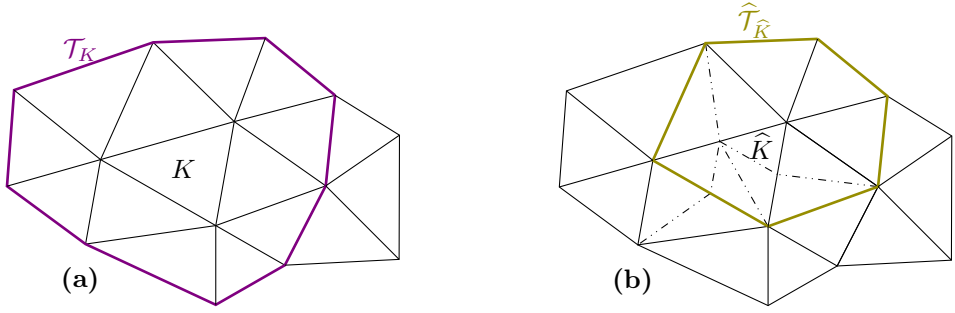
\begin{figure}
\centering
\begin{tikzpicture}[scale=.7]
\coordinate (zero) at (0,0);
\coordinate (ref) at (1,0);
\coordinate (z2) at ($(zero)!2!80:(ref)$);
\coordinate (z3) at ($(zero)!2.5!150:(ref)$);
\coordinate (z4) at ($(zero)!2.2!20:(ref)$);
\coordinate (z5) at ($(z2)!0.93!50:(z4)$);
\coordinate (z6) at ($(z3)!0.93!50:(z2)$);
\coordinate (z7) at ($(z3)!0.65!290:(zero)$);
\coordinate (z8) at ($(zero)!1.5!330:(ref)$);

\coordinate (z2mid) at ($(z2)!0.5!0:(z3)$);
\coordinate (z3mid) at ($(z3)!0.5!0:(zero)$);
\coordinate (z0mid) at ($(zero)!0.5!0:(z4)$);
\coordinate (z4mid) at ($(z4)!0.5!0:(z2)$);
\coordinate (mid) at ($(zero)!0.5!0:(z2)$);

\coordinate (z41) at ($(z4)!0.93!170:(z2)$);
\coordinate (z42) at ($(z4)!.8!250:(z2)$);
\coordinate (z21) at ($(z2)!.8!105:(z4)$);

\coordinate (z31) at ($(z3)!.8!125:(z2)$);
\coordinate (z32) at ($(z3)!.7!180:(z2)$);

\coordinate (zero1) at ($(zero)!1.5!270:(ref)$);

\draw  (z2)--(z3) -- (zero);
\draw (zero) -- (z8) -- (z4) -- (z5) -- (z2);
\draw   (z3) -- (z7) -- (zero) --(z4) -- (z2) -- (z6) -- (z3);
\draw[line width=1,violet]  (z5) -- (z21) -- (z6) -- (z31) -- (z32) -- (z7) -- (zero1) -- (z8) -- (z4) --(z5);
\draw (z8) -- (z41) -- (z42) -- (z5);

\draw (z2) -- (zero);
\draw (z4) -- (z41);
\draw (z4) -- (z42);
\draw (z2) -- (z21);
\draw (z3) -- (z31);
\draw (z3) -- (z32);
\draw (zero) -- (zero1);

\node at (barycentric cs:z31=1,z6=1,z3=-.3){$\mathcolor{violet}{\mathcal{T}_K}$};
\node at (barycentric cs:zero=1,z2=1,z3=1) {$K$};

\node at (barycentric cs:z32=1,zero1=1,z3mid=-.8) {\textbf{(a)}};
\end{tikzpicture}
\hfil
\begin{tikzpicture}[scale=.7]
\coordinate (zero) at (0,0);
\coordinate (ref) at (1,0);
\coordinate (z2) at ($(zero)!2!80:(ref)$);
\coordinate (z3) at ($(zero)!2.5!150:(ref)$);
\coordinate (z4) at ($(zero)!2.2!20:(ref)$);
\coordinate (z5) at ($(z2)!0.93!50:(z4)$);
\coordinate (z6) at ($(z3)!0.93!50:(z2)$);
\coordinate (z7) at ($(z3)!0.65!290:(zero)$);
\coordinate (z8) at ($(zero)!1.5!330:(ref)$);

\coordinate (z2mid) at ($(z2)!0.5!0:(z3)$);
\coordinate (z3mid) at ($(z3)!0.5!0:(zero)$);
\coordinate (z0mid) at ($(zero)!0.5!0:(z4)$);
\coordinate (z4mid) at ($(z4)!0.5!0:(z2)$);
\coordinate (mid) at ($(zero)!0.5!0:(z2)$);

\coordinate (z41) at ($(z4)!0.93!170:(z2)$);
\coordinate (z42) at ($(z4)!.8!250:(z2)$);
\coordinate (z21) at ($(z2)!.8!105:(z4)$);

\coordinate (z31) at ($(z3)!.8!125:(z2)$);
\coordinate (z32) at ($(z3)!.7!180:(z2)$);

\coordinate (zero1) at ($(zero)!1.5!270:(ref)$);

\draw  (z4) -- (z2) -- (z3);
\draw (zero) -- (z8) -- (z4);
\draw (z5) -- (z2);
\draw   (z3) -- (z7) -- (zero);
\draw (z4) -- (z2) -- (z6) ;
\draw (z6) -- (z31) -- (z32) -- (z7) -- (zero1) -- (z8) -- (z4);
\draw[dash dot dot] (z6) -- (z2mid);
\draw[dash dot dot] (z3mid) -- (z7);
\draw[dash dot dot] (z2mid) --(mid);
\draw[dash dot dot] (zero) -- (z2mid) -- (z3mid);
\draw[dash dot dot] (z4) -- (mid);
\draw (z8) -- (z41) -- (z42) -- (z5);
\draw[line width=1,olive] (z21) -- (z6) -- (z3) -- (zero) -- (z4) -- (z5) -- (z21);

\draw (z2) -- (zero);
\draw (z4) -- (z41);
\draw (z4) -- (z42);
\draw (z2) -- (z21);
\draw (z3) -- (z31);
\draw (z3) -- (z32);
\draw (zero) -- (zero1);

\node at (barycentric cs:z21=1,z6=1,z2=-.4){$\mathcolor{olive}{\widehat{\mathcal{T}}_{\widehat{K}}}$};
\node at (barycentric cs:z2mid=1,z2=1,mid=1) {$\widehat{K}$};

\node at (barycentric cs:z32=1,zero1=1,z3mid=-.8) {\textbf{(b)}};
\end{tikzpicture}
\caption{\textbf{(a):} We consider a coarse mesh $\mathcal{T}$ with a triangle $K$ and its patch $\mathcal{T}_K$ marked by the purple line. \textbf{(b):} We consider a refinement $\widehat{\mathcal{T}}$ of the coarse mesh, with the newly added edges marked by the $- \cdot \cdot$ lines. A simplex $\widehat{K} \in \operatorname{succ} ( K )$ is chosen and its fine simplex patch $\widehat{\mathcal{T}}_{\widehat{K}}$ is included by the olive line.}
\label{Fig:Refined zones}
\end{figure}

\begin{lemma} \label{Lem:Prop Phat}
For all $\widehat{u} \in H_{k,0}^{\CR} ( \widehat{\mathcal{T}} )$  it holds that
\begin{align}
\Vert \nabla \widehat{P} \widehat{u} \Vert_{\mathbf{L}^2 ( K )}
& \lesssim 
\Vert \nabla_{\widehat{\mathcal{T}}} \widehat{u} \Vert_{\mathbf{L}^2 ( Z_K )}
\qquad 
\forall 
K \in \mathcal{T},
\label{SubEq:Prop Phat - stability}
\end{align}
where $Z_K$ is given by
\begin{align}
Z_K =
\begin{cases}
\omega_{\mathcal{T}_K^1}
& K \in \mathcal{R}^{1/2},
\\
K & K \in \mathcal{T} \setminus \mathcal{R}^{1/2}.
\end{cases}
\label{Eq:Double layer simplex patch}
\end{align}
\end{lemma}

\begin{proof}
\textbf{@\labelcref{SubEq:Prop Phat - stability}:} Let $K \in \mathcal{T}$ be some coarse triangle and consider any $\widehat{u} \in H_{k,0}^{\CR} ( \widehat{\mathcal{T}} )$. Since $M_{k,0}^{\CR} \widehat{v} \in H_{k,0}^{\CR} ( \mathcal{T} )$ for all $\widehat{v} \in H_{k,0}^{\CR} ( \widehat{\mathcal{T}} )$ (see Lemma \ref{Lem:FineToCoarse - successive application}), we may apply \eqref{SubEq:Mixed Operator properties - loc bounded} on the coarse level and obtain that
\begin{align}
\Vert \nabla \widehat{P} \widehat{u} \Vert_{\mathbf{L}^2 ( K )}
=
\Vert \nabla M_{k,0}^{\mathcal{R}}  ( M_{k,0}^{\widehat{\mathcal{R}}} \widehat{u} ) \Vert_{\mathbf{L}^2 ( K )}
\overset{\eqref{SubEq:Mixed Operator properties - loc bounded}}{\lesssim}
\Vert \nabla_{\widehat{\mathcal{T}}}  M_{k,0}^{\widehat{\mathcal{R}}} \widehat{u} \Vert_{\mathbf{L}^2 ( D_K )}.
\label{Eq:Local continuity Phat - coarse estimate}
\end{align}
If $K \in \mathcal{T} \setminus \mathcal{R}^{1/2}$ we know from \eqref{Eq:L2 stability region} that $D_K = K$. Therefore, we repeat the same argument for $\Vert \nabla_{\widehat{\mathcal{T}}}  M_{k,0}^{\widehat{\mathcal{R}}} \widehat{u} \Vert_{\mathbf{L}^2 ( D_K )} = 
\Vert \nabla_{\widehat{\mathcal{T}}}  M_{k,0}^{\widehat{\mathcal{R}}} \widehat{u} \Vert_{\mathbf{L}^2 ( K )}$ and obtain that $\Vert \nabla \widehat{P} \widehat{u} \Vert_{\mathbf{L}^2 ( K )} \lesssim \Vert \nabla_{\widehat{\mathcal{T}}}  \widehat{u} \Vert_{\mathbf{L}^2 ( K )}$.

If $K \in \mathcal{R}^{1/2}$, we have $D_K = \omega_{\mathcal{T}_K}$ from \eqref{Eq:L2 stability region}. Arguing as in \eqref{Eq:Local continuity Phat - coarse estimate} for $\Vert \nabla_{\widehat{\mathcal{T}}}  M_{k,0}^{\widehat{\mathcal{R}}} \widehat{u} \Vert_{\mathbf{L}^2 ( D_K )} = \Vert \nabla_{\widehat{\mathcal{T}}}  M_{k,0}^{\widehat{\mathcal{R}}} \widehat{u} \Vert_{\mathbf{L}^2 ( \omega_{\mathcal{T}_K} )}$, we see that
\begin{align*}
\Vert \nabla_{\widehat{\mathcal{T}}}  M_{k,0}^{\widehat{\mathcal{R}}} \widehat{u} \Vert_{\mathbf{L}^2 ( Z_K )}^2
&=
\sum_{K \in \mathcal{T}_K}
\sum_{\widehat{K} \in \operatorname{succ} ( K )}
\Vert \nabla M_{k,0}^{\widehat{\mathcal{R}}} \widehat{u} \Vert_{\mathbf{L}^2 ( \widehat{K} )}^2
\lesssim
\sum_{K \in \mathcal{T}_K}
\sum_{\widehat{K} \in \operatorname{succ} ( K )}
\Vert \nabla_{\widehat{\mathcal{T}}}  \widehat{u} \Vert_{\mathbf{L}^2 ( \omega_{\widehat{\mathcal{T}}_{\widehat{K}}} )}^2.
\end{align*}
We use that for all $\widehat{K} \in \operatorname{succ} ( K^{\prime} )$ where $K^{\prime} \in \mathcal{T}_K$ the inclusion $\omega_{\widehat{\mathcal{T}}_{\widehat{\mathcal{T}}}} \subseteq \omega_{\mathcal{T}_{K^{\prime}}}$ holds. This is illustrated by the meshes \textbf{(a)} and \textbf{(b)} in Figure \ref{Fig:Refined zones}. Consequently
\begin{align}
\bigcup_{K^{\prime} \in \mathcal{T}_K} 
\bigcup_{\widehat{K} \in \operatorname{succ} ( K^{\prime} )}
\omega_{\widehat{\mathcal{T}}_{\widehat{K}}}
&\subseteq
\bigcup_{K^{\prime} \in \mathcal{T}_K} 
\omega_{\mathcal{T}_{K^{\prime}}}
= 
\omega_{\mathcal{T}_K^1}
=
Z_K
\label{Eq:Fine inclusion TK1}
\end{align}
holds true. Combining the finite overlap property \eqref{Eq:Finite overlap and minimal angle constants} with the inclusion in \eqref{Eq:Fine inclusion TK1}, we deduce that
\begin{align*}
\sum_{K \in \mathcal{T}_K}
\sum_{\widehat{K} \in \operatorname{succ} ( K )}
\Vert \nabla_{\widehat{\mathcal{T}}}  \widehat{u} \Vert_{\mathbf{L}^2 ( \omega_{\widehat{\mathcal{T}}_{\widehat{K}}} )}^2
&\lesssim
\Vert \nabla_{\widehat{\mathcal{T}}} \widehat{u} \Vert_{\mathbf{L}^2 ( Z_K)}^2.
\end{align*}
By inserting this and the previous estimate into \eqref{Eq:Local continuity Phat - coarse estimate} we conclude that $ \Vert \nabla \widehat{P} \widehat{u} \Vert_{\mathbf{L}^2 ( K )} \lesssim \Vert \nabla_{\widehat{\mathcal{T}}} \widehat{u} \Vert_{\mathbf{L}^2 ( Z_K)}$.
\end{proof}

%
%

\begin{proof}[Proof of Lemma \ref{Lem:Existence Phat}]
Consider any $\widehat{u} \in \CR_{k,0} ( \widehat{\mathcal{T}} )$.

\textbf{@\labelcref{SubEq:Conditions Phat - Identiy}:} From $\mathcal{T} \setminus \mathcal{R}^{1/2} = \widehat{\mathcal{T}} \setminus\widehat{\mathcal{R}}^{1/2}$ and Proposition \ref{Prop:S-conforming prop} it follows that $( M_{k,0}^{\widehat{\mathcal{R}}}  \widehat{u} ) \vert_K = \widehat{u} \vert_K$. Repeating the same arguments on the coarse level for $M_{k,0}^{\mathcal{R}}$ and $M_{k,0}^{\widehat{\mathcal{R}}}  \widehat{v}$ reveals that
$
(\widehat{P} \widehat{u} ) \vert_K
=
M_{k,0}^{\mathcal{R}} 
(
M_{k,0}^{\widehat{\mathcal{R}}}  \widehat{u} )   \vert_K
=
( M_{k,0}^{\widehat{\mathcal{R}}}  \widehat{u} ) \vert_K
=
\widehat{u} \vert_K
$ 
for all triangles $K \in \mathcal{T} \setminus \mathcal{R}^{1/2}$. 

\textbf{@\labelcref{SubEq:Conditions Phat - approx propterties vol}:} The proof of this property is based on the approximation properties of both $M_{k,0}^{\mathcal{R}}$ and $M_{k,0}^{\widehat{\mathcal{R}}}$. We observe that
\begin{align*}
\Vert ( \operatorname{Id} - \widehat{P} ) \widehat{u} \Vert_{L^2 ( K )}
\leq
\Vert ( \operatorname{Id} - M_{k,0}^{\widehat{\mathcal{R}}}) \widehat{u} \Vert_{L^{2}(  K )}
+
\Vert ( \operatorname{Id} - M_{k,0}^{\mathcal{R}} ) M_{k,0}^{\widehat{\mathcal{R}}} \widehat{u} \Vert_{L^{2}( K )}.
\end{align*}
Applying Lemma \ref{Lem:Partially conforming} \eqref{SubEq:Mixed Operator properties - approx} and the finite overlap property \eqref{Eq:Finite overlap and minimal angle constants} to $\Vert ( \operatorname{Id} - M_{k,0}^{\widehat{\mathcal{R}}}) \widehat{u} \Vert_{L^{2}(  K )}$ reveals
\begin{align*}
\Vert ( \operatorname{Id} - M_{k,0}^{\widehat{\mathcal{R}}}) \widehat{u} \Vert_{L^{2}(  K )}^{2}
&=
\sum_{\widehat{K} \in \operatorname{succ} ( K )}
\Vert ( \operatorname{Id} - M_{k,0}^{\widehat{\mathcal{R}}}) \widehat{u} \Vert_{L^{2}(  \widehat{K} )}
\overset{\eqref{SubEq:Mixed Operator properties - approx}}{\lesssim}
\sum_{\widehat{K} \in \operatorname{succ} ( K )}
h_{\widehat{K}}^2
\Vert \nabla_{\widehat{\mathcal{T}}} \widehat{u} \Vert_{\mathbf{L}^2 ( D_{\widehat{K}} )}^2
\\
&\overset{\eqref{Eq:Finite overlap and minimal angle constants}}{\lesssim}
h_K^2 \Vert \nabla_{\widehat{\mathcal{T}}} \widehat{u} \Vert_{\mathbf{L}^2 ( D_K )}^2 
\overset{D_K \subseteq Z_K}{\leq} h_K^2 \Vert \nabla_{\widehat{\mathcal{T}}} \widehat{u} \Vert_{\mathbf{L}^2 ( Z_K )}^2 .
\end{align*}
The same arguments together with the local boundedness of $M_{k,0}^{\widehat{\mathcal{R}}}$ leads to
\begin{align*}
\Vert ( \operatorname{Id} - M_{k,0}^{\mathcal{R}} ) M_{k,0}^{\widehat{\mathcal{R}}} \widehat{u} \Vert_{L^{2}( K )}
&
\overset{\eqref{SubEq:Mixed Operator properties - approx}}{\lesssim}
h_K \Vert \nabla_{\mathcal{T}} M_{k,0}^{\widehat{\mathcal{R}}} \widehat{u} \Vert_{L^2 ( D_K )}
\overset{\eqref{SubEq:Mixed Operator properties - loc bounded} , \eqref{Eq:Fine inclusion TK1}}{\lesssim}
h_K \Vert \nabla_{\widehat{\mathcal{T}}} \widehat{u} \Vert_{L^2 ( Z_K )}.
\end{align*}


\textbf{@\labelcref{SubEq:Conditions Phat - approx propterties edge}:} 
As for \eqref{SubEq:Conditions Phat - approx propterties vol} observe that
\begin{align}
\Vert  \llbrace ( 1 - \widehat{P} ) \widehat{u} \rrbrace_E \Vert_{L^2 ( E)}
& \leq
\Vert \llbrace( 1 - M_{k,0}^{\widehat{\mathcal{R}}} ) \widehat{u} \rrbrace_E \Vert_{L^2 ( E)}
+
\Vert \llbrace (1 - M_{k,0}^{\mathcal{R}} ) M_{k,0}^{\widehat{\mathcal{R}}} \widehat{u} \rrbrace_E \Vert_{L^2 ( E)}
\label{Eq:Phat approx - edge estimate}
\end{align}
by using the definition of $\widehat{P}$. The first term in \eqref{Eq:Phat approx - edge estimate} is estimated through the application of a scaled trace inequalities in a standard way on the fine mesh $\widehat{\mathcal{T}}$ and the local approximation properties of $M_{k,0}^{\widehat{\mathcal{R}}}$. The second summand is treated similarly.
\end{proof}

Combining Lemmata \ref{Lem:Upper bound by tangential jumps} and \ref{Lem:Existence Phat} allows us to estimate the term $a_{\widehat{\mathcal{T}}} ( \widehat{u}_h - u_h, \widehat{u}_h - \widehat{u}_h^{\ast})$ from the proof of the main theorem as stated in Lemma \ref{Lem:Conformity estimation}.

\begin{proof}[Proof of Lemma \ref{Lem:Conformity estimation}]
For a given right-hand side function $f \in L^2 ( \Omega )$, let $u_h \in \CR_{k,0} ( \mathcal{T} )$ and $\widehat{u}_h \in \CR_{k,0} ( \mathcal{T} )$ be the coarse and fine $\CR_k$ solutions to \eqref{Eq:DiscretePoisson}. Furthermore let $\widehat{u}_h^{\ast} := I_{k,0}^{\widehat{\mathcal{T}}, \CR} ( J_{k,0} u_h ) \in \CR_{k,0} ( \widehat{\mathcal{T}} )$ be given by Lemma \ref{Lem:Upper bound by tangential jumps}. We denote the discrete error by $\widehat{e}_h :=  \widehat{u}_h - u_h$ and let $\widehat{v}_h \in \CR_{k,0} ( \widehat{\mathcal{T}} )$ be given by $\widehat{v}_h := \widehat{u}_h - \widehat{u}_h^{\ast}$, so that $ a_{\widehat{\mathcal{T}}} ( \widehat{u}_h - u_h, \widehat{u}_h - \widehat{v}_h^{\ast} ) = a_{\widehat{\mathcal{T}}} ( \widehat{e}_h , \widehat{v}_h^{\ast} )$. Recall that $\widehat{P} \widehat{v}_h \in \CR_{k,0} ( \mathcal{T} ) \cap \CR_{k,0} ( \widehat{\mathcal{T}} )$ (cf., \eqref{Eq:PHat image}). Then the solution properties of both the coarse solution $u_h$ and the fine solution $\widehat{u}_h$, imply Galerkin orthogonality $a_{\widehat{\mathcal{T}}} ( \widehat{e}_h, \widehat{P} \widehat{v}_h ) =0$ and therefore $a_{\widehat{\mathcal{T}}}  ( \widehat{e}_h ,  \widehat{v}_h^{\ast} )
=
a_{\widehat{\mathcal{T}}}  ( \widehat{e}_h , (1 - \widehat{P} ) \widehat{v}_h^{\ast} )$. Since $\widehat{u}_h$ solves \eqref{Eq:DiscretePoisson} for all $\widehat{v} \in \CR_{k,0} ( \widehat{\mathcal{T}} )$ and $( 1 - \widehat{P} ) \widehat{v}_h^{\ast} \in \CR_{k,0} ( \widehat{\mathcal{T}} )$ by construction, we employ integration by parts and the $\CR_k$ jump conditions over all fine edges $\widehat{E} \in \mathcal{E} ( \widehat{\mathcal{T}} )$ which results in
\begin{align*}
a_{\widehat{\mathcal{T}}}  ( \widehat{e}_h , (1 - \widehat{P} ) \widehat{v}_h^{\ast} )
=&
\sum_{\widehat{K} \in \widehat{\mathcal{T}}}
\int_{\widehat{K}}
( f + \Delta u_h )
( ( 1 - \widehat{P} ) \widehat{v}_h^{\ast} ) 
+
\sum_{\widehat{E} \in \mathcal{E} ( \widehat{\mathcal{T}} )} 
\int_{\widehat{E}}
 \llbracket \langle \nabla_{\widehat{\mathcal{T}}} u_h , \mathbf{n}_{\widehat{E}} \rangle 
( ( 1 -\widehat{P} ) \widehat{v}_h^{\ast} )  \rrbracket_{\widehat{E}}
\\
\overset{\eqref{Eq:Jump of products}}{=} &
\sum_{\widehat{K} \in \widehat{\mathcal{T}}}
\int_{\widehat{K}}
( f + \Delta u_h )
( ( 1 - \widehat{P} ) \widehat{v}_h^{\ast}) 
+
\sum_{\widehat{E} \in \mathcal{E} ( \widehat{\mathcal{T}} )} 
\int_{\widehat{E}}
\langle \llbracket \nabla_{\widehat{\mathcal{T}}} u_h \rrbracket_{\widehat{E}} , \mathbf{n}_{\widehat{E}} \rangle
\llbrace ( 1 -\widehat{P} ) \widehat{v}_h^{\ast} \rrbrace_{\widehat{E}}.
\end{align*}
Recall $(( 1 - \widehat{P} ) \widehat{v}_h^{\ast}) \vert_K = 0$ for all $K \in \mathcal{T} \setminus \mathcal{R}^{1/2}$ from \eqref{SubEq:Conditions Phat - Identiy} which implies that $( 1 - \widehat{P} ) \widehat{v}_h^{\ast}  \in \CR_{k,0} ( \mathcal{R}^{1/2} )$, we write
\begin{align*}
S_{\operatorname{vol}} := 
\sum_{\widehat{K} \in \widehat{\mathcal{R}}^{1/2}}
\int_{\widehat{K}}
( f + \Delta u_h )
( ( 1 - \widehat{P} ) \widehat{v}_h^{\ast} )
& \quad \text{and} \quad
S_{\operatorname{edge}} 
:= 
\sum_{\widehat{E} \in \mathcal{E}_{\Omega} ( \widehat{\mathcal{R}} )} 
\int_{\widehat{E}}
\langle \llbracket \nabla_{\widehat{\mathcal{T}}} u_h \rrbracket_{\widehat{E}} , \mathbf{n}_{\widehat{E}} \rangle
\llbrace ( 1 -\widehat{P} ) \widehat{v}_h^{\ast} \rrbrace_{\widehat{E}}
\end{align*}
and observe that the previous identity can be written as
\begin{align}
a_{\widehat{\mathcal{T}}}  ( \widehat{e}_h , (1 - \widehat{P} ) \widehat{v}_h^{\ast} )
&=
S_{\operatorname{vol}} + S_{\operatorname{edge}}.
\label{Eq:Svol + Sedge}
\end{align}

First $S_{\operatorname{vol}}$ is considered. Rearranging the sum and applying the Cauchy-Schwarz inequality and the volume approximation properties \eqref{SubEq:Conditions Phat - approx propterties vol} of $\widehat{P}$, we observe that
\begin{align}
S_{\operatorname{vol}}
&=
\sum_{K \in \mathcal{R}^{1/2}} 
\int_{K}
( f + \Delta u_h )
( ( 1 - \widehat{P} ) \widehat{v}_h^{\ast} )
\nonumber
\leq
\sum_{K \in \mathcal{R}^{1/2}} 
\Vert f + \Delta u_h \Vert_{L^2 ( K )}
\Vert ( 1 - \widehat{P} ) \widehat{v}_h^{\ast} \Vert_{L^2 ( K )}
\\
&\lesssim
\sum_{K \in \mathcal{R}^{1/2}} 
h_K
\Vert f + \Delta u_h \Vert_{L^2 ( K )}
\Vert \nabla_{\widehat{\mathcal{T}}} \widehat{v}_h^{\ast} \Vert_{\mathbf{L}^2 ( \omega_{\mathcal{T}_K^1} )}.
\label{Eq:Svol estimate}
\end{align}

For $S_{\operatorname{edge}}$, we first observe that for all edges $\widehat{E} \in \mathcal{E}_{\Omega} ( \widehat{\mathcal{R}} )$ which lie in the interior of some coarse triangle $K \in \mathcal{T}$, the coarse solution $u_h$ is continuous over $\widehat{E}$ since $u_h \vert_K \in \mathbb{P}_k ( K )$. Consequently $\llbracket u_h \rrbracket_{\widehat{E}} = 0$ and we rewrite $S_{\operatorname{edge}}$ as
\begin{align*}
S_{\operatorname{edge}}
&=
\sum_{E \in \mathcal{E}_{\Omega} ( \mathcal{R} )}
\sum_{\widehat{E} \in \operatorname{succ} ( E )}
\int_{\widehat{E}}
\langle \llbracket \nabla_{\mathcal{T}} u_h \rrbracket_{\widehat{E}} , \mathbf{n}_{\widehat{E}} \rangle
\llbrace ( 1 -\widehat{P} ) \widehat{v}_h^{\ast} \rrbrace_{\widehat{E}}
\nonumber
=
\sum_{E \in \mathcal{E}_{\Omega} ( \mathcal{R} )}
\int_E \langle \llbracket \nabla_{\mathcal{T}} u_h \rrbracket_E , \mathbf{n}_E \rangle
\llbrace ( 1 -\widehat{P} ) \widehat{v}_h^{\ast} \rrbrace_E.
\end{align*}
Applying the Cauchy-Schwarz inequality and the edge interpolation properties \eqref{SubEq:Conditions Phat - approx propterties edge} of $\widehat{P}$, we obtain that
\begin{align}
S_{\operatorname{edge}}
&\leq
\sum_{E \in \mathcal{E}_{\Omega} ( \mathcal{R} )}
\Vert \langle \llbracket \nabla_{\mathcal{T}} u_h \rrbracket_E , \mathbf{n}_E\rangle \Vert_{L^2 ( E )}
\Vert \llbrace ( 1 - \widehat{P} ) \widehat{v}_h^{\ast} \rrbrace_E \Vert_{L^2 ( E)}
\nonumber
\\ 
&\lesssim
\sum_{E \in \mathcal{E}_{\Omega} ( \mathcal{R} )}
\sum_{K^{\prime} \in \mathcal{T}_E}
h_{K^{\prime}}^{1/2}
\Vert \langle \llbracket \nabla_{\mathcal{T}} u_h \rrbracket_E, \mathbf{n}_E \rangle \Vert_{L^2 ( E )}
\Vert \nabla_{\widehat{\mathcal{T}}} \widehat{v}_h^{\ast} \Vert_{\mathbf{L}^2 ( \omega_{\mathcal{T}_{K^{\prime}}^1} )}.
\label{Eq:Sedge complete estimate}
\end{align}
By combining \eqref{Eq:Svol estimate} -- \eqref{Eq:Sedge complete estimate} and inserting them into \eqref{Eq:Svol + Sedge}, we find that $a_{\widehat{\mathcal{T}}}  ( \widehat{e}_h , (1 - \widehat{P} ) \widehat{v}_h^{\ast} )$ is estimated by
\begin{align*}
a_{\widehat{\mathcal{T}}} ( \widehat{e}_h, ( 1 - \widehat{P}) \widehat{v}_h^{\ast} )
\lesssim &
\sum_{K \in \mathcal{R}^{1/2}}  
h_K
\Vert f + \Delta u_h \Vert_{L^2 ( K )}
\Vert \nabla_{\widehat{\mathcal{T}}} \widehat{v}_h^{\ast} \Vert_{\mathbf{L}^2 ( \omega_{\mathcal{T}_K^1} )}
+ 
\\
&+
\sum_{E \in \mathcal{E}_{\Omega} ( \mathcal{R} )}
\sum_{K^{\prime} \in \mathcal{T}_E}
h_{K^{\prime}}^{1/2}
\Vert  \langle \llbracket \nabla_{\mathcal{T}} u_h \rrbracket_E , \mathbf{n}_{E} \rangle \Vert_{L^2 ( E )}
\Vert \nabla_{\widehat{\mathcal{T}}} \widehat{v}_h^{\ast}
\Vert_{\mathbf{L}^2 ( \omega_{\mathcal{T}_{K^{\prime}}^1} )}.
\end{align*}
Rearranging the terms and using the fact that $\mathcal{R} \subseteq \mathcal{R}^{1/2}$ and $\mathcal{E}_{\Omega} ( \mathcal{R} ) \subseteq \mathcal{E}_{\Omega} ( \mathcal{R}^{1/2})$, we further observe that
\begin{align*}
a_{\widehat{\mathcal{T}}} ( \widehat{e}_h, ( 1 - \widehat{P}) \widehat{v}_h^{\ast} )
\lesssim &
\sum_{K \in \mathcal{R}^{1/2}}
\left(
h_K
\Vert f + \Delta u_h \Vert_{L^2 ( K )}
+ 
\sum_{E \in \mathcal{E}_{\Omega} ( K )}
h_K^{1/2}
\Vert  \langle \llbracket \nabla_{\mathcal{T}} u_h \rrbracket_E , \mathbf{n}_E \rangle \Vert_{L^2 ( E )}
\right)
\Vert \nabla_{\widehat{\mathcal{T}}} \widehat{v}_h^{\ast}
\Vert_{\mathbf{L}^2 ( \omega_{\mathcal{T}_K^1} )}.
\end{align*}
Since $\operatorname{card} \mathcal{E}_{\Omega} ( K ) \leq 3$, the term in brackets in the previous estimate consists of at most four terms. This and an application of the discrete Cauchy-Schwarz inequality to the right-hand side of previous estimate, yield
\begin{align*}
&\sum_{K \in \mathcal{R}^{1/2}}
\left(
h_K
\Vert f + \Delta u_h \Vert_{L^2 ( K )}
+
\sum_{E \in \mathcal{E}_{\Omega} ( K )}
h_K^{1/2}
\Vert  \langle \llbracket \nabla_{\mathcal{T}} u_h \rrbracket_E , \mathbf{n}_{E} \rangle \Vert_{L^2 ( E )}
\right)
\Vert \nabla_{\widehat{\mathcal{T}}} \widehat{v}_h^{\ast}
\Vert_{\mathbf{L}^2 ( \omega_{\mathcal{T}_K^1} )}
\\
&\leq 
2 \left( \sum_{K \in \mathcal{R}^{1/2}}
h_K^2
\Vert f + \Delta u_h \Vert_{L^2 ( K )}^2
+
\sum_{E \in \mathcal{E}_{\Omega} ( K )}
h_K
\Vert  \langle \llbracket \nabla_{\mathcal{T}} u_h \rrbracket_E , \mathbf{n}_{E} \rangle \Vert_{L^2 ( E )}^2
\right)^{1/2}
\left(
\sum_{K \in \mathcal{R}^{1/2}}
\Vert  \nabla_{\widehat{\mathcal{T}}} \widehat{v}_h^{\ast} \Vert_{\mathbf{L}^2 ( \omega_{\mathcal{T}_K^1} )}^2 \right)^{1/2}.
\end{align*}
The finite overlap property and the inclusion $\omega_{\mathcal{R}^2} \subseteq \Omega$ lead to $\sum_{K \in \mathcal{R}^{1/2}} \Vert  \nabla_{\widehat{\mathcal{T}}} \widehat{v}_h^{\ast} \Vert_{\mathbf{L}^2 ( \omega_{\mathcal{T}_K^1} )}^2  \lesssim 
\Vert \nabla_{\widehat{\mathcal{T}}} \widehat{v}_h \Vert_{\mathbf{L}^{2}( \Omega )}^2$.
We arrive at the final estimate 
\begin{align*} 
a_{\widehat{\mathcal{T}}} ( \widehat{e}_h, ( 1 - \widehat{P}) \widehat{v}_h^{\ast} )
 \lesssim
\mu ( u_h ; \mathcal{R}^{1/2}, f)
\Vert \nabla_{\widehat{\mathcal{T}}}  \widehat{v}_h^{\ast} \Vert_{\mathbf{L}^2 ( \Omega )},
\end{align*}
by using the equivalence of local mesh-sizes in \eqref{Eq:Equivalence local mesh size and volume}.
\end{proof}

\section{Stability and reduction}
\label{Sec:StabilityReductionEfficiency}

In this section we prove that the estimator $\eta ( \cdot ; \mathcal{T},f)$ from \eqref{SubEq:A posteriori - global} satisfies the axioms \eqref{Eq:Stability} and \eqref{Eq:Reduction}. Both proofs rely on standard techniques for residual estimators since $\eta ( \cdot; \mathcal{T}, f)$ is of this type.

\begin{lemma}[Stability] \label{Lem:Stability}
For any $k \geq 1$ odd and admissible refinement $\widehat{\mathcal{T}} \in \mathbb{T} ( \mathcal{T} )$ it holds that 
\begin{align*}
\vert \eta ( v ; \mathcal{T} \cap \widehat{\mathcal{T}}, f) 
- 
\eta ( \widehat{v}; \widehat{\mathcal{T}} \cap \mathcal{T}, f) \vert
& \lesssim
\Vert \nabla_{\mathcal{T}} v 
-  
\nabla_{\widehat{\mathcal{T}}} \widehat{v} \Vert_{\mathbf{L}^{2}( \Omega )}
\qquad \forall v \in \CR_{k,0} ( \mathcal{T}), \
\widehat{v} \in \CR_{k,0} ( \widehat{\mathcal{T}} ).
\end{align*}
\end{lemma}

\begin{proof}
Recall from \eqref{Eq:A posteriori}, that $\eta( \cdot ; \mathcal{T} \cap \widehat{\mathcal{T}} , f)^2 = \sum_{K \in \mathcal{T} \cap \widehat{\mathcal{T}}} \mu (\cdot ; K, f)^2 + \nu ( \cdot; K)^2$. This leads to
\begin{align}
\vert \eta ( v ; \mathcal{T} \cap \widehat{\mathcal{T}}, f) 
- 
\eta ( \widehat{v}; \widehat{\mathcal{T}} \cap \mathcal{T}, f) \vert
& \leq
\left(
\sum_{K \in \mathcal{T} \cap \widehat{\mathcal{T}}}
\vert \mu (v ; K, f) - \mu ( \widehat{v} ; K, f) \vert^2
+
\vert \nu (v ; K) - \nu ( \widehat{v} ; K ) \vert^2
\right)^{1/2}.
\label{Eq:Stability - first estimate}
\end{align}
Choosing any $K \in \mathcal{T} \cap \widehat{\mathcal{T}}$, we first consider the term $\vert \mu (v ; K, f) - \mu ( \widehat{v} ; K, f) \vert$. Using the definition of $\mu ( \cdot ; K , f)$ from \eqref{SubEq:A posteriori - conformity estimator} we obtain
\begin{align*}
\vert \mu (v ; K, f) - \mu ( \widehat{v} ; K, f) \vert
\leq &
\vert K \vert^{1/2}
\left\vert \Vert f + \Delta v \Vert_{L^2 ( K )} 
-
\Vert f + \Delta \widehat{v} \Vert_{L^{2}( K )} \right\vert
+
\\
&+
\vert K \vert^{1/4}
\sum_{E \in \mathcal{E}_{\Omega} ( K )}
\left\vert \Vert \langle \llbracket \nabla_{\mathcal{T}} v \rrbracket_E , \mathbf{n}_E \rangle \Vert_{L^{2}( E )}
-
 \Vert \langle \llbracket \nabla_{\widehat{\mathcal{T}}} \widehat{v} \rrbracket_E , \mathbf{n}_E \rangle \Vert_{L^{2}( E )}
\right\vert
\\
\leq &
h_K \Vert \Delta ( v - \widehat{v}) \Vert_{L^{2}(  K )}
+
h_K^{1/2} \sum_{E \in \mathcal{E}_{\Omega} ( K )} \Vert \left\langle \llbracket \nabla_{\mathcal{T}}  v - \nabla_{\widehat{\mathcal{T}}}  \widehat{v} \rrbracket_E , \mathbf{n}_E \right\rangle \Vert_{L^{2}( E )}
=: S_1 + S_2.
\end{align*}
We use an inverse inequality \cite[(12.1)]{ErnGuermondI} to estimate $
S_1 
=
h_K \Vert \Delta ( v - \widehat{v}) \Vert_{L^{2}(  K )}
\lesssim \Vert \nabla ( v -  \widehat{v}) \Vert_{\mathbf{L}^2 ( K )}
$.

Similarly, we use the trace inequality and the orthogonality of the normal vector $\mathbf{n}_E$ and the tangent vector $\mathbf{t}_E$ and obtain 
\begin{align*}
S_2
&=
h_K^{1/2} 
\sum_{E \in \mathcal{E}_{\Omega} ( K )} 
\Vert \langle \llbracket \nabla_{\mathcal{T}}  v -  \nabla_{\widehat{\mathcal{T}}} \widehat{v} \rrbracket_E , \mathbf{n}_E \rangle \Vert_{L^{2}( E )}
\leq
h_K^{1/2} 
\sum_{E \in \mathcal{E}_{\Omega} ( K )} 
\Vert  \llbracket \nabla_{\mathcal{T}}  v -  \nabla_{\widehat{\mathcal{T}}} \widehat{v} \rrbracket_E \Vert_{L^{2}(  E )}
\\
&\lesssim
\sum_{E \in \mathcal{E}_{\Omega} ( K )} 
\sum_{K^{\prime} \in \mathcal{S}_E}
\frac{h_K^{1/2}}{h_{K^{\prime}}^{1/2}}
\Vert \nabla_{\mathcal{T}} v -  \nabla_{\widehat{\mathcal{T}}}  \widehat{v} \Vert_{\mathbf{L}^2 ( K^{\prime} )},
\end{align*} 
where $\mathcal{S}_E$ consists of the two adjacent triangles of $E$. Based on the location of the edge $E \in \mathcal{E}_{\Omega} ( \mathcal{T} \cap \widehat{\mathcal{T}})$ we distinguish two cases. For this let $\Gamma_{\mathcal{R}}$ be the interface between $\omega_{\mathcal{T} \cap \widehat{\mathcal{T}}}$ and $\omega_{\mathcal{R}}$.

\textbf{Case 1: $E \in \mathcal{E}_{\Omega} ( \mathcal{T} \cap \widehat{\mathcal{T}}) \setminus \mathcal{E} ( \Gamma_{\mathcal{R}} )$.} This implies that $\mathcal{S}_E = \mathcal{T}_E \subseteq \mathcal{T} \cap \widehat{\mathcal{T}}$. Therefore we know from the equivalence of local mesh sizes \eqref{Eq:Equivalence local mesh size and volume} that $h_K^{1/2} / h_{K^{\prime}}^{1/2} \lesssim 1$.

\textbf{Case 2: $E \in \mathcal{E} ( \Gamma_{\mathcal{R}} )$.} In this case we know that there exists a triangle $\widehat{K} \in \widehat{\mathcal{R}}$, such that $\mathcal{S}_E = \{ K, \widehat{K} \}$. Consider any $K^{\prime} \in \mathcal{S}_E$. If $K^{\prime} = K$ then  $h_K^{1/2} / h_{K^{\prime}}^{1/2} = 1$. If $K^{\prime} = \widehat{K}$, then there exists exactly one triangle $K^{\prime \prime} \in \mathcal{T}$ with $K^{\prime} \in \operatorname{succ} ( K^{\prime \prime} )$. From $E \in \mathcal{E} (\Gamma_{\mathcal{R}})$ follows $ \vert K^{\prime} \vert = \vert K^{\prime \prime}\vert / 2$. This reduces the estimate of $h_K^{1/2} / h_{K^{\prime}}^{1/2} \lesssim h_K^{1/2} / h_{K^{\prime \prime}}^{1/2}$ to \textbf{Case 1}.


Therefore, we have shown that $S_2 \lesssim \sum_{E \in \mathcal{E}_{\Omega} ( K )} \sum_{K^{\prime} \in \mathcal{T}_E} \Vert \nabla_{\mathcal{T}} v - \nabla_{\widehat{\mathcal{T}}} \widehat{v} \Vert_{\mathbf{L}^2 ( K^{\prime} )}$. Notice that the upper bound in this inequality consists of at most six terms. In summary, we get
\begin{align}
\vert \mu (v ; K, f) - \mu ( \widehat{v} ; K, f) \vert^2
\lesssim &
\left(
\Vert \nabla ( v -  \widehat{v}) \Vert_{\mathbf{L}^2 ( K )}
+ 
\sum_{E \in \mathcal{E}_{\Omega} ( K )} 
\sum_{K^{\prime} \in \mathcal{T}_E }
\Vert \nabla_{\mathcal{T}} v -  \nabla_{\widehat{\mathcal{T}}}  \widehat{v} \Vert_{\mathbf{L}^2 ( K^{\prime} )}
\right)^2
\nonumber
\\
\lesssim&
\Vert \nabla_{\mathcal{T}} v - \nabla_{\widehat{\mathcal{T}}} \widehat{v} \Vert_{\mathbf{L}^2 ( \omega_{\mathcal{T}_K} )}^2.
\label{Eq:Stability - mu}
\end{align}

The estimation of the term $\vert \nu (v ; K) - \nu ( \widehat{v} ; K ) \vert$ is identical with the estimation of the normal jump in the previous case, by replacing $\mathbf{n}_E$ with $\mathbf{t}_E$, hence further details are omitted. We conclude that
\begin{align}
\vert \nu (v ; K) - \nu ( \widehat{v} ; K ) \vert^2
& \lesssim
\Vert \nabla_{\mathcal{T}} v - \nabla_{\widehat{\mathcal{T}}} \widehat{v} \Vert_{\mathbf{L}^2 ( \omega_{\mathcal{T}_K} )}^2.
 \label{Eq:Stability -nu}
\end{align}
Combining \eqref{Eq:Stability - mu} and \eqref{Eq:Stability -nu} into \eqref{Eq:Stability - first estimate} and using the finite overlap property \eqref{Eq:Finite overlap and minimal angle constants} concludes the proof.
\end{proof}

\begin{lemma}[Reduction] \label{Lem:Reduction}
There exist constants $0 < \rho_2 < 1$ and $\Lambda_2 > 0$ such that
\begin{align}
\eta ( \widehat{v}, \widehat{\mathcal{T}} \setminus \mathcal{T}, f)^2
& \leq
\rho_2 \eta ( v, \mathcal{T} \setminus \widehat{\mathcal{T}}, f)^2 
+
\Lambda_2 
\Vert \nabla_{\mathcal{T}} v 
-  
\nabla_{\widehat{\mathcal{T}}} \widehat{v} \Vert_{\mathbf{L}^{2}( \Omega )}^2.
\end{align}
\end{lemma}

\begin{proof}
The proof splits into two parts. First we prove
\begin{align}
\eta ( \widehat{v}, \widehat{\mathcal{T}} \setminus \mathcal{T}, f) 
&\leq 
2^{-1/4} 
\eta ( v, \mathcal{T} \setminus \widehat{\mathcal{T}}, f)
+ 
\widetilde{\Lambda}_2 
\Vert \nabla_{\mathcal{T}} v 
-  
\nabla_{\widehat{\mathcal{T}}} \widehat{v} \Vert_{\mathbf{L}^{2}( \Omega )},
\label{Eq:Reduction - sqrt estimate}
\end{align} 
for some $\widetilde{\Lambda}_2  > 0$. 

Second, we apply a Young's inequality to \eqref{Eq:Reduction - sqrt estimate} to obtain the reduction property in \eqref{Eq:Reduction}.

\textbf{Step 1:} Let $v \in \CR_{k,0} ( \mathcal{T})$ and $\widehat{v} \in \CR_{k,0} ( \widehat{\mathcal{T}})$ be any coarse and fine $\CR_k$ function and observe
\begin{align*}
\eta ( \widehat{v}; \widehat{\mathcal{T}} \setminus \mathcal{T}, f) 
& \leq
 \eta ( v; \widehat{\mathcal{T}} \setminus \mathcal{T}, f)
+
\vert
\eta (\widehat{v}; \widehat{\mathcal{T}} \setminus \mathcal{T}, f )
- 
\eta (v;  \widehat{\mathcal{T}} \setminus \mathcal{T}, f )
\vert
=:
R_1 + R_2.
\end{align*}
$R_2$ has the same general structure as the left-hand side term of \eqref{Eq:Stability}. We therefore follow the arguments of the proof of Lemma \ref{Lem:Stability} on the fine mesh $\widehat{\mathcal{T}}$ and conclude that there exists $\widetilde{\Lambda}_2 > 0$ such that $R_2 \leq \widetilde{\Lambda}_2 \Vert \nabla_{\mathcal{T}} v - \nabla_{\widehat{\mathcal{T}}} \widehat{v} \Vert_{\mathbf{L}^{2}( \Omega )}$.
The term $R_1$ requires some more details. Using the definition of $\eta ( \cdot; \widehat{K}, f)$ from \eqref{SubEq:A posteriori - local}, we rearrange $R_1^2$ into
\begin{align*}
R_1^2
&=
\sum_{\widehat{K} \in \widehat{\mathcal{T}} \setminus \mathcal{T}}
\eta ( v ; \widehat{K} ,f )^2
=
\sum_{\widehat{K} \in \widehat{\mathcal{T}} \setminus \mathcal{T}}
\mu ( v; \widehat{K}, f)^2 
+
\nu (v ; \widehat{K})^2
=
\sum_{K \in \mathcal{T} \setminus \widehat{\mathcal{T}}}
\sum_{\widehat{K} \in \operatorname{succ} ( K )}
\mu ( v; \widehat{K}, f)^2 
+
\nu (v ; \widehat{K})^2
 \\
&= 
\sum_{K \in \mathcal{T} \setminus \widehat{\mathcal{T}}}
\alpha_K^2
+
\beta_K^2
+ 
\gamma_K^2
\end{align*}
with the terms $\alpha_K^2,\beta_K^2,\gamma_K^2 \in \mathbb{R}_{\geq 0}$ which are given by
\begin{subequations} \label{Eq:Reduction - 3term}
\begin{align}
\alpha_K^2 
&:=
\sum_{\widehat{K} \in \operatorname{succ} ( K )}
\vert \widehat{K} \vert
\Vert f + \Delta v \Vert_{L^2 ( \widehat{K} )}^2,
\label{SubEq:Reduction - 3term - alpha}
\\
\beta_K^2
&:=
\sum_{\widehat{K} \in \operatorname{succ} ( K )}
\sum_{E \in \mathcal{E}_{\Omega} ( \widehat{K} )}
\vert \widehat{K} \vert^{1/2} 
\Vert \langle \llbracket \nabla_{\mathcal{T}} v \rrbracket_{\widehat{E}} , \mathbf{n}_{\widehat{E}} \rangle \Vert_{L^{2}( \widehat{E} )}^2,
\label{SubEq:Reduction - 3term - beta}
\\
\gamma_K^2 
&:=
\sum_{\widehat{K} \in \operatorname{succ} ( K )}
\sum_{E \in \mathcal{E} ( \widehat{K} )}
\vert \widehat{K} \vert^{1/2} 
\Vert \langle \llbracket \nabla_{\mathcal{T}} v \rrbracket_{\widehat{E}} , \mathbf{t}_{\widehat{E}} \rangle \Vert_{L^{2}( \widehat{E} )}^2.
\label{SubEq:Reduction - 3term - gamma}
\end{align}
\end{subequations}
We first consider $\alpha_K^2$. Since $K \in \mathcal{T}$ and $v \in \CR_{k,0} ( \mathcal{T} )$ we have $\Delta ( \left. v \right\vert_K ) \in  \mathbb{P}_{k-2} ( K )$. Also $\widehat{K} \in \operatorname{succ} ( K )$ implies $\vert \widehat{K} \vert \leq \vert K \vert / 2$. Consequently we estimate 
\begin{align*}
\alpha_K^2 
& \leq
\frac{\vert K \vert}{2} 
\sum_{\widehat{K} \in \operatorname{succ} ( K )}
\Vert f + \Delta v \Vert_{L^2 ( \widehat{K} )}^2
=
\frac{\vert K \vert}{2} 
\Vert f + \Delta v \Vert_{L^2 ( K )}^2.
\end{align*}
For $\beta_K^2$ we consider any edge $\widehat{E} \in \mathcal{E}_{\Omega} ( \operatorname{succ} ( K ) )$. Clearly, if the interior of $\widehat{E}$ is contained in the interior of $K$, then $\llbracket \nabla_{\mathcal{T}} v \rrbracket_{\widehat{E}} \equiv 0$ holds and hence
\begin{align*}
\beta_K^2
& \leq
\frac{\vert K \vert^{1/2}}{\sqrt{2}}
\sum_{\widehat{K} \in \operatorname{succ} ( K )}
\sum_{E \in \mathcal{E}_{\Omega} ( \widehat{K} )}
\Vert \langle \llbracket \nabla_{\mathcal{T}} v \rrbracket_{\widehat{E}} , \mathbf{n}_{\widehat{E}} \rangle \Vert_{L^{2}( \widehat{E} )}^2
=
\frac{\vert K \vert^{1/2}}{\sqrt{2}}
\sum_{E \in \mathcal{E}_{\Omega} ( K )}
\sum_{\widehat{E} \in \operatorname{succ} ( E)}
\Vert \langle \llbracket \nabla_{\mathcal{T}} v \rrbracket_{\widehat{E}} , \mathbf{n}_{\widehat{E}} \rangle \Vert_{L^{2}( \widehat{E} )}^2
\\
& \leq
\frac{\vert K \vert^{1/2}}{\sqrt{2}}
\sum_{E \in \mathcal{E}_{\Omega} ( K )}
\Vert \left\langle [ \nabla_{\mathcal{T}} v ]_{E} , \mathbf{n}_{E} \right\rangle \Vert_{L^{2}( E )}^2.
\end{align*}
Since the same arguments hold verbatim for $ \gamma_K^2$, no further details are provided and we obtain that $\gamma_K^2 \leq \vert K \vert^{1/2}\Vert \sum_{E \in \mathcal{E} ( K )} \left\langle [ \nabla_{\mathcal{T}} v ]_{\widehat{E}} , \mathbf{t}_{\widehat{E}} \right\rangle \Vert_{L^{2}( E )}^2/ \sqrt{2}$. Combining all three estimates, we conclude that
\begin{align*}
R_1^2
& \leq
\frac{1}{\sqrt{2}}
\sum_{K \in \mathcal{T} \setminus \widehat{\mathcal{T}}}
\mu ( v ; K , f)^2
+
\nu ( v ; K )^2
=
\frac{1}{\sqrt{2}} 
\sum_{K \in \mathcal{T} \setminus \widehat{\mathcal{T}}}
\eta ( v ; K ,f)^2
=
\frac{1}{\sqrt{2}}
\eta ( v ; \mathcal{T} \setminus \widehat{\mathcal{T}} ,f)^2.
\end{align*}
Taking the square root and combining this estimate with the estimate of $R_2$ implies \eqref{Eq:Reduction - sqrt estimate}.

\textbf{Step 2:} Having proved the estimate \eqref{Eq:Reduction - sqrt estimate}, we employ a Young's inequality to arrive at
\begin{align*}
\eta ( \widehat{v} ; \widehat{\mathcal{T}} \setminus \mathcal{T}, f)^2
& \leq
2^{-1/2} ( 1 + \lambda) 
\eta ( v ; \mathcal{T} \setminus \widehat{\mathcal{T}} ,f)^2
+
\widetilde{\Lambda}_2 ( 1 + \lambda^{-1})
\Vert \nabla_{\mathcal{T}} v - \nabla_{\widehat{\mathcal{T}}} \widehat{v} \Vert_{\mathbf{L}^{2}( \Omega )}^2.
\end{align*}
Setting $\lambda := ( \sqrt{2} - 1)/2$, we compute that $2^{-1/2} ( 1 + \lambda) = (1 + \sqrt{2})/\sqrt{8} =: \rho_2 < 1$.
\end{proof}

\begin{figure}
\begin{center}
\includegraphics[width = .23 \textwidth]{"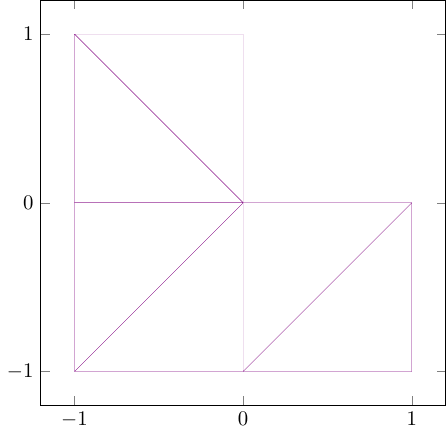"}
\\
\includegraphics[width = .23 \textwidth]{"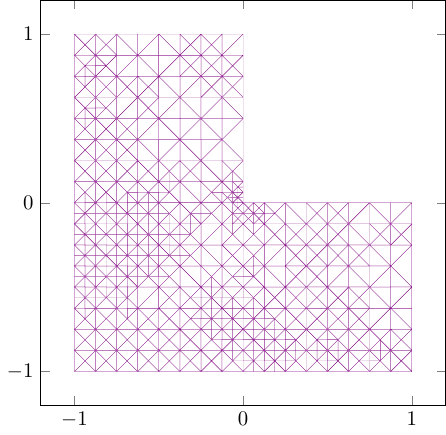"}
\hfil
\includegraphics[width = .23 \textwidth]{"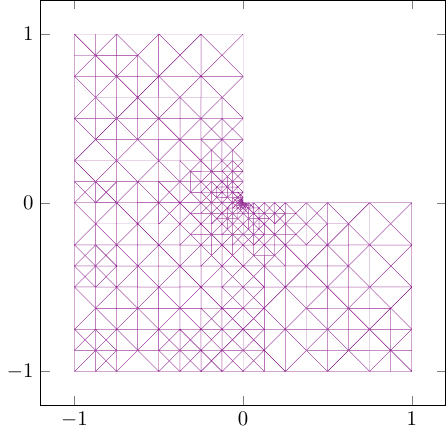"}
\hfil
\includegraphics[width = .23 \textwidth]{"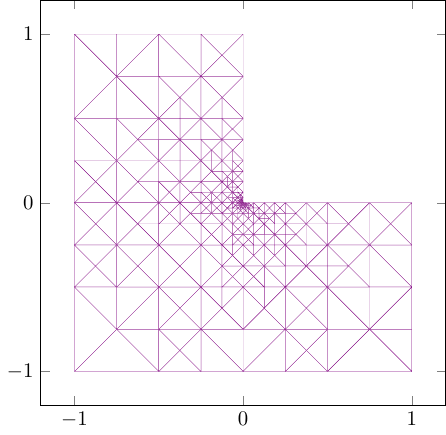"}
\hfil
\includegraphics[width = .23 \textwidth]{"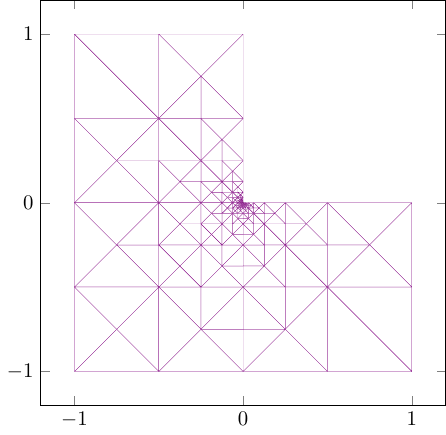"}
\end{center}
\caption{Top: initial triangulation $\mathcal{T}_0$. Bottom: examples of refined meshes generated by the adaptive algorithm. From left to right we have that $\operatorname{card} ( \mathcal{T} ) = 893$ for $k=1$, $\operatorname{card} ( \mathcal{T} ) = 941$ for $k = 3$, $\operatorname{card} ( \mathcal{T} ) = 975$ for $k = 5$ and $ \operatorname{card} ( \mathcal{T} ) = 926$ for $k = 7$.}
\label{Fig:Refinement history}
\end{figure}

\section{Numerical experiments}
\label{Sec:NumericalExperiments}

In Sections \ref{Sec:dRel} and \ref{Sec:StabilityReductionEfficiency}, we prove that the Crouzeix--Raviart finite element method for odd $k \geq 1$ along with the estimator form \eqref{SubEq:A posteriori - global} is stable \eqref{Eq:Stability}, has the reduction property \eqref{Eq:Reduction} and is discretely reliability \eqref{Eq:Discrete reliability}. For a theoretical proof of optimal rates, we would still need to show quasi-orthogonality \eqref{Eq:Quasi orthogonoality - general} but in this section, we provide numerical evidence that the odd high-order Crouzeix--Raviart AFEM loop \eqref{Eq:Adaptive Algorithm} converges with optimal rates.

\subsection{Numerical realization}

All numerical experiments are performed with the finite element software \texttt{NGSolve} \cite{NGSolve}. 
Since only the lowest order $\CR_1$ space is already available in the standard version of \texttt{NGSolve}, we implemented the higher order $\CR_k$ spaces of odd polynomial degree $k \geq 3$ using \texttt{mylittlengsolve}, which is intended for the self-implementation of custom finite element spaces. See https://github.com/NGSolve/mylittlengsolve for reference. 

The adaptive mesh refinement as illustrated in Figure \ref{Fig:Refinement history} uses the residual based error estimator described in \eqref{Eq:A posteriori}. The bulk parameter for the Dörfler marking \cite{Doerfler_convergent} is chosen as $\theta = 1/2$. Both adaptive algorithm and the uniform mesh refinement strategy is terminated once the dimension of $\CR_{k,0} ( \mathcal{T}_{\ell} )$ exceeds $10^5$. Here $\mathcal{T}_{\ell}$ for $\ell \in \mathbb{N}_0$ denotes the $\ell$-th admissible refinement of $\mathcal{T}_0$.

\subsection{L-shaped corner singularity}

Let us consider the L-shaped domain $\Omega := \left] -1,1 \right[ ^2 \setminus [0,1 [^2$ with interior angle $3 \pi / 2$ at the origin. We consider $u \in H_0^1 ( \Omega )$ from \cite{Grisward-Singularities} given by
\begin{align}
u ( r, \theta )
&:=
r^{2/3} 
\sin \left( \frac{2}{3}\left( \alpha - \frac{\pi}{2}\right) \right)
\left( r^2 \cos^2 \alpha - 1 \right)
\left( r^2 \sin^2 \alpha- 1 \right).
\end{align}
in polar coordinates
, i.e., for any $\mathbf{x} \in \Omega$, $r := \Vert \mathbf{x} \Vert$ is the distance from the origin and $\alpha \in \left[\pi/2, 2 \pi\right]$ is the angle of $\mathbf{x}$ measured against the $x_1$ axis. 
By choosing $f := - \Delta u$, the function $u$ is the exact solution to the variational Poisson problem \eqref{Eq:VariationalPoisson}. In Figure \ref{Fig:Convergence history}, we display the convergence history of the estimator $\eta ( \cdot ; \mathcal{T}_{\ell} )$ and the exact error $\Vert \nabla u - \nabla_{\mathcal{T}} u_{\ell} \Vert_{\mathbf{L}^{2}(\Omega)}$ in the $H^1$ semi norm for the $\CR_k$ FEM using polynomial degrees $k \in \{ 1, 3, 5, 7 \}$ on both uniform and adaptive meshes. As mentioned before, the adaptive algorithm is subject to the bulk parameter $\theta = 1/2$ for the marking strategy.
\begin{figure}
\begin{center}
\includegraphics[width = .47 \textwidth]{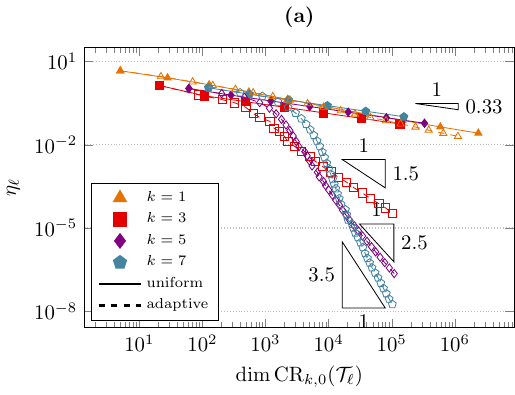}
\hfil
\includegraphics[width = .47 \textwidth]{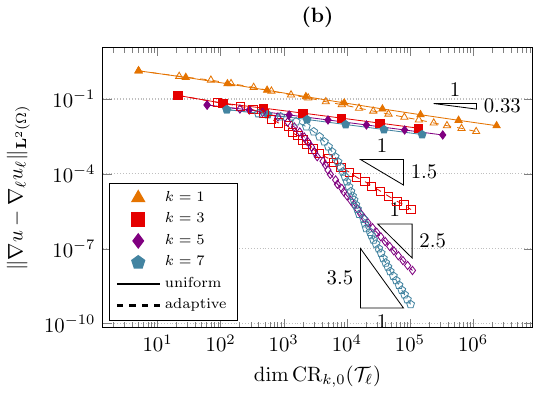}
\end{center}
\caption{Convergence history of the Crouzeix--Raviart element using both AFEM and uniform refinement for $k \in \{ 1,3,5,7 \}$. \textbf{(a)} shows the convergence history of the estimator $\eta_{\ell} := \eta ( \cdot ; \mathcal{T}_{\ell}, f)$ and \textbf{(b)} shows the convergence history of the error $u -u_{\ell}$ with respect to the piecewise $H^1$ semi-norm}
\label{Fig:Convergence history}
\end{figure}
Since $u$ exhibits a corner singularity at the origin, uniform mesh refinement leads to a sub-optimal convergence rate of $-1/3$ with respect to the number of DOFs. The adaptive algorithm on the other hand recovers the optimal convergence rate of $-k/2$ for all polynomial degrees $k \in \{ 1, 3, 5, 7 \}$. Since the optimal rate for the lowest order case of $k=1$ and the sub-optimal rate are very close in the pre-asymptotic range, we have extended the experiment to $10^6$ DOFs.

As expected, we see that the adaptive algorithm concentrates its refinement around the corner singularity of the exact solution $u$. In the bottom row of Figure \ref{Fig:Refinement history}, we display four meshes for $k \in \{ 1, 3, 5, 7 \}$ and with about $900$ elements. We see that the grading of the meshes increases with the polynomial degree, since the higher polynomial degree yields higher approximation orders in the smooth regions of $u$.

\subsection{Conclusion}

We have shown that the estimator $\eta ( \cdot; \mathcal{T}, f)$ is discretely reliable \eqref{Eq:Discrete reliability} (cf. Theorem \ref{Thm:Discrete reliabilty}), satisfies the stability axiom \eqref{Eq:Stability} (cf. Lemma \ref{Lem:Stability}) and  the reduction axiom \eqref{Eq:Reduction} (cf. Lemma\ref{Lem:Reduction}). Stability and reduction are obtained using standard arguments stemming from the optimality proofs of the conforming AFEM using $S_{k,0} ( \mathcal{T} )$ and its associated residual estimator.
The proof of discrete reliability for odd high-order $\operatorname{CR}_k$ FEM is rather technical and relies on the explicit construction of suitable approximation- and right-inverse operators. For a complete proof of optimal convergence of the odd high-order $\operatorname{CR}_k$ AFEM algorithm, one would also need to prove the axiom of generalized quasi-orthogonality \eqref{Eq:Quasi orthogonoality - general} but this is a topic of future research. However, the numerical experiments in Section \ref{Sec:NumericalExperiments} provide numerical evidence that the arbitrary high-order $\operatorname{CR}_k$ AFEM converges with optimal rates.

\appendix

\section{Technical results}

\begin{lemma} \label{A:Lem:Equivalence Jz JEk-1 for constants}
Let $\mathcal{S} \subseteq \mathcal{T}$ be a sub-mesh. For any vertex $\mathbf{z} \in \mathcal{V} ( \mathcal{S})$, consider some constant function $c : \omega_{\mathcal{S}_{\mathbf{z}}} \to \mathbb{R}$. We have that
\begin{align}
F_{\mathcal{S}, \mathbf{z}} ( c )
&=
c
= 
F_{E,k-1} ( c ) 
\qquad \forall E \in \mathcal{E} ( \mathcal{S}_{\mathbf{z}} ).
\label{Eq:Constant equivalence}
\end{align}
\end{lemma}

\begin{proof}
Note that 
$
F_{\mathcal{S}, \mathbf{z}} ( c )
=
(
\sum_{K \in \mathcal{S}_{\mathbf{z}}} c) / \operatorname{card} \mathcal{S}_{\mathbf{z}}
= c
$
holds true  for all vertices $\mathbf{z} \in \mathcal{V} ( \mathcal{S} )$ due to \eqref{SubEq:Point eval}. By construction we know that $b_{E,k-1}$ is continuous across $E$ with $\left. b_{E,k-1} \right\vert_E = 1$ (cf. \eqref{SubEq:Basis functions - nonconf}). Consequently, through an application of \eqref{SubEq:Quasi biduality - edge}, we compute that
$
F_{E,k-1} ( c ) 
 =
c F_{E,k-1}  ( 1 )
= 
c F_{E,k-1} ( b_{E,k-1} )
=
c.
$ 
 \end{proof}

\begin{lemma}[Proof of \eqref{A:Eq:Functional estimates}] \label{A:Lem:Functional estimates}
For every sub-mesh $\mathcal{S} \subseteq \mathcal{T}$ and function $u \in H_k^{\CR} ( \mathcal{T} )$ we have that
\begin{align*}
\vert F_{E,j} ( u ) \vert
& \lesssim
\begin{cases}
h_K^{-1}
(\Vert u \Vert_{L^2 ( K )} 
+ 
h_K \Vert \nabla u \Vert_{\mathbf{L}^2( K)} ) 
&
j = k-1,
\\
 \Vert \nabla u \Vert_{\mathbf{L}^2 ( K )}
&
j \in \range{k-2},
\end{cases}
\\
\vert F_{K, \boldsymbol{\alpha}} ( u ) \vert
& \lesssim
h_K^{-1}
\Vert u \Vert_{L^2 ( K )} ,
\\
\vert F_{\mathcal{S},\mathbf{z}} ( u ) \vert
&\lesssim
h_{\mathcal{T}_{\mathbf{z}}}^{-1} 
\Vert u \Vert_{L^2 ( \omega_{\mathcal{T}_{\mathbf{z}}} )},
\end{align*}
for every edge $E \in \mathcal{E} ( \mathcal{T} )$, triangle $K \in \mathcal{T}$, vertex $\mathbf{z} \in \mathcal{S}$ and (multi-) indices $j \in \range{k-1}$ and $\boldsymbol{\alpha} \in \multirange{2}{k-3}$.
\end{lemma}

\begin{proof}
\textbf{@\labelcref{A:SubEq:Functional estimates - edge}:} Consider any function $u \in H_k^{\CR} ( \mathcal{T} )$ and edge $E \in \mathcal{E} ( \mathcal{T} )$. Combining \eqref{Eq:Constant equivalence} with \eqref{SubEq:Quasi biduality - edge}, we deduce that
\begin{align}
F_{E,j} ( c )
&=
c \delta_{j,k-1}
\qquad 0 \leq j \leq k-1
\label{Eq:Edge constants}
\end{align}
for any constant function $c : \Omega \to \mathbb{R}$. Thus for $j \in \range{k-2}$, an application of the multiplicative trace inequality \cite[(12.16)]{ErnGuermondI} and the Poincare inequality, we observe that
\begin{align*}
\vert F_{E,j} ( u ) \vert
&=
\vert F_{E,j} ( u - \overline{u}_K ) \vert
\leq
\Vert g_{E,j} \Vert_{L^{2}(  E )} 
\Vert u - \overline{u}_K \Vert_{L^{2}(  E )}
\\
& \lesssim
\Vert g_{E,j} \Vert_{L^{2}(  E )}
\Vert u - \overline{u}_K \Vert_{L^{2}(  K )}^{1/2}
(
h_K^{-1/2} \Vert u - \overline{u}_K \Vert_{L^{2}(  K )}^{1/2}
+
\Vert \nabla u \Vert_{\mathbf{L}^2 ( K )}^{1/2}
)
\lesssim
h_K^{1/2}
\Vert g_{E,j} \Vert_{L^{2}(  E )}
\Vert \nabla u \Vert_{\mathbf{L}^2 ( K )}.
\end{align*} 
Considering the term $ \Vert g_{E,j} \Vert_{L^2 ( E )}$, combining the definition \eqref{SubEq:Moment functions - ege}, with affine pullbacks to the interval $[ -1, 1]$, we have that $\Vert g_{E,j} \Vert_{L^{2}(  E )}
\lesssim
\vert E \vert^{-1/2}
\left( \Vert P_j^{(1,1)} \Vert_{L^{2}( [-1,1] )}
+
\Vert P_{k-1}^{(1,1)} \Vert_{L^{2}( [-1,1] )}\right)$
Some elementary computations using \cite[Eq. 18.9.15]{NIST:DLMF} reveals $\Vert P_j^{(1,1)} \Vert_{L^{2}( [-1,1] )}^2 = 4(j+1)/(j+2) \leq 4$ for all $j \in \mathbb{N}_0$.
A final application of the equivalence of local mesh-sizes \eqref{Eq:Equivalence local mesh size and volume} yields $\vert F_{E,j} ( u ) \vert \lesssim \Vert \nabla u \Vert_{\mathbf{L}^2 ( K )}$ if $j \in \range{k-2}$. For $j = k-1$, it is not possible to add a constant function into the estimate due to \eqref{Eq:Edge constants}. The multiplicative trace- and a Young's inequality reveals that $\vert F_{E,k-1} \vert
\lesssim
\Vert g_{E,k-1} \Vert_{L^{2}(  E )}
\left( h_K^{-1/2}  \Vert u \Vert_{L^{2}(  K )}
+
 h_K^{1/2} \Vert \nabla u \Vert_{L^2 ( K )} \right)$.
As for the case $j \in \range{k-2}$ we have that $\Vert g_{E,k-1} \Vert_{L^{2}(  E )} \lesssim \vert E \vert^{-1/2}$ and factoring out $h_K^{-1/2}$ in the last estimate leads to $\vert F_{E,k-1} ( u ) \vert
\lesssim
\left( \Vert u \Vert_{L^2 ( K )} 
+ 
h_K 
\Vert \nabla u \Vert_{\mathbf{L}^2 ( K )} \right)/h_K$.

\textbf{@\labelcref{A:SubEq:Functional estimates - volume}:} Similarly, consider any $K \in \mathcal{T}$. Applying \cite[Prop. 3.37]{verfuerth_book_neu} and \eqref{Eq:Pairwise orthogonality} yields
\begin{align*}
\vert F_{K, \boldsymbol{\alpha}} ( u )\vert
&\lesssim
\vert K \vert^{-1} 
\Vert P_{K, \boldsymbol{\alpha}}^{( 1,1,1 )} \Vert_{L^2 ( K )}
\Vert u \Vert_{L^2 ( K )}
\lesssim
\vert K \vert^{-1}
\lesssim
\Vert W_K^{1/2} P_{K, \boldsymbol{\alpha}}^{( 1,1,1 )} \Vert_{L^2 ( K )}
\Vert u \Vert_{L^2 ( K )}
\lesssim
h_K^{-1} 
\Vert u \Vert_{L^2 ( K )}.
\end{align*}

\textbf{@\labelcref{SubEq:Point eval - functional estimate}:} This is treated similarly as \eqref{A:SubEq:Functional estimates - volume} by applying \eqref{Prop:Vertex value}, \eqref{SubEq:Point eval - wheight estimate} and the local equivalence of mesh-sizes \eqref{Eq:Equivalence local mesh size and volume}.
\end{proof}

\section{Postponed proofs for \eqref{SubEq:Non-Conf -  error estimate}}
\label{A:Sec:Check prop}

In this section, we check the two properties from the proof of Lemma \ref{Lem:Upper bound by tangential jumps} \eqref{SubEq:Non-Conf -  error estimate}, which we have postponed to the appendix.

\begin{lemma}[Proof of \eqref{Eq:Localized companion estimate}] \label{Lem:Proof (3.3)}
Let $\widehat{\mathcal{T}} \in \mathbb{T} ( \mathcal{T} )$ be given. Then for any $\CR_k$ function $u \in \CR_{k,0} ( \mathcal{T} )$
\begin{align*}
\Vert \nabla_{\widehat{\mathcal{T}}} \widehat{u}^{\ast} - \nabla_{\mathcal{T}} u \Vert_{\mathbf{L}^2 ( \Omega )}
& \lesssim
\Vert \nabla_{\mathcal{T}} ( J_{k,0} u - u ) \Vert_{\mathbf{L}^2 ( \omega_{\mathcal{T} \setminus \widehat{\mathcal{T}}} )},
\end{align*}
holds true, where $\widehat{u}^{\ast} \in \CR_{k,0} ( \widehat{\mathcal{T}} )$ is given by $\widehat{u}^{\ast} = I_{k,0}^{\widehat{\mathcal{T}},\CR} ( J_{k,0} u )$. 
\end{lemma}

\begin{proof}
We first consider any unrefined triangle $K \in \mathcal{T} \cap \widehat{\mathcal{T}}$. For every fine $\CR_k$ function $\widehat{v} \in \CR_{k,0} ( \widehat{\mathcal{T}} )$, we combine \eqref{Eq:Local desciption} with \eqref{SubEq:Prop LkK - Local projection} and observe that $
( I_{k,0}^{\mathcal{T}, \CR} \widehat{v} ) \vert_K
=
I_{K,k}^{\operatorname{loc}} \widehat{v}
=
\widehat{v} \vert_K
$.
Consequently, using this and the fact that $I_{k,0}^{\mathcal{T}, \CR} \widehat{u}^{\ast} = u$ from \eqref{SubEq:Non-Conf - right inverse}, we obtain $u \vert_K = \widehat{u}^{\ast} \vert_K$ for all $K \in \mathcal{T} \cap \widehat{\mathcal{T}}$, implying
$
\Vert \nabla_{\widehat{\mathcal{T}}} \widehat{u}^{\ast}
- 
\nabla_{\mathcal{T}} u \Vert_{\mathbf{L}^2 ( \Omega )}
= 
\Vert \nabla_{\widehat{\mathcal{T}}} ( \widehat{u}^{\ast} - u ) \Vert_{\mathbf{L}^2 ( \omega_{ \widehat{\mathcal{T}} \setminus \mathcal{T}} )}.
$
Here we use the fact that $\nabla_{\widehat{\mathcal{T}}} u = \nabla_{\mathcal{T}} u$ by definition of the piecewise gradient. For the estimate of $\Vert \nabla_{\widehat{\mathcal{T}}} ( \widehat{u}^{\ast} - u ) \Vert_{\mathbf{L}^2 ( \omega_{ \widehat{\mathcal{T}} \setminus \mathcal{T}} )}$, consider any refined triangle $\widehat{K} \in \widehat{\mathcal{T}} \setminus \mathcal{T}$. We obtain
\begin{align*}
\Vert \nabla ( \widehat{u}^{\ast} - u ) \Vert_{\mathbf{L}^2 ( \widehat{K} )}
&=
\Vert \nabla ( I_{k,0}^{\widehat{\mathcal{T}} ,\CR} ( J_{k,0}^{\ast} u ) 
- 
u) \Vert_{\mathbf{L}^2 ( \widehat{K} )}
\overset{\eqref{SubEq:Prop LkK - Local projection}}{=}
\Vert \nabla I_{\widehat{K},k}^{\operatorname{loc}} ( J_{k,0}^{\ast} u  
- 
u) \Vert_{\mathbf{L}^2 ( \widehat{K} )}
\lesssim
\Vert \nabla (J_{k,0}^{\ast} u
- 
u) \Vert_{\mathbf{L}^2 ( \widehat{K} )}
\end{align*}
by using the local boundedness of $I_{k,0}^{\widehat{\mathcal{T}},\CR}$ form \eqref{SubEq:Non-Conf Approx Operator - loc bound} in the last step. Recall $\omega_{ \widehat{\mathcal{T}} \setminus \mathcal{T}} = \omega_{\mathcal{T} \setminus \widehat{\mathcal{T}}}$ from \eqref{Eq:Refined domain}. We take the previous estimate and deduce that
\begin{align*}
\Vert \nabla ( \widehat{u}^{\ast} - u ) \Vert_{\mathbf{L}^2 (\omega_{ \widehat{\mathcal{T}} \setminus \mathcal{T}} )}
&=
\sum_{\widehat{K} \in \widehat{\mathcal{T}} \setminus \mathcal{T}}
\Vert \nabla ( I_{k,0}^{,\widehat{\mathcal{T}} \CR} ( J_{k,0}^{\ast} u ) 
- 
u) \Vert_{\mathbf{L}^2 ( \widehat{K} )}^2
\lesssim
\sum_{\widehat{K} \in \widehat{\mathcal{T}} \setminus \mathcal{T}}
\Vert \nabla (J_{k,0}^{\ast} u
- 
u) \Vert_{\mathbf{L}^2 ( \widehat{K} )}^2
\\
&\lesssim
\sum_{K \in \mathcal{T} \setminus \widehat{\mathcal{T}}}
\sum_{\widehat{K} \in \operatorname{succ} ( K )}
\Vert \nabla (J_{k,0}^{\ast} u
- 
u) \Vert_{\mathbf{L}^2 ( \widehat{K} )}^2
\lesssim
\sum_{K \in \mathcal{T} \setminus \widehat{\mathcal{T}}}
\Vert \nabla (J_{k,0}^{\ast} u
- 
u) \Vert_{\mathbf{L}^2 ( K )}^2
\\
&\lesssim
\Vert \nabla_{\mathcal{T}} (J_{k,0}^{\ast} u
- 
u) \Vert_{\mathbf{L}^2 ( \omega_{\mathcal{T} \setminus \widehat{\mathcal{T}}} )}^2.
\qedhere
\end{align*}
\end{proof}

\begin{lemma}[Proof of \eqref{Eq:C7}] \label{Lem:C7}
Consider any simplex $K \in \mathcal{T}$, and recall the definition of $\varPsi_K : H_{k,0}^{\CR} ( \mathcal{T} ) \to \mathbb{R}_{\geq 0}$ from \eqref{Eq:Local tangetial jumps}. If for any $v \in \CR_{k,0} ( \mathcal{T} ) \vert_{\omega_{\mathcal{T}_K}} $, $\varPsi_K ( v ) = 0$ holds, then $ \left. v \right\vert_{\omega_{\mathcal{T}_K}}  \in (\CR_{k,0} ( \mathcal{T} ) \cap H_0^1 ( \Omega )) \vert_{\omega_{\mathcal{T}_K}}$.
\end{lemma}

\begin{proof}
Let $v \in  \CR_{k,0} ( \mathcal{T} ) \vert_{\omega_{\mathcal{T}_K}} $ such that $\varPsi_K ( v ) = 0$ be given. By construction, we have that $\langle \llbracket \nabla_{\mathcal{T}} v \rrbracket_E , \mathbf{t}_E \rangle = 0$ for all $E \in \mathcal{E} ( \mathcal{T}_K )$. This implies that for every $E \in \mathcal{E} ( \mathcal{T}_K )$ there exists a constant $c_E \in \mathbb{R}$ such that $\llbracket v \rrbracket_E = c_E$. We distinguish two cases.

\textbf{Case 1: $E \in \mathcal{E} _{\Omega}( \mathcal{T}_K)$.} In this case, we know that $\mathcal{T}_E \subseteq \mathcal{T}_K$ and that $E$ is an inner edge. Since $v \in \CR_k ( \mathcal{T}_K )$, we know that $\llbracket v \rrbracket_E = c_E \perp_{L^2 (E)} \mathbb{P}_{k-1} ( E )$. Consequently $c_E = 0$ and it follows that $v$ is continuous over $E$.

\textbf{Case 2: $E \in \mathcal{E}_{\partial \Omega} ( \mathcal{T}_K)$.} Since $E \subseteq \partial \Omega$, we know that $\llbracket v \rrbracket_E = v \vert_E = c_E$. Since $v \in \CR_k ( \mathcal{T}_K ) \cap H_{k,0}^{\CR} (\mathcal{T})$, it follows from the previous identity $c_E  \perp_{L^2 (E)} \mathbb{P}_{k-1} ( E )$ due to the zero boundary condition in the $\CR_k$ sense. Therefore, we have $c_E = 0$.


{Cases 1-2} imply that $v$ is continuous over all edges in the interior of $\omega_{\mathcal{T}_K}$ and that $ v \vert_E = 0$ for all boundary edges of $\partial \omega_{\mathcal{T}_K}$. Consequently we know that $v \in S_{k,0} ( \mathcal{T} ) \vert_{\omega_{\mathcal{T}_K}} = (\CR_{k,0} (\mathcal{T}) \cap H_0^1 ( \Omega ))\vert_{\omega_{\mathcal{T}_K}}$.
\end{proof}

\printbibliography

@preamble{ "\newcommand{\noopsort}[1]{} "
	# "\newcommand{\printfirst}[2]{#1} "
	# "\newcommand{\singleletter}[1]{#1} "
	# "\newcommand{\switchargs}[2]{#2#1} " }

@article {Doerfler_convergent,
    AUTHOR = {D{\"o}rfler, Willy},
     TITLE = {A convergent adaptive algorithm for {P}oisson's equation},
   JOURNAL = {SIAM J. Numer. Anal.},
  FJOURNAL = {SIAM Journal on Numerical Analysis},
    VOLUME = {33},
      YEAR = {1996},
    NUMBER = {3},
     PAGES = {1106--1124},
      ISSN = {0036-1429},
     CODEN = {SJNAAM},
   MRCLASS = {65N50 (65N55)},
  MRNUMBER = {1393904 (97e:65139)},
MRREVIEWER = {S. F. McCormick},
       DOI = {10.1137/0733054},
      
}

@article {Stevenson_adapt1,
    AUTHOR = {Stevenson, Rob},
     TITLE = {Optimality of a standard adaptive finite element method},
   JOURNAL = {Found. Comput. Math.},
    VOLUME = {7},
      YEAR = {2007},
    NUMBER = {2},
     PAGES = {245--269}
}

@preamble{
   "\def\cprime{$'$} "
}

@book {verfuerth_book_neu,
    AUTHOR = {Verf{\"u}rth, R{\"u}diger},
     TITLE = {A posteriori error estimation techniques for finite element
              methods},
 PUBLISHER = {Oxford University Press},
   ADDRESS = {Oxford},
      YEAR = {2013}
}

@article {Bonito2010,
    AUTHOR = {Bonito, Andrea and Nochetto, Ricardo H.},
     TITLE = {Quasi-optimal convergence rate of an adaptive discontinuous
              {G}alerkin method},
   JOURNAL = {SIAM J. Numer. Anal.},
    VOLUME = {48},
      YEAR = {2010},
    NUMBER = {2},
     PAGES = {734--771}
}

@misc{NIST:DLMF,
         key = "{\relax DLMF}",
       title = "{\it NIST Digital Library of Mathematical Functions}",
howpublished = "http://dlmf.nist.gov/, Release 1.0.13 of 2016-09-16",
         url = "http://dlmf.nist.gov/",
        note = "F.~W.~J. Olver, A.~B. {Olde Daalhuis}, D.~W. Lozier, B.~I. Schneider,
                R.~F. Boisvert, C.~W. Clark, B.~R. Miller and B.~V. Saunders, eds."}

@article {ABCM02,
    AUTHOR = {Arnold, Douglas N. and Brezzi, Franco and Cockburn, Bernardo
              and Marini, L. Donatella},
     TITLE = {Unified analysis of discontinuous {G}alerkin methods for
              elliptic problems},
   JOURNAL = {SIAM J. Numer. Anal.},
    VOLUME = {39},
      YEAR = {2001/02},
    NUMBER = {5},
     PAGES = {1749--1779}
}

@article {Baran_Stoyan,
    AUTHOR = {Baran, {\'A}. and Stoyan, G.},
     TITLE = {Gauss-{L}egendre elements: a stable, higher order
              non-conforming finite element family},
   JOURNAL = {Computing},
    VOLUME = {79},
      YEAR = {2007},
    NUMBER = {1},
     PAGES = {1--21}
}

@article {CrouzeixRaviart,
    AUTHOR = {Crouzeix, M. and Raviart, P.-A.},
     TITLE = {Conforming and nonconforming finite element methods for
              solving the stationary {S}tokes equations. {I}},
   JOURNAL = {Rev. Fran\c caise Automat. Informat. Recherche
              Op\'erationnelle S\'er. Rouge},
    VOLUME = {7},
      YEAR = {1973},
    NUMBER = {R-3},
     PAGES = {33--75}
}

@article {Brenner_Crouzeix,
    AUTHOR = {Brenner, Susanne C.},
     TITLE = {Forty years of the {C}rouzeix-{R}aviart element},
   JOURNAL = {Numer. Methods Partial Differential Equations},
    VOLUME = {31},
      YEAR = {2015},
    NUMBER = {2},
     PAGES = {367--396}
}

@article {CDS,
    AUTHOR = {Ciarle{t, Jr.}, Patrick and Dunkl, Charles F. and Sauter, Stefan A.},
     TITLE = {A family of {C}rouzeix-{R}aviart finite elements in 3{D}},
   JOURNAL = {Anal. Appl. (Singap.)},
    VOLUME = {16},
      YEAR = {2018},
    NUMBER = {5},
     PAGES = {649--691}
}

@ARTICLE{Crouzeix_Falk,
  author = {M. Crouzeix and R.S. Falk},
  title = {Nonconforming finite elements for {S}tokes problems},
  journal = {Math. Comp.},
  year = {1989},
  volume = {186},
  pages = {437-456}
}

@ARTICLE{Fortin_Soulie,
  author = {M. Fortin and M. Soulie},
  title = {A nonconforming quadratic finite element on triangles},
  journal = {International Journal for Numerical Methods in Engineering},
  year = {1983},
  volume = {19},
  pages = {505-520}
}

@article {Fortin_d3,
    AUTHOR = {Fortin, M.},
     TITLE = {A three-dimensional quadratic nonconforming element},
   JOURNAL = {Numer. Math.},
    VOLUME = {46},
      YEAR = {1985},
    NUMBER = {2},
     PAGES = {269--279}
}

@article {ChaLeeLee,
    AUTHOR = {Cha, Youngjoon and Lee, Miyoung and Lee, Sungyun},
     TITLE = {Stable nonconforming methods for the {S}tokes problem},
   JOURNAL = {Appl. Math. Comput.},
    VOLUME = {114},
      YEAR = {2000},
    NUMBER = {2-3},
     PAGES = {155--174}
}

@book {ErnGuermondI,
    AUTHOR = {Ern, Alexandre and Guermond, Jean-Luc},
     TITLE = {Finite elements {I}---{A}pproximation and interpolation},     
 PUBLISHER = {Springer, Cham},
      YEAR = {2021}
}

@incollection{nochetto2011primer,
  title={Primer of adaptive finite element methods},
  author={Nochetto, Ricardo H and Veeser, Andreas},
  booktitle={Multiscale and adaptivity: modeling, numerics and applications},
  pages={125--225},
  year={2011},
  publisher={Springer}
}

@article{BohneSauterCiarletdD,
      title={Crouzeix-Raviart elements on simplicial meshes in $d$ dimensions}, 
      author={Nis-Erik Bohne and Ciarlet Jr., Patrick and Stefan Sauter},
      year={2024},
      eprint={2407.04361},
      archivePrefix={arXiv},
      primaryClass={math.NA},
      url={https://arxiv.org/abs/2407.04361},
	  pubstate = {To appear in: Foundations of Computational Mathematics}
}

@article{CarstensenPuttkammer-howToProofDiscreteReliabiltiyNonConf,
    AUTHOR = {Casten Carstensen and Sophie Puttkammer},
     TITLE = {How to prove discrete reliability for nonconforming finite element methods},
   JOURNAL = {Journal of Computational Mathematics},
    VOLUME = {38},
      YEAR = {2020},
    NUMBER = {1},
     PAGES = {142--175},
       DOI = {10.4208/jmc.1908-m2018-0174},
}

@article {ArnoldBrezzi-mixedNonConvFEM85,
    AUTHOR = {Arnold, D. N. and Brezzi, F.},
     TITLE = {Mixed and nonconforming finite element methods:
              implementation, postprocessing and error estimates},
   JOURNAL = {RAIRO Mod\'el. Math. Anal. Num\'er.},
  FJOURNAL = {RAIRO Mod\'elisation Math\'ematique et Analyse Num\'erique},
    VOLUME = {19},
      YEAR = {1985},
    NUMBER = {1},
     PAGES = {7--32},
      ISSN = {0764-583X,1290-3841},
   MRCLASS = {65N30},
  MRNUMBER = {813687},
       DOI = {10.1051/m2an/1985190100071},
      
}

@article {CKNS-AFEMOptimalRatesConf,
    AUTHOR = {Cascon, J. Manuel and Kreuzer, Christian and Nochetto, Ricardo
              H. and Siebert, Kunibert G.},
     TITLE = {Quasi-optimal convergence rate for an adaptive finite element
              method},
   JOURNAL = {SIAM J. Numer. Anal.},
  FJOURNAL = {SIAM Journal on Numerical Analysis},
    VOLUME = {46},
      YEAR = {2008},
    NUMBER = {5},
     PAGES = {2524--2550},
      ISSN = {0036-1429,1095-7170},
   MRCLASS = {65N30 (41A25)},
  MRNUMBER = {2421046},
MRREVIEWER = {Hans-Peter\ Helfrich},
       DOI = {10.1137/07069047X},
       
}

@article {CarstensenHellwig-DiscretePoissonFreidrichs,
    AUTHOR = {Carstensen, Carsten and Hellwig, Friederike},
     TITLE = {Constants in discrete {P}oincar\'e{} and {F}riedrichs
              inequalities and discrete quasi-interpolation},
   JOURNAL = {Comput. Methods Appl. Math.},
  FJOURNAL = {Computational Methods in Applied Mathematics},
    VOLUME = {18},
      YEAR = {2018},
    NUMBER = {3},
     PAGES = {433--450},
      ISSN = {1609-4840,1609-9389},
   MRCLASS = {65N30 (65J05)},
  MRNUMBER = {3824773},
       DOI = {10.1515/cmam-2017-0044},
       
}

@article {Shedensack-DiscreteHelmholtz,
    AUTHOR = {Bringmann, Philipp and Ketteler, Jonas W. and Schedensack,
              Mira},
     TITLE = {Discrete {H}elmholtz {D}ecompositions of {P}iecewise
              {C}onstant and {P}iecewise {A}ffine {V}ector and {T}ensor
              {F}ields},
   JOURNAL = {Found. Comput. Math.},
  FJOURNAL = {Foundations of Computational Mathematics. The Journal of the
              Society for the Foundations of Computational Mathematics},
    VOLUME = {25},
      YEAR = {2025},
    NUMBER = {2},
     PAGES = {417--461},
      ISSN = {1615-3375,1615-3383},
   MRCLASS = {65D18 (65N30 74B05 74S05 76D07)},
  MRNUMBER = {4884463},
       DOI = {10.1007/s10208-024-09642-1},
       
}

@article {GBS-PWStokesNumMat,
    AUTHOR = {Gr\"a{\ss}le, Benedikt and Bohne, Nis-Erik and Sauter, Stefan},
     TITLE = {The pressure-wired {S}tokes element: a mesh-robust version of
              the {S}cott-{V}ogelius element},
   JOURNAL = {Numer. Math.},
  FJOURNAL = {Numerische Mathematik},
    VOLUME = {156},
      YEAR = {2024},
    NUMBER = {5},
     PAGES = {1781--1807},
}

@article {Vohralik-DiscretePoincare,
    AUTHOR = {Vohral\'ik, Martin},
     TITLE = {On the discrete {P}oincar\'e-{F}riedrichs inequalities for
              nonconforming approximations of the {S}obolev space {$H^1$}},
   JOURNAL = {Numer. Funct. Anal. Optim.},
  FJOURNAL = {Numerical Functional Analysis and Optimization. An
              International Journal},
    VOLUME = {26},
      YEAR = {2005},
    NUMBER = {7-8},
     PAGES = {925--952},
      ISSN = {0163-0563,1532-2467},
   MRCLASS = {65N30 (26D15 46E35)},
  MRNUMBER = {2192029},
       DOI = {10.1080/01630560500444533},
       
}

@article {CFPP-AxiomsOfAdaptivity,
    AUTHOR = {Carstensen, C. and Feischl, M. and Page, M. and Praetorius,
              D.},
     TITLE = {Axioms of adaptivity},
   JOURNAL = {Comput. Math. Appl.},
  FJOURNAL = {Computers \& Mathematics with Applications. An International
              Journal},
    VOLUME = {67},
      YEAR = {2014},
    NUMBER = {6},
     PAGES = {1195--1253},
      ISSN = {0898-1221,1873-7668},
   MRCLASS = {65N50 (65N12 65N22 65N30 65N38)},
  MRNUMBER = {3170325},
MRREVIEWER = {Tsu-Fen\ Chen},
       DOI = {10.1016/j.camwa.2013.12.003},
      
}

@article {Rabus-OptimalRatesNCFEM,
    AUTHOR = {Rabus, H.},
     TITLE = {A natural adaptive nonconforming {FEM} of quasi-optimal
              complexity},
   JOURNAL = {Comput. Methods Appl. Math.},
  FJOURNAL = {Computational Methods in Applied Mathematics},
    VOLUME = {10},
      YEAR = {2010},
    NUMBER = {3},
     PAGES = {315--325},
      ISSN = {1609-4840,1609-9389},
   MRCLASS = {65N30 (65N12 65N50)},
  MRNUMBER = {2770297},
MRREVIEWER = {Tsu-Fen\ Chen},
       DOI = {10.2478/cmam-2010-0018},
       
}

@article {BohneGrässleSauter-PressureImprovedStokes,
    AUTHOR = {Bohne, Nis-Erik and Gr\"a{\ss}le, Benedikt and Sauter, Stefan
              A.},
     TITLE = {Pressure-improved {S}cott-{V}ogelius type elements},
   JOURNAL = {Calcolo},
  FJOURNAL = {Calcolo. A Quarterly on Numerical Analysis and Theory of
              Computation},
    VOLUME = {62},
      YEAR = {2025},
    NUMBER = {1},
     PAGES = {Paper No. 8, 33},
      ISSN = {0008-0624,1126-5434},
   MRCLASS = {65N30 (65N12 76D07)},
  MRNUMBER = {4844036},
       DOI = {10.1007/s10092-024-00627-8},
       
}

@article {Proriol-orthoPolyTriangle,
    AUTHOR = {Proriol, Joseph},
     TITLE = {Sur une famille de polynomes \`a{} deux variables orthogonaux
              dans un triangle},
   JOURNAL = {C. R. Acad. Sci. Paris},
  FJOURNAL = {Comptes Rendus Hebdomadaires des S\'eances de l'Acad\'emie des
              Sciences},
    VOLUME = {245},
      YEAR = {1957},
     PAGES = {2459--2461},
      ISSN = {0001-4036},
   MRCLASS = {33.00},
  MRNUMBER = {95994},
}

@article {BMS-NCAFEMOptimal,
    AUTHOR = {Becker, Roland and Mao, Shipeng and Shi, Zhongci},
     TITLE = {A convergent nonconforming adaptive finite element method with
              quasi-optimal complexity},
   JOURNAL = {SIAM J. Numer. Anal.},
  FJOURNAL = {SIAM Journal on Numerical Analysis},
    VOLUME = {47},
      YEAR = {2010},
    NUMBER = {6},
     PAGES = {4639--4659},
      ISSN = {0036-1429,1095-7170},
   MRCLASS = {65N30 (65N12 65N15 65N50)},
  MRNUMBER = {2595052},
MRREVIEWER = {Javier\ de Frutos},
       DOI = {10.1137/070701479},
      
}

@article {BCGT-StabFreeHHO,
    AUTHOR = {Bertrand, Fleurianne and Carstensen, Carsten and Gr\"a{\ss}le,
              Benedikt and Tran, Ngoc Tien},
     TITLE = {Stabilization-free {HHO} a posteriori error control},
   JOURNAL = {Numer. Math.},
  FJOURNAL = {Numerische Mathematik},
    VOLUME = {154},
      YEAR = {2023},
    NUMBER = {3-4},
     PAGES = {369--408},
      ISSN = {0029-599X,0945-3245},
   MRCLASS = {65N12 (65N30 65Y20)},
  MRNUMBER = {4630545},
       DOI = {10.1007/s00211-023-01366-8},
       
}

@book {DPD-HHOBook,
    AUTHOR = {Di Pietro, Daniele Antonio and Droniou, J\'er\^ome},
     TITLE = {The hybrid high-order method for polytopal meshes},
    SERIES = {Modeling, Simulation and Applications},
    VOLUME = {19},
      NOTE = {Design, analysis, and applications},
 PUBLISHER = {Springer, Cham},
      YEAR = {[2020] \copyright 2020},
     PAGES = {xxxi+525},
      ISBN = {978-3-030-37202-6; 978-3-030-37203-3},
   MRCLASS = {65-02 (74-01 76-01 78Mxx)},
  MRNUMBER = {4230986},
       DOI = {10.1007/978-3-030-37203-3},
      
}

@misc{NGSolve,
author = {Schöberl, Joachim},
title = {Finite element software NETGEN/NGSolve Version 6.2},
url = {https://ngsolve.org/}}

@book {Grisward-Singularities,
    AUTHOR = {Grisvard, P.},
     TITLE = {Singularities in boundary value problems},
    SERIES = {Recherches en Math\'ematiques Appliqu\'ees [Research in
              Applied Mathematics]},
    VOLUME = {22},
 PUBLISHER = {Masson, Paris; Springer-Verlag, Berlin},
      YEAR = {1992},
     PAGES = {xiv+199},
      ISBN = {2-225-82770-2},
   MRCLASS = {35-02 (35A20 35J25 35J65)},
  MRNUMBER = {1173209},
MRREVIEWER = {V.\ S.\ Rabinovich},
}

@book {Mitchell-NVB,
    AUTHOR = {Mitchell, William F.},
     TITLE = {Unified multilevel adaptive finite element methods for
              elliptic problems},
      NOTE = {Thesis (Ph.D.)--University of Illinois at Urbana-Champaign},
 PUBLISHER = {ProQuest LLC, Ann Arbor, MI},
      YEAR = {1988},
     PAGES = {119},
   MRCLASS = {99-05},
  MRNUMBER = {2637128},
}

@article {Brenner-confComp,
    AUTHOR = {Brenner, Susanne C.},
     TITLE = {Two-level additive {S}chwarz preconditioners for nonconforming
              finite element methods},
   JOURNAL = {Math. Comp.},
  FJOURNAL = {Mathematics of Computation},
    VOLUME = {65},
      YEAR = {1996},
    NUMBER = {215},
     PAGES = {897--921},
      ISSN = {0025-5718,1088-6842},
   MRCLASS = {65N30 (65F35 65N55)},
  MRNUMBER = {1348039},
MRREVIEWER = {I.\ Norman\ Katz},
       DOI = {10.1090/S0025-5718-96-00746-6},
}
\end{document}